\documentclass[fleqn,11pt]{article}
\usepackage[margin=1in]{geometry}                
\usepackage{authblk}

\usepackage{easyReview}
\usepackage[dvipsnames]{xcolor}

\usepackage{amsmath,amsthm,amsfonts,amssymb}
\usepackage{graphicx}
\usepackage{hyperref}
\hypersetup{colorlinks,linkcolor=blue,citecolor=blue}

\newtheorem{definition}{Definition}
\newtheorem{theorem}{Theorem}
\newtheorem{lemma}[theorem]{Lemma}
\newtheorem{corollary}[theorem]{Corollary}
\newtheorem{prop}{Proposition}

\usepackage{xcolor}
\usepackage{bm,soul}
\usepackage{cancel}
\usepackage{mathtools}
\usepackage[ruled,linesnumbered]{algorithm2e}
\SetKwComment{Comment}{/* }{ */}
\usepackage{enumitem}

\newcommand\calB{\mathcal{B}}
\newcommand\calF{\mathcal{F}}
\newcommand\calM{\mathcal{M}}
\newcommand\calP{\mathcal{P}}     
\newcommand\calS{\mathcal{S}}     
\newcommand\calX{\mathcal{X}}     
\newcommand\calY{\mathcal{Y}}
\newcommand\calZ{\mathcal{Z}}
\newcommand\CycSp{\mathcal{C}}    
\newcommand\RelCyc{\mathcal{C_R}} 
\newcommand\expand[2]{\mathcal{E}_{#1}\left(#2\right)}   
\newcommand\expandd[2]{\mathcal{E}_{#1}(#2)}             
\newcommand\symdiff{\mathop{\triangle}}            
\usepackage{scalerel}
\DeclareMathOperator*{\symdiffseries}{\scalerel*{\triangle}{\sum}}   
\newcommand\pInt{ \stackrel{\tt pi}{\longleftrightarrow} }
\newcommand\slInt{ \stackrel{\tt sli}{\longleftrightarrow} }             
\newcommand\smCyc{\textnormal{(smaller cycles)}}
 
\DeclareMathOperator\Span{span}                         

\begin{document}
	
	{\title{Theory and algorithms for clusters of cycles in graphs for material networks}

		\author[1]{Perrin E. Ruth\thanks{pruth@umd.edu}}
		\author[1]{Maria Cameron\thanks{mariakc@umd.edu}}
		\affil[1]{\small{Department of Mathematics, University of Maryland, College Park, MD 20742, USA}}}
		\maketitle


\begin{abstract}
	Analysis of complex networks, particularly material networks such as the carbon skeleton of hydrocarbons generated in hydrocarbon pyrolysis in carbon-rich systems, is essential for effectively describing, modeling, and predicting their features.  An important and the most challenging part of this analysis is the extraction and effective description of cycles, when many of them coalesce into complex clusters. A deterministic minimum cycle basis (MCB) is generally non-unique and biased to the vertex enumeration. The union of all MCBs, called the set of relevant cycles, is unique, but may grow exponentially with the graph size. To resolve these issues, we propose a method to sample an MCB uniformly at random. The output MCB is statistically well-defined, and its size is proportional to the number of edges. We review and advance the theory of graph cycles from previous works of Vismara, Gleiss et al., and Kolodzik et al. In particular, we utilize the polyhedron-interchangeability ({\tt pi}) and short loop-interchangeability ({\tt sli}) classes to partition the relevant cycles. We introduce a postprocessing step forcing pairwise intersections of relevant cycles to consist of a single path. This permits the definition of a dual graph whose nodes are cycles and edges connect pairs of intersecting cycles. The {\tt pi} classes identify building blocks for crystalline structures. The {\tt sli} classes group together sets of large redundant cycles. We present the application to an amorphous hydrocarbon network, where we (i) theorize how the number of relevant cycles may explode with system size and (ii) observe small polyhedral structures related to diamond.
\end{abstract}

\tableofcontents


\section{Introduction}



\begin{figure*}
	\centering
	\includegraphics{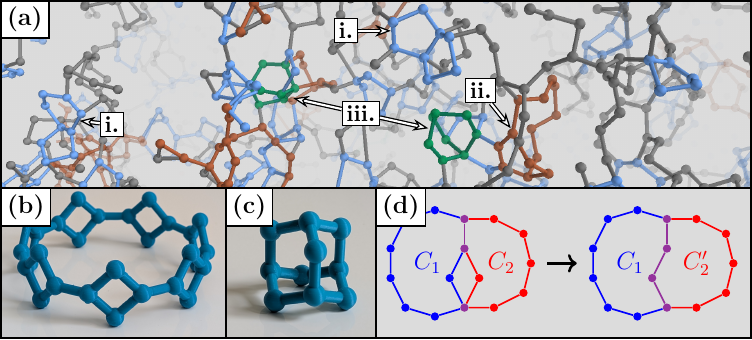}
	\caption{\textbf{(a)} Snapshot of the giant biconnected component in a simulation initialized to adamantane ($\rm C_{10}H_{16}$) at 4000K and 40.5GPa. 
	We highlight (i) \emph{essential}~\cite{gleissInterchangeability2000} pentagonal and hexagonal rings belonging to every MCB in blue, (ii) a {\tt sli} class composed of two 11-membered rings separated by a hexagonal ring in orange, and (iii) two {\tt pi} classes with the same shape as barallene ($\rm C_8H_8$) in green.
	\textbf{(b)} A graph made of diamond-shaped cycles chained into a loop like a bracelet. An MCB of this graph includes all of the small diamond-shaped loops and one large loop. \textbf{(c)} Carbon skeleton of adamantane, the smallest diamondoid, i.e., diamond is formed by repeatedly tiling adamantane. An MCB of adamantane includes any 3 of the 4 hexagonal faces. \textbf{(d)} Depiction of the process of modifying a cycle $C_2\to C_2'$ such that the cycles $C_1$ and $C_2$ overlap over a single path.}
	\label{fig:intro_graphs}
\end{figure*}

\subsection{Motivation}

Amorphous material networks appear in numerous contexts such as polymers~\cite{torres-knoopModeling2018}, energetic carbon~\cite{dernovMapping2025,deringerMachine2017}, soot~\cite{villa-alemanRaman2022}, and silicate melts~\cite{clarkShockramp2023}. 
Understanding the network topology of these systems is critical because it influences their material properties.
The local atomistic structures within these networks can be leveraged to make predictions about their thermodynamic limits and to reduce combinatorial complexity.
Here, we review, discuss, and extend theory and algorithms for the seemingly simple task of sampling rings from material networks. 
The assumptions under which these algorithms are efficient are consistent with the carbon skeletons of large hydrocarbons arising in hydrocarbon pyrolysis, 
a complex chemical reaction system composed of carbon and hydrogen at extreme temperatures and pressures.

Conventional approaches for studying hydrocarbon pyrolysis include molecular dynamics {(MD)}~\cite{chengThermodynamics2023,chenTransferable2019,dufour-decieuxTemperature2022,dufour-decieuxPredicting2023} simulations and kinetic Monte Carlo~\cite{dufour-decieuxAtomicLevel2021,dufour-decieuxTemperature2022}. 
MD is flexible and can be used to simulate various systems and thermodynamic conditions. 
However, MD simulations are expensive. They may take days or even weeks to achieve nanosecond timescales. Furthermore, they need to be performed anew for every set of initial conditions. 
Kinetic Monte Carlo~\cite{gillespieGeneral1976,highamModeling2008} simulates a chemical system using a set of known chemical reactions. 
These reactions and their reaction rates can be learned from a set of molecular dynamics simulations. This approach is limited by the combinatorial growth in the number of reactions, e.g., a simulation with $\sim10^3$ atoms results in $\sim10^4$ potential reactions~\cite{yangLearning2017,yangDataDriven2019,chenTransferable2019}.
The combinatorial growth of hydrocarbons has been studied as early as 1875, where Cayley used generating functions to estimate the number of alkanes of a given length~\cite{cayley1875ueber,rains1999cayley}.
Due to this growth, only small molecules counts can be predicted accurately by this approach.
To address this limitation, Refs.~\cite{dufour-decieuxAtomicLevel2021,dufour-decieuxTemperature2022} propose kinetic models with only local reactions accounting for chemical changes near the reaction site. 
This approach reduces the number of reactions and can be used to make predictions about the size of the largest molecule. 
However, like MD simulations, the local kinetic Monte Carlo requires a new run for each set of initial conditions. 

Our research group utilized random graphs as equilibrium models for hydrocarbon pyrolysis~\cite{dufour-decieuxPredicting2023,ruthCyclic2025}. 
For hydrogen-rich systems, the configuration model~\cite{newmanRandom2001,molloySize1998} -- a classical random graph model -- was used to estimate the molecule size distribution~\cite{dufour-decieuxPredicting2023}. 
This model is particularly effective for small alkanes ($\le20$ carbon atoms) or fully saturated acyclic hydrocarbons.
However, as the hydrogen-to-carbon ratio decreases, the configuration model significantly exaggerates the size of the largest molecule, mostly because it does not promote the formation of small carbon rings abundant in hydrocarbon molecules.
Using Refs.~\cite{karrerRandom2010,newmanAssortative2002} as building blocks, we introduced a random graph model with disjoint loops that accurately predicted the size of the largest molecule~\cite{ruthCyclic2025}. This disjoint loop model also predicted the phase transition between the states with and without a giant molecule, i.e., a molecule that contains a macroscopic fraction of the total number of atoms. 
In carbon-rich systems, the rings cluster together, and the disjoint loop assumption fails. See Figure~\ref{fig:intro_graphs}~(a). In this setting, proper statistics for carbon rings are required.


\subsection{An overview}

At a macroscopic scale, a ring cluster can be identified as a biconnected component or a collection of nodes where each pair of nodes is connected by two or more disjoint paths~\cite{diestelGraph2017,bollobasModern1998}. A method for computing biconnected components is given in the 1971 paper by Hopcroft and Tarjan~\cite{hopcroftAlgorithm1973}. At a microscopic scale, we want to obtain the building blocks of these clusters, i.e., their loops. In many cases, the na\"ive approach of computing all cycles is computationally intractable. The total number of cycles grows exponentially with graph size, many of which are redundant. Instead, most methods for sampling loops utilize the cycle space or the vector space of cycles.

Classically, the loops of a graph are sampled from a \emph{minimum cycle basis} (MCB) or a basis of the cycle space that minimizes the sum of cycle lengths~\cite{kavithaCycle2009}. In chemical systems, a minimum cycle basis is also called a smallest set of smallest rings, as introduced by Plotkin~\cite{plotkinMathematical1971}. {MCBs} are not without their flaws. Most notably, they are not unique, so the loops sampled via an MCB depend on node labels and algorithmic implementation. This issue has led to several competing definitions for sampling cycles in chemistry literature~\cite{bergerCounterexamples2004}.

Here, we focus on the approach introduced by Vismara, 1997~\cite{vismaraUnion1997}, where the set of \emph{relevant cycles} defined as the union of all MCBs is considered. 
Equivalently, the relevant cycles are the unique cycles that cannot be decomposed into smaller cycles. 
In the worst case, the number of relevant cycles grows exponentially with respect to the number of nodes, but in practice the number of relevant cycles is typically much smaller~\cite{mayEfficient2014}. 
Nevertheless, this is an undesirable property of the relevant cycles. 
To guarantee properties such as the number of relevant cycles and their lengths can be computed in polynomial time, Vismara partitions the relevant cycles into families of related cycles~\cite{vismaraUnion1997}.

Another approach would be to investigate Vismara's cycle families directly. However,  these families are not unique, thus returning us to our original issue. 
To resolve this issue, Gleiss~{et al.}, 2000,~\cite{gleissInterchangeability2000} introduce the \emph{interchangeability relation}, an equivalence relation that partitions the relevant cycles into unique interchangeability classes. 
This relation helps to characterize the lack of uniqueness in MCBs.
Two relevant cycles are interchangeable if one cycle can be swapped for the other in a cycle basis, and if the two cycles only differ by equal or shorter length cycles. 
Often, this relation groups together large cycles that differ by smaller cycles. 
For example, the cycles around the perimeter of the bracelet graph in Figure~\ref{fig:intro_graphs} (b) form a single interchangeability class. 
In other cases, this relation groups together cycles that belong to three-dimensional structures. 
For instance, the four hexagonal faces of adamantane in Figure~\ref{fig:intro_graphs} (c) form a single interchangeability class.

Kolodzik {et al.}, 2012~\cite{kolodzikUnique2012}, argue these interchangeability classes are too coarse and introduce the \emph{unique ring family (URF)-pair-relation}. This relation is more restrictive, where two relevant cycles are URF-pair-related if they share edges and if they differ by \emph{strictly shorter} cycles. This separates the cases above. The large cycles in the bracelet graph are URF-pair-related because they differ by smaller diamond cycles. The hexagonal faces in adamantane are not URF-pair-related because they differ by equal length hexagonal faces. The URF-pair-relation is not an equivalence relation. Instead, the unique ring families are given by the transitive closure of the URF-pair-relation.

\subsection{The goal and a summary of main results}

{The goal of this work} is to review and extend theory and algorithms that describe and extract complex amorphous ring clusters in carbon-rich simulation data of hydrocarbon pyrolysis. To achieve this, we adapt and advance concepts and methods from Refs.~\cite{gleissInterchangeability2000,kolodzikUnique2012}. 
Our theoretical developments are the following:
\begin{itemize}
	\item 
	At a course level, we introduce the \emph{polyhedron-interchangeability} ({\tt pi}) relation, a modification of the interchangeability relation from Ref.~\cite{gleissInterchangeability2000}.
	Unlike interchangeability, {\tt pi} relates pairs of cycles that can be swapped between MCBs. We prove {\tt pi} is an equivalence relation (Theorem~\ref{Thm:DirInt_Eq_Rel}) and the resulting equivalence classes can be computed using a given MCB (Theorem~\ref{Thm:DirInt_Comp}). The theory of {\tt pi} also characterizes how we can modify MCBs by swapping cycles.
	\item At a granular level, we introduce the \emph{short loop-interchangeability} ({\tt sli}) relation, adapted from the URF-pair-relation from Ref.~\cite{kolodzikUnique2012}. {\tt sli} removes the requirement that cycles share edges from the URF-pair-relation. As a result, {\tt sli} is an equivalence relation. 
	{\tt sli} was first introduced in the dissertation of Gleiss~\cite{gleissShort2001} as stronger interchangeability, where it was shown to be an equivalence relation. 
	We prove the {\tt sli} classes can be computed using a given MCB. 
	We suggest to use the {\tt pi} and {\tt sli} relations together to analyze ring clusters.
\end{itemize}

Using this theory we develop the following methods for sampling cycles in a graph.

\begin{itemize}
	\item \emph{Random sampling of minimum cycle bases.} We introduce a method for sampling MCBs uniformly at random. 
	Of course, a random MCB is not deterministic, but it does have a concrete definition as a probabilistic object. 
	Thus, the practitioner can confidently sample loops from a random MCB knowing it is statistically well-defined.
	\item \emph{Pairwise intersection of cycles.} 
	A graph is defined by nodes and by edges denoting connections between nodes. 
	Similarly, we are interested in loops and their intersections that resulting in the formation of loop cluster. 
	In this work, we fully characterize the way pair of cycles in MCBs can intersect. 
	Let the intersection of two cycles contain more than one path as in Fig.~\ref{fig:intro_graphs}~(d). 
	We introduce a method for modifying one of these cycles to make the intersection consist of a single path. 
	This property is desirable as individual paths are conceptually simple and they are easy to represent computationally as node lists.
	
	Using this, we define a dual graph, which can be interpreted as a graph of cycles. The notion of dual graph has a long history in planar graphs~\cite{diestelGraph2017}, and graphs of cycles have been considered for modeling chemical systems as well~\cite{torres-knoopModeling2018,nouleho2019improving}. Our contribution is defining the structure of these dual graphs for MCBs.
\end{itemize}

We apply this theory to the hydrocarbon pyrolysis system. In most cases, we find that lack of uniqueness of MCBs is largely caused by rings at varying lengths as in Figure~\ref{fig:intro_graphs}~(b). We find ring statistics vary greatly between the relevant cycles and an MCB. In particular, the relevant cycles introduce redundant large rings. 
Furthermore, our analysis of the relevant cycles naturally uncovers three-dimensional structures. 
These structures take the form of polyhedra with non-flat faces as in Figure~\ref{fig:intro_graphs} (c). 
These polyhedra are useful for describing crystalline materials. 
Proper sampling of these polyhedra will require further theoretical advancements.



This manuscript is organized as follows. In Section~\ref{sec:background}, we provide the necessary background for the cycle space including the relevant cycles from Ref.~\cite{vismaraUnion1997} and the partitions of the relevant cycles from Refs.~\cite{gleissInterchangeability2000,kolodzikUnique2012}. In Section~\ref{sec:Theory}, we introduce the {\tt pi} and {\tt sli} relations used to partition the relevant cycles. In Section~\ref{sec:algorithm}, we provide algorithms for computing the {\tt pi} and {\tt sli} classes. Our Python codes are available on Github~\cite{RuthGithub}. These codes take a simple graph from the NetworkX package~\cite{hagbergExploring2008} as input.
In Section~\ref{sec:sampling}, we provide an algorithm for sampling MCBs uniformly at random, and we characterize the intersection of cycles in MCBs. In Section~\ref{sec:results}, we show how this approach can be used to characterize ring clusters from hydrocarbon pyrolysis simulations. 
These molecular dynamics simulations are obtained via LAMMPS~\cite{plimptonFast1995,thompsonLAMMPS2022} using the ReaxFF~\cite{vanduinReaxFF2001,chenowethReaxFF2008,srinivasanDevelopment2015,ashrafExtension2017} force field permitting chemical reactions.

\section{Background}
\label{sec:background}

\subsection{Cycle space and fundamental cycle bases}

\label{sub:cyc_space}

Cycle spaces have a rich history in electronic circuit theory, chemistry, and other contexts. We start with a brief discussion of classical results in this field. We also introduce notation, which is merged from Refs.~\cite{gleissInterchangeability2000,kavithaCycle2009,vismaraUnion1997}. For a more detailed history and discussion of these topics we recommend the books on graph theory~\cite{bollobasModern1998,diestelGraph2017}, review of cycle bases~\cite{kavithaCycle2009}, review of cycles in chemical systems~\cite{bergerCounterexamples2004}, and references therein. 

\begin{figure}[t]
	\centering
	\includegraphics{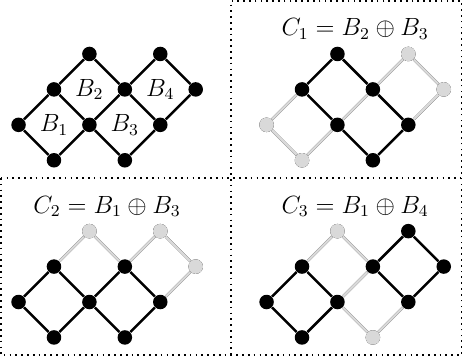}
	\caption{A graph with four face cycles $B_1$, $B_2$, $B_3$, and $B_4$ forming a basis. Three additional cycles $C_1$, $C_2$, and $C_3$ are shown. $C_1$ is a single loop while $C_2$ and $C_3$ are not.}
	\label{fig:cyc_space} 
\end{figure}

Let $G=(V,E)$ be a graph with $n\coloneq|V|$ nodes and $m\coloneq|E|$ edges. Formally, a \textit{cycle} is an edge-induced subgraph of $G$ where all nodes have even degree. More intuitively, a cycle is a collection of loops that do not share edges. See Fig.~\ref{fig:cyc_space}. We note, nearly all cycles we consider are simple loops, including all cycles introduced later in this section in fundamental and minimum cycle bases.

The utility of this definition is that the set of cycles -- denoted by $\CycSp$ -- forms a vector space called the \textit{cycle space}. 
The operation we apply to cycles is the symmetric difference $C_1\oplus C_2=(C_1\setminus C_2)\cup(C_2\setminus C_1)$. 
The scalars of this vector space, 
0 and 1, comprise the field $GF(2)$ with addition and multiplication modulo 2.
Note, a cycle is necessarily its own additive inverse, $C\oplus C = \emptyset$.  

Since a cycle $C\in\CycSp$ is a set of edges, it may be represented as an \textit{incidence vector}: a vector in $\{0,1\}^m$ where element $i$ indicates whether edge $i$ is contained in $C$. 
{Not all incidence vectors correspond to cycles.} 
In vector form, the symmetric difference of cycles can be computed via an element-wise XOR operation (or by vector addition modulo 2). We distinguish that the cycle space is only a subspace of the set of incidence vectors, and ultimately we want a lower-dimensional representation of cycles.

\begin{figure*}[t]
	\centering
	\includegraphics[width=6.5in]{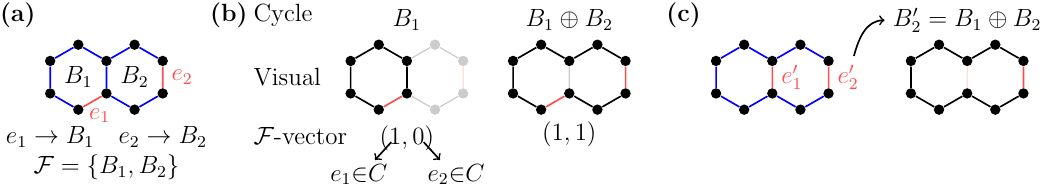}
	\caption{\textbf{(a)} A graph $G$ with edges colored according to whether they are inside (blue) or outside (red) the spanning tree $T$. 
		The face cycles $B_1$ and $B_2$ form a fundamental cycle basis $\calF$, where they are composed of an edge in $E(G){\setminus} E(T)$ joined by a path in $T$.
		\textbf{(b)} A depiction of the cycles $B_1$ and $B_1\,{\oplus}\, B_2$ alongside their binary representation as fundamental cycle basis $\calF$ vectors. \textbf{(c)} Another spanning tree $T'$ resulting in a different fundamental cycle basis that contains $B_2'\,{=}\,B_1\,{\oplus}\, B_2$.}
	\label{fig:FCB}
\end{figure*}

This notion of a vector space leads to a natural definition of a basis. Specifically, a \textit{cycle basis} is a set of linearly independent cycles $\calB=\{B_1,\hdots,B_\nu\}$ that span the cycle space, where $\nu$ is the dimension of the cycle space. Then, we may expand a cycle $C$ as follows,
\begin{equation}
	C = b_1 B_1\oplus\hdots\oplus b_\nu B_\nu,\quad b_1,\hdots,b_\nu\in\{0,1\}.
	\label{eq:cyc_exp_format_full}
\end{equation}
Note, lower case letters are used for scalars, upper case for cycles, and math italics (e.g. $\calB$) for sets of cycles. Computationally, it is easier to store cycles as binary vectors $(b_1,\hdots,b_\nu)$ where the addition operator $\oplus$ is still interpreted as the element-wise XOR operation.

Now that we know how to use a cycle basis, the next question is ``how do we construct such a basis?''
A classical approach is to utilize a spanning tree.


\begin{definition}
	{Let $T$ be a spanning tree of $G$, or a forest if $G$ is disconnected, and $\{e_i\}$ the edges in $G$ outside of $T$, i.e., $e_i\in E(G)\setminus E(T)$. The \emph{fundamental cycle basis} induced by $T$ is the set of cycles $\calF=\{B_i\}$, where $B_i$ is the cycle composed of the edge $e_i$ and the path that connects its endpoints in $T$.}
	\label{def:FCB}
\end{definition}
\noindent An example of a fundamental cycle basis is given in Fig.~\ref{fig:FCB} (a). To get the expansion of a cycle $C$ with respect to $\calF$ we focus on the edges $\{e_i\}$. Note, there exists a unique basis cycle $B_i\in\calF$ for each edge $e_i$. 
So, $B_i$ is in the expansion of $C$ with respect to $\calF$ if and only if $\{e_i\}$ is contained in $C$. Furthermore, the coefficient $b_i$ as defined in Eq.~\eqref{eq:cyc_exp_format_full} is the indicator that $e_i\in C$ and $(b_1,\hdots,b_\nu)$ is the incidence vector of $C$ restricted to edges in $E(G)\setminus E(T)$.
Following this procedure, we find the rank of the cycle space is the number of edges in $E(G)\setminus E(T)$:
\begin{equation}
	\nu = m-n+(\text{\# Connected Components}),
\end{equation}
where a tree with $n$ nodes has $n-1$ edges, and $T$ has a spanning tree of each connected component of $G$. However, we have not yet verified that $\calF$ is a basis.


It is easy to see that $\calF$ is linearly independent, since each cycle $B_i$ has an edge $e_i$ that the other cycles in $\calF$ do not. To see that $\calF$ spans $\CycSp$, we take an arbitrary cycle $C\in\CycSp$. Then, for each basis cycle $B_i\in\calF$ we add it to $C$ if the corresponding edge $e_i\in E(G)\setminus E(T)$ is in $C$:
\begin{equation}
	C' = C\oplus \bm{1}_{[e_1\in C]}B_1\oplus\hdots\oplus\bm{1}_{[e_\nu \in C]}B_\nu
\end{equation}
where $\bm{1}$ is the indicator function. By cancellation from the $\oplus$-operator, $C'$ contains no edges in $E(G)\setminus E(T)$. In other terms, $C'$ is contained in the spanning forest $T$. Trees lack loops, so $C'$ is necessarily the empty cycle $\emptyset$, and $C$ is contained in the span of $\calF$.

\subsection{Minimum cycle bases and relevant cycles}

\label{sub:MCB+Vis_Fam}
Fundamental cycle bases are useful for analyzing the cycle space as a vector space. However, these bases should not be used when sampling loop statistics. Specifically, fundamental cycle bases are not unique and depend on the spanning tree used to construct them. Furthermore, a given fundamental cycle basis may contain large cycles that are not intuitively relevant.

To elaborate on this issue consider the graph depicted in Figure~\ref{fig:FCB} with two hexagons that share an edge. The spanning tree in Figure~\ref{fig:FCB} (a) results in a fundamental cycle basis composed of these hexagons, as desired. However, the spanning tree in Figure~\ref{fig:FCB} (c) produces a fundamental cycle basis containing the larger cycle with 10 nodes surrounding the perimeter of this graph. We argue, that it is more natural to describe this larger cycle as the sum of smaller cycles.

A common approach taken to address this concern is to take a cycle basis that contains only short cycles:
\begin{definition}
	A \emph{mimimum cycle basis (MCB)} $\calM$ is a cycle basis that minimizes the sum of cycle lengths
	\begin{equation*}
		{\sf Cost}(\calM) = \sum_{B\in\calM} |B|
	\end{equation*}
	where $|B|$ is the number of edges in cycle $B$.
\end{definition}
\noindent As one might hope, the two hexagons in Figure~\ref{fig:FCB} form a unique MCB.

This definition has an interesting history in many fields. Ref.~\cite{kavithaCycle2009} offers a survey of algorithms to compute MCBs. 
Ref.~\cite{bergerCounterexamples2004} reviews methods for counting cycles in chemical systems. 
We remark that, in chemical systems, a cycle is typically called a ring, and an MCB is referred to as the ``Smallest Set of Smallest Rings''. In Section~\ref{sub:MCB_comp}, we review an algorithm from Ref.~\cite{kavithaFaster2004} for computing MCBs. This algorithm runs in $O(mn^2)$ time. In Section~\ref{sub:MCB_COB}, we show how one can expand an arbitrary cycle in terms of an MCB output by this algorithm. 

\begin{figure}[t]
	\centering
	\includegraphics{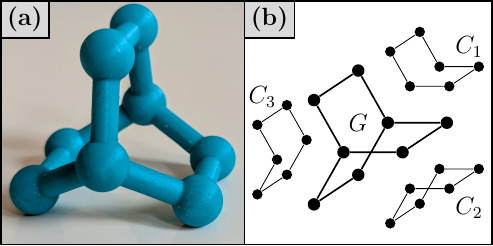}
	\caption{
		A graph $G$ that does not have a unique MCB. Ignoring double bonds, $G$ is the carbon skeleton of barallene with formula $\rm C_8 H_8$. \textbf{(a)} 3D printed visualization of $G$. \textbf{(b)} Two-dimensional representation of $G$. The three hexagonal relevant cycles $C_1$, $C_2$, and $C_3$ are shown separately. Any pair of these cycles forms an MCB, where the third cycle is the sum of the other two, e.g., $\calM_1=\{C_1,C_2\}$ is an MCB and $C_3=C_1\oplus C_2$.
	}
	\label{fig:MCB_Unique}
\end{figure}

Unfortunately, MCBs are not always unique -- see Figure~\ref{fig:MCB_Unique}. Therefore, if we evaluate summary statistics from cycles in a given MCB, then we obtain inconsistent results depending on which MCB is used. For instance, consider the number of times node 2 in Figure~\ref{fig:MCB_Unique} participates in a cycle. If we compute this statistic using the MCB $\calM_1=\{C_1,C_2\}$ then node 2 participates in two cycles $C_1$ and $C_2$, but if we take $\calM_2=\{C_2,C_3\}$ as an MCB then node 2 participates in only one cycle $C_2$. As mentioned in Ref.~\cite{kolodzikUnique2012}, this is an issue with pattern matching of chemical rings in SMARTS, {a programming library for identifying molecular patterns and properties}~\cite{SMARTS}. Additionally, the OEChem {library for chemistry and cheminformatics} avoids using MCBs in chemical ring perception citing lack of uniqueness, and states ``Smallest Set of Smallest Rings (SSSR) considered Harmful''~\cite{SSSR_Harmful}.

Because of this issue, several works in computational chemistry consider definitions that extend MCBs to larger sets of cycles. We do not review all of these definitions here, and refer the interested reader to Ref.~\cite{bergerCounterexamples2004}. Instead, we review the notion of relevant cycles from Ref.~\cite{vismaraUnion1997}, which serves as a foundation for the definitions discussed in this manuscript. The relevant cycles have two equivalent definitions:
\begin{definition}
	The set of relevant cycles $\RelCyc$ is the union of all MCBs.\label{def:rel}
\end{definition}
\begin{definition}
	A cycle $C\in\CycSp$ is relevant ($C\in\RelCyc$) if and only if it is not generated by smaller cycles, i.e., there does not exist a set of smaller cycles $\calX$, $|C'|<|C|$ for $C'\in\calX$, such that
	\begin{equation}
		C = \bigoplus_{C'\in\calX} C'
	\end{equation}
	where $\bigoplus$ is the series representation of the $\oplus$-operator.
	\label{def:rel_2}
\end{definition}
We include a proof of the equivalence of these definitions after introducing additional theory about the cycle space -- see Corollary~\ref{cor:rel_def_eq} in Sec.~\ref{sub:basis_ex}. For the graph depicted in Figure~\ref{fig:MCB_Unique}, the relevant cycles are the three hexagons $\RelCyc=\{C_1,C_2,C_3\}$. This set of relevant cycles has several useful properties~\cite{bergerCounterexamples2004}. 
\begin{itemize}
	\item \textit{Uniqueness}: The set $\RelCyc$ is unique, so statistics on the relevant cycles are well-defined.
	\item \textit{Completeness}: The relevant cycles can be used to generate the cycle space since $\RelCyc$ contains a cycle basis.
	\item \textit{Computability}: There exists a polynomial time algorithm for computing the set of relevant cycles.~\cite{vismaraUnion1997}.
	\item \textit{Chemical relevance}: 
	The set of relevant cycles contains all building blocks of chemical structures present in molecules. 
\end{itemize}

One downside is that the number of relevant cycles may grow exponentially for certain types of graphs.
For example, in the bracelet-like graph composed of $k$ diamonds chained together into a loop shown in Figure~\ref{fig:intro_graphs} (b). 
All MCBs in this graph are composed of the $k$ diamonds and one large loop of length $3k$. Each large loop either takes the top or bottom path of a given diamond, so the number of large relevant cycles grows exponentially as $2^k$, while the number of nodes grow linearly as $4k$.~\cite{vismaraUnion1997}
	
It is remarkable that the algorithm from Ref.~\cite{vismaraUnion1997} for computing the relevant cycles of a graph runs in polynomial time.
In broad steps, the relevant cycles are computed by (1) generating a candidate set of cycles (a superset of $\RelCyc$) and (2) filtering out the cycles that are irrelevant. The key idea that makes this procedure feasible in polynomial time is that the cycles are grouped into families of equal length. 


\begin{figure}[t]
	\centering
	\includegraphics{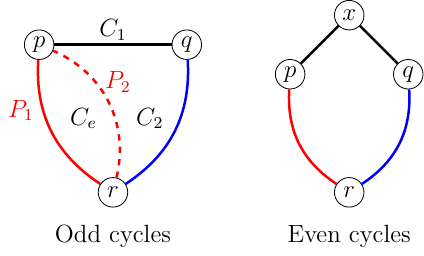}
	\caption{
		Depiction of the descriptors used to organize the relevant cycles into V-families as in Ref.~\cite{vismaraUnion1997}. 
		For an odd length cycle (left), we fix its largest index node $r$ and the edge $(p,q)$ at the opposite end of the cycle. For an even length cycle (right), we fix its largest index node $r$ and the triple $(p,x,q)$ at the opposite end of the cycle. 
		In the left, we also draw the exchange of the path $P_1$ for $P_2$ by adding the cycle $C_e=P_1\oplus P_2$ to $C_1$ to obtain $C_2=C_1\oplus C_e$.
	}
	\label{fig:Vis_Fams}
\end{figure}

\subsubsection{Vismara's families of cycles}


Here we review the construction of cycle families from Ref.~\cite{vismaraUnion1997}. 
The key idea to make the algorithm for computing the relevant cycles polynomial is to group the relevant cycles into a polynomial number of families.
%
At the same time, it has a theoretical drawback that it is attached to a given enumeration of the vertices.

Consider a simple cycle $C$. We select antipodal parts of $C$ to determine the family it belongs to.
First, we fix the node $r$ with largest index in $C$. The largest index node is selected because it is unique up to vertex labels. Second, if $C$ is an odd length cycle, then we fix the edge $(p,q)$ at the opposite end of the cycle. If $C$ has even length, then we fix the triple $(p,x,q)$ at the opposite end of the cycle. In either case, the paths from $r$ to $p$ and $q$ in $C$ are of equal length. See Figure~\ref{fig:Vis_Fams}. We call these families V-families. 

\begin{definition}[Vismara cycle family]\phantom{.}
	
	\begin{itemize}
		\item The odd V-family with descriptors $r,p,q$ is the set of cycles with largest index node $r$ consisting of shortest paths from $r$ to $p$ and $q$ joined by the edge $(p,q)$.
		
		\item The even V-family with descriptors $r,p,x,q$ is the set of cycles with largest index node $r$ consisting of shortest paths from $r$ to $p$ and $q$ joined by the triple $(p,x,q)$.
	\end{itemize}
	

	
	\label{def:vis_fam}
\end{definition}

\noindent All cycles in a V-family are of equal length. We only consider shortest paths from $r$ to $p$ and $q$ because it is a necessary condition for a cycle to be relevant -- see Lemma~\ref{lem:shortest_paths}.

\begin{figure}[t]
	\centering
	\includegraphics{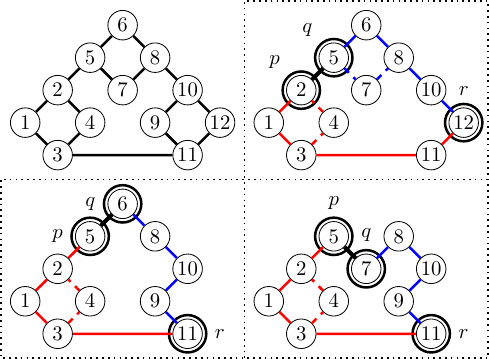}
	\caption{A labeled graph composed of three diamonds chained into a loop. The relevant cycles of this graph are the three diamonds and eight large loops of length nine. For the given node labels, the large cycles form three V-families depicted separately with their descriptors.}
	\label{fig:three_diamonds}
\end{figure}

As an example, consider the graph depicted in Figure~\ref{fig:three_diamonds}. 
This graph has eight large relevant cycles of odd length nine, distinguished by the path they take through each diamond.
Four of these cycles take the path $(10,12,11)$ in the 
diamond $(9,11,12,10)$
making their largest index node $r=12$. 
All four cycles share the same opposing edge $(2,5)$ and belong to the same family. The remaining large relevant cycles take $(10,9,11)$ instead of $(10,12,11)$ in the diamond $(9,11,12,10)$, so they have largest index node $r=11$. The opposing edge for these remaining cycles is $(5,6)$ if they take the path $(5,6,8)$ in the diamond $(5,6,8,7)$ or $(5,7)$ if they take the path $(5,7,8)$. This results in two additional V-families, each of which has two cycles, one for each path in the diamond $(1,2,4,3)$.

Ultimately, splitting cycles into V-families as in Definition~\ref{def:vis_fam} is effective because shortest paths can be exchanged via cycle addition. For instance, let us consider two cycles $C_1$ and $C_2$ in a V-family that differ by taking disjoint paths $P_1$ and $P_2$ from $r$ to $p$ respectively. See, e.g., Figure~\ref{fig:Vis_Fams}, left. Additionally, consider the cycle $C_e$ obtained by merging these paths, given by
\begin{equation}
	C_e \coloneq P_1\oplus P_2.\label{eq:Ce_Path_Merge}
\end{equation}
Then, we exchange $C_1$ for $C_2$ as follows
\begin{equation}
	C_1 = C_2 \oplus C_e\label{eq:Ce_Ex_Vis}
\end{equation}
where the path $P_2$ cancels between $C_2$ and $C_e$ and is replaced by $P_1$. 

In general, one cycle $C_1$ in a family can be exchanged for another cycle $C_2$ in the same family by adding at most two cycles. Specifically, one cycle is used to swap paths from $r$ to $p$, and the other is used to swap paths from $r$ to $q$. Note, these exchange cycles are shorter than the equal length cycles $C_1,C_2$, e.g., for $C_e$ given by Eq.~\eqref{eq:Ce_Path_Merge}
\begin{equation}
	|C_e|\le |P_1|+|P_2|=2|P_1|<|C_1|.\label{eq:Ce_Ex_Len}
\end{equation}
This exchange procedure can be used to show two useful properties for computing cycle families.

First, this can be used to show that the paths from $r$ to $p$ and $q$ in a relevant cycle are shortest paths, because otherwise they can be exchanged for shortest paths. The proof of this claim relies on the fact that the paths from $r$ to $p$ and $q$ are less than half the length of the cycle.

\begin{lemma}[Ref.~\cite{vismaraUnion1997}, Lemma 2] 
	Let $C$ be a relevant cycle and $P$ a path in $C$ that is at most half the length of $C$, $|P|\le|C|/2$. Then, $P$ is a shortest path.
	\label{lem:shortest_paths}
\end{lemma}

\begin{proof}
	Suppose towards a contradiction that $P$ is not a shortest path. Let $P'$ be a shortest path with the same endpoints, such that $|P'|<|P|$. Let $C_e$ be the cycle used to exchange paths $P$ and $P'$ given by
	\begin{equation}
		C_e = P\oplus P'.
	\end{equation}
	Let $C'$ be given by
	\begin{equation}
		C' \coloneq C\oplus C_e = (C\setminus P)\oplus P'.
	\end{equation}
	where the second equality uses associativity and that the path $P$ belongs to $C$. Then, our initial cycle is given by the sum
	\begin{equation}
		C = C'\oplus C_e.
	\end{equation}
	Both $C'$ and $C_e$ are strictly smaller than $C$
	\begin{align}
		|C_e| &\le |P|+|P'|< 2|P|\le |C|,\\
		|C'| &\le |C| - |P| + |P'| < |C|.
	\end{align}
	Therefore, $C$ is not relevant by Definition~\ref{def:rel_2}, because it is the sum of shorter cycles. This contradicts our initial assumptions, so $P$ is a shortest path.
\end{proof}


Shortest paths from a given root node $r$ to the other vertices can be obtained by breadth-first search. 
Neglecting some detail -- how to avoid nodes with index greater than $r$ -- the algorithm for generating candidate cycle families is to perform breadth-first search on all vertices $r\in V$ and identify node pairs $p$ and $q$ that form structures as in Fig.~\ref{fig:Vis_Fams}. 
We discuss this algorithm and our modifications to it in Sec.~\ref{sub:Vis_Fam}.


Another benefit of this exchange procedure is that we only need to test if one cycle in a V-family is relevant. In Ref.~\cite{vismaraUnion1997}, a single cycle \textit{prototype} is taken arbitrarily from a V-family. If this prototype is relevant, then all cycles in the V-family are relevant. 
This is computationally feasible to test because there are a polynomial number of cycle prototypes.
In the following lemma, $C_1$ and $C_2$ are cycles in a family, and $\calX$ is a set of two cycles used to exchange the paths from $r$ to $p$ and $q$. The cycles $\calX$ used to exchange paths are shorter than $C_1$ and $C_2$. See Eq.~\eqref{eq:Ce_Ex_Len}.

\begin{lemma}
	Let $C_1,C_2\in\CycSp$ be equal length cycles that differ by smaller cycles, i.e., there exists a set of cycles $\calX$, $|C'|<|C_1|$ for $C'\in\calX$, such that
	\begin{equation}
		C_1 = C_2\oplus\bigoplus_{C'\in\calX} C'.\label{eq:C1_C2_sm_c}
	\end{equation}
	Then, $C_1$ is relevant if and only if $C_2$ is relevant. 
\end{lemma}

\begin{proof}
	Let $C_2$ not relevant. We want to show that $C_1$ is not relevant. By Definition~\ref{def:rel_2}, $C_2$ is the sum of smaller cycles $\calY$, $|C'|<|C_2|$ for $C'\in\calY$, such that
	\begin{equation}
		C_2 = \bigoplus_{C'\in\calY}C'.
	\end{equation}
	Then, by Eq.~\eqref{eq:C1_C2_sm_c}, $C_1$ is given by the pair of sums
	\begin{equation}
		C_1 = \bigoplus_{C'\in\calY} C' \oplus \bigoplus_{C''\in\calX}C''.
	\end{equation}
	Cycles in the intersection $C'\in\calX\cap\calY$ cancel in the above double sum, since $C'\oplus C'$. Thus, $C_1$ is given by the single sum
	\begin{equation}
		C_1 = \bigoplus_{C'\in\calX\symdiff\calY} C'
	\end{equation}
	where $\symdiff$ is the symmetric difference $\calX\symdiff\calY = (\calX\setminus\calY)\cup(\calY\setminus\calX)$ for more general sets. Finally, $C_1$ is not relevant by Definition~\ref{def:rel_2}, because it is the sum of smaller cycles $\calX\symdiff\calY\subseteq\calX\cup\calY$. 
	
	To see the reverse statement, we note that the relationship given by Eq.~\eqref{eq:C1_C2_sm_c} is symmetric. In particular, we can add cycles $C'\in\calY$ to both sides of Eq.~\eqref{eq:C1_C2_sm_c} to move them to the left-hand side
	\begin{equation}
		C_1 \oplus \bigoplus_{C'\in\calY} C' = C_2.
	\end{equation}
	By the same argument as above, if $C_1$ is not relevant, then $C_2$ is not relevant.
\end{proof}

\subsection{Lemmas about exchanging basis cycles}
\label{sub:basis_ex}

	Much of the theory in this text relies on the procedure of swapping a given cycle in a basis for another cycle. This notion of basis exchange can be used to define a matroid, but it is easier to understand these properties as a consequence of the cycle space being a vector space. In this section, we state this notion of basis exchange in Lemma~\ref{Lem:Bas_Ex}; construct conditions for a cycle to be relevant in Lemma~\ref{Lem:Relevant_MCB_Test}; and verify that Definitions~\ref{def:rel} and~\ref{def:rel_2} are relevant in Corollary~\ref{cor:rel_def_eq}.


	For these lemmas, it is useful to introduce further notation for expanding a cycle with respect to a basis. Recall from Eq.~\eqref{eq:cyc_exp_format_full}, a cycle $C$ can be expanded as the sum of basis cycles $B_i$ with binary coefficients $b_i$. Thus, $C$ is simply the sum of cycles $B_i$ with coefficients $b_i=1$. We define the following function to output these cycles. 
	\begin{definition}
		Let $\calB$ be a cycle basis. The \emph{expansion operator} $\expand{\calB}{\cdot}$ is the function on cycles $C\in\CycSp$ that outputs the unique subset $\expand{\calB}{C}\subseteq\calB$ that sums to $C$
		\begin{equation}
			C = \bigoplus_{\mathclap{B\in\expand{\calB}{C}}} B,\quad \forall C\in\CycSp.
		\end{equation}
		\label{def:ex_op}
	\end{definition}

	In the following lemma we show a cycle $C$ can be exchanged for cycles in its expansion $\expand{\calB}{C}$. {This well-known lemma has many equivalent forms, and is often referred to as the matroid property of the cycle space, see e.g., Proposition 4 in Ref.~\cite{gleissInterchangeability2000}.}

\begin{lemma}[Basis exchange]
	Let $\calB$ be a cycle basis, $C\in\CycSp$ be a cycle, and $B\in\expand{\calB}{C}$ be an arbitrary cycle in the expansion of $C$ with respect to $\calB$. If we replace $B$ with $C$ in $\calB$ then we obtain a new cycle basis $\calB' {=}\, (\calB{\setminus} \{B\}){\cup}\{C\}$.\nolinebreak 
	\label{Lem:Bas_Ex}
\end{lemma}
\begin{proof}
	To start, we expand $C\in\CycSp$ with respect to the basis $\calB$ as follows
	\begin{equation}
		C = B_1{\oplus} B_2 {\oplus}\hdots {\oplus} B_k, \enspace B_i{\in}\calB\text{, }1{\le} i{\le} k.\label{eq:C_exp_reg_BasEx}
	\end{equation}
	Note, in the above equation we fix $\expand{\calB}{C}=\{B_1,\hdots,B_k\}$. We want to show that if we swap $B_i$ with $C$ in $\calB$, then we obtain a new cycle basis. Without loss of generality, we will show that $\calB'\coloneq (\calB\setminus\{B_1\})\cup\{C\}$ is a cycle basis.
	To show this, we add $C$ and $B_1$ to both sides of Eq.~\eqref{eq:C_exp_reg_BasEx} to obtain
	\begin{equation}
		B_1 = C\oplus B_2\oplus\hdots\oplus B_k.\label{eq:B1_exp_reg_BasEx}
	\end{equation}
	Equation~\eqref{eq:B1_exp_reg_BasEx} shows that $B_1$ is in the span of $\calB'$, and that our original set $\calB$ is contained in the span of $\calB'$. $\calB$ is a cycle basis, so $\calB'$ must span the full cycle space. $C$ is linearly independent of $\{B_2,\hdots,B_k\}$, otherwise $B_1$ would be in the span of $\{B_2,\hdots,B_k\}$ by Eq.~\eqref{eq:B1_exp_reg_BasEx}. Thus, $\calB'$ is a basis.
\end{proof}

	
	{In the following lemma, we introduce an alternate definition of relevance in terms of the expansion of a cycle with respect to a MCB. We will use this lemma as a backbone for our theoretical results in this manuscript and for testing the relevance of cycles in V-families.
	}

\begin{lemma}[Relevance criterion]
	Let $C\in\CycSp$ be a cycle, $\calM$ an MCB, and
	$$L = \max\{|C'|:C'\in\expand{\calM}{C}\}$$
	the length of the largest cycle in the expansion of $C$ with respect to $\calM$.
	Then, the following are true
	\begin{enumerate}
		\item $L\le |C|$, i.e., no cycles in $\expand{\calM}{C}$ are larger than $C$,
		\item $L=|C|$ if and only if $C$ is relevant $C\in\mathcal{C_R}$. Furthermore, if $C\in\mathcal{C_R}$, then for any $C'\in\expand{\calM}{C}$, $|C'|=|C|$, the set $\calM'=(\calM\setminus\{C'\})\cup\{C\}$ is an MCB.
	\end{enumerate}
	\label{Lem:Relevant_MCB_Test}
\end{lemma}
\begin{proof}
	(1) Let $C'\in\expand{\calM}{C}$ be an arbitrary cycle in the expansion of $C$ with respect to $\calM$. By Lemma~\ref{Lem:Bas_Ex}, $\calB=(\calM\setminus \{C'\})\cup\{C\}$ is a cycle basis. The sum of cycle lengths in $\calB$ is
	\begin{equation}
		\sum_{B\in\calB} |B| = \sum_{B\in\calM} |B| - |C'| + |C|.
	\end{equation}
	$\calM$ is an MCB, so the sum of cycle lengths in $\calB$ is greater than or equal to the sum of cycle lengths in $\calM$, and by the above equation
	\begin{equation}
		\begin{split}
		\sum_{B\in\calM} |B| - |C'| + |C| \ge \sum_{B\in\calM} |B| \quad \\\longrightarrow \quad |C| \ge |C'|.
		\end{split}
	\end{equation}
	Thus, $C$ is not smaller than any of the cycles in its expansion with respect to $\calM$ as claimed.
	
	(2, $\Rightarrow$) 
		Let $\calM'$ be an MCB containing $C$. 
		Let $C'\in\expand{\calM}{C}$ be independent of $\calM\setminus\{C\}$. By item 1, $|C'|\le |C|$. 
		Then, $\calB=\calM'\setminus\{C\}\cup\{C'\}$ is another cycle basis.
		\begin{equation}
			\sum_{\mathclap{B\in\calM'}} |B| \le \sum_{\mathclap{B\in\calB}} |B| = \sum_{\mathclap{B\in\calM'}} |B| - |C| + |C'|.
		\end{equation}
		Hence $|C'|\ge|C|$. Thus $|C'|=|C|$.
\end{proof}

\begin{corollary}[{Ref~\cite{vismaraUnion1997}, Lemma 1}]
	A cycle $C$ is not relevant if and only if it is generated by smaller cycles. In other terms, Definition~\ref{def:rel_2} is equivalent to Definition~\ref{def:rel}.
	\label{cor:rel_def_eq}
\end{corollary}
\begin{proof}
	($\Rightarrow$) Let $C$ not be relevant. Let $\calM$ be a given MCB. $C$ is the sum of shorter cycles
	\begin{equation}
		C = \bigoplus_{B\in\expand{\mathcal{M}}{C}} B
	\end{equation}
	where $|B|<|C|$ for $B\in\expandd{\calM}{C}$ by Lemma~\ref{Lem:Relevant_MCB_Test}.
	
	($\Leftarrow$) Let $C$ be the sum of smaller cycles $\calX$, i.e., $C=\bigoplus_{C'\in\calX}C'$ where $|C'|<|C|$ for $C'\in\calX$. Then, the expansion of $C$ with respect to an MCB $\calM$ satisfies
	\begin{equation}
		\expand{\calM}{C} \subseteq \bigcup_{\mathclap{C'\in\calX}} \expand{\calM}{C'}.
	\end{equation}
	By construction of $\calX$ and Statement 1 of Lemma~\ref{Lem:Relevant_MCB_Test}, the cycles in $\expand{\calM}{C}$ are smaller 
	\begin{equation}
		|B|\le |C'|<|C|,\quad C'\in\calX,B\in\expand{\calM}{C'}
	\end{equation}
	as claimed.
\end{proof}

\subsection{Partitions of the relevant cycles}
\label{sub:background_partitions}

The V-families output by Vismara's algorithm~\cite{vismaraUnion1997} are often more useful than the relevant cycles themselves. 
Unfortunately, these families are not unique, as they rely on a particular enumeration of the vertices. 
Here we review two works that extend V-families to unique partitions of the relevant cycles~\cite{gleissInterchangeability2000,kolodzikUnique2012}.


Gleiss {et al.}~\cite{gleissInterchangeability2000} extend the V-families to include cycles that can be obtained by the basis exchange property. They achieve this with the following.



\begin{definition}[Ref.~\cite{gleissInterchangeability2000}, Definition 6]
	Two relevant cycles $C_1,C_2\in\RelCyc$ of equal length $|C_1|=|C_2|$ are \emph{interchangeable} if there exist cycles $\calX\subset\RelCyc$, $C_1,C_2\notin\calX$, such that
	\begin{enumerate}
		\item $\calX\cup\{C_2\}$ is a linearly independent set of relevant cycles,
		\item $C_1=C_2\oplus \bigoplus_{C'\in\calX} C'$,
		\item $|C'|\le |C_1|$ for $C'\in\calX$.
	\end{enumerate}
	\label{def:intchg}
\end{definition}

In words, $C_1$ and $C_2$ are interchangeable if they are of equal length, and $C_1$ is a sum of $C_2$ and cycles of smaller or equal length.

By linear independence (condition 1), there exists a basis $\calB_2$ such that $\calX\cup\{C_2\}\subseteq\calB_2$. 
The expansion of $C_1$ is $\expand{\calB_2}{C_1} = \calX\cup\{C_2\}$ by condition 2.
By Lemma~\ref{Lem:Bas_Ex}, the exchange $\calB_1=(\calB_2\setminus\{C_2\})\cup\{C_1\}$ is a basis.

Interchangeability is an equivalence relation that partitions the relevant cycles into classes. 
Cycles in a V-family are interchangeable, so the number of interchangeability classes is polynomial in number.
For the chained diamonds in Figure~\ref{fig:three_diamonds}, all large relevant cycles belong to the same interchangeability class.
Interchangeability classes are tricky to compute because they can not be obtained using a single cycle prototype from each V-family~\cite{gleissInterchangeability2000}. With additional theory, interchangeability classes can be computed in polynomial time~\cite{bergerComputing2017}.



\begin{figure*}
	\centering
	\includegraphics{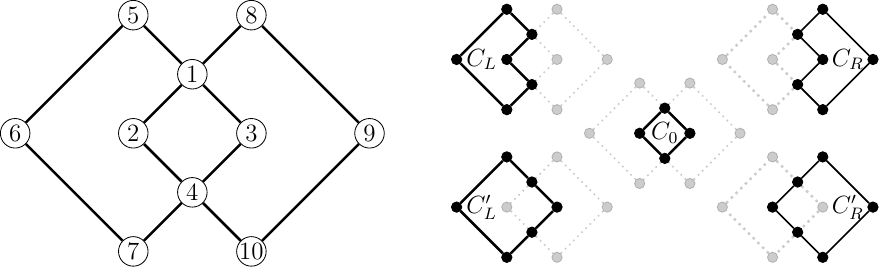}
	\caption{ (Adapted from Ref.~\cite{gleissInterchangeability2000}, Figure 4). A graph consisting of two overlapping diamonds. The graph has five relevant cycles $\RelCyc=\{C_L,C_L',C_0,C_R,C_R'\}$, depicted separately. There are four MCBs $\calM_{1,1}=\{C_L,C_0,C_R\}$, $\calM_{1,2}=\{C_L,C_0,C_R'\}$, $\calM_{2,1}=\{C_L',C_0,C_R\}$, $\calM_{2,2}=\{C_L',C_0,C_R'\}$, where each MCB has one left cycle in $\{C_L,C_L'\}$, one right cycle in $\{C_R,C_R'\}$, and center cycle $C_0$.
	}
	\label{fig:overlapping_diamonds}
\end{figure*}

$C_1,C_2$ being interchangeable does not guarantee the existence of an MCB $\calM_2$ where $\calM_1=(\calM_2\setminus\{C_2\})\cup\{C_1\}$ is also minimal. To see this, consider the graph in Figure~\ref{fig:overlapping_diamonds}.
First, we observe that $C_L = C_0\oplus C_L'$ and $C_0 = C_L \oplus C_L'$. Similarly, for the right cycles
\begin{equation}
	\begin{split}
	C_R = C_0\oplus C_R' &=( {C_L \oplus C_L'}) \mathrel{\oplus} C_R' \\&= C_L \oplus (C_L'\oplus C_R')
	\end{split}
	\label{eq:ind_ex}
\end{equation}
where $|C_L|=|C_L'|=|C_R'|=|C_R|$. 
Thus, $C_L$ and $C_R$ are interchangeable, as in Definition~\ref{def:intchg}, where $\calX=\{C_L',C_R'\}$.
However, each MCB has exactly one left and right cycle, so we cannot exchange $C_L$ for $C_R$ between MCBs.

{The interchangeability relation splits cycles into classes that may be too coarse for some applications.~\cite{kolodzikUnique2012}} 
As an illustrative example, consider the cube graph, where the relevant cycles are the six face cycles $F_1,F_2,\hdots,F_6$. Any face is the sum of the other face cycles, e.g.,
\begin{equation}
	F_6 = F_1 \oplus  F_2\oplus F_3 \oplus F_4 \oplus F_5.\label{eq:cube_face_intchg}
\end{equation}
See Figure~\ref{fig:cube}. 
Thus, all face cycles belong to the same interchangeability class. 
Kolodzik {et al.} \cite{kolodzikUnique2012} argue that this interchangeability class is too coarse and propose the face cycles should be placed in separate groups. They achieve this using a more restrictive pairwise relation that closely resembles the relation of cycles in Vismara's families.

\begin{figure}
	\centering
	\includegraphics{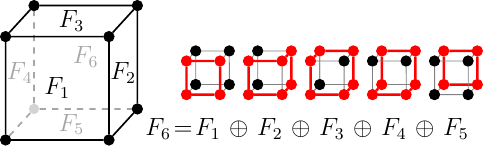}
	\caption{Depiction of the cube graph with labeled face cycles $F_1,F_2,\hdots,F_6$. The partial sums of $ F_1\oplus F_2\oplus\hdots \oplus F_5$ are depicted in red to show how they sum to $F_6$.}
	\label{fig:cube}
\end{figure}


\begin{definition}[Ref.~\cite{kolodzikUnique2012}, Definition 1]
	Let $C_1,C_2\in\RelCyc$, where $|C_1|=|C_2|$, then $C_1$ and $C_2$ are unique ring family (URF)-pair-related if
	\begin{enumerate}
		\item $C_1$ and $C_2$ share at least one common edge
		\item There exists a set of strictly {smaller} cycles $\calX$ such that $C_1 = C_2\oplus \bigoplus_{C'\in \calX}C'$.
	\end{enumerate}
	\label{def:URF_Relation}
\end{definition}

The key distinction here is the cycles $\calX$ used to exchange $C_1$ and $C_2$ are strictly smaller. The faces of the cube are of the same length, so Eq.~\eqref{eq:cube_face_intchg} does not satisfy the URF-pair-relation. 
For the graph composed of chained diamonds in Figure~\ref{fig:three_diamonds}, the large cycles are URF-pair-related because they are exchanged by adding shorter diamonds. More generally, cycles in a V-family satisfy the URF-pair-relation.

The URF-pair-relation is not transitive, so it is not an equivalence relation. Instead, the URFs are given by the transitive closure of the URF-pair-relation. 
{The URFs can be computed in polynomial time using the V-families.~\cite{kolodzikUnique2012}}

\section{Theory: Concepts}
\label{sec:Theory}

\begin{figure}[t]
	\centering
	\includegraphics{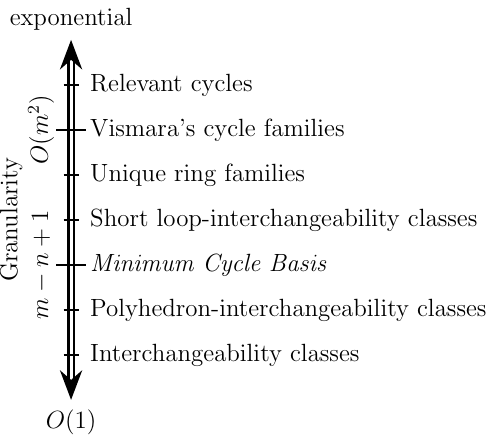}
	\caption{
		Relative granularity of each partition of the relevant cycles. The more granular partitions are refinements of the coarser ones. The axis represents the number of sets in each partition. Notably, a {\tt pi} class is a collection of {\tt sli} classes, and a {\tt sli} class is a collection of V-families. 
		An MCB is not a partition of the relevant cycles, but it is useful for comparing the granularity of each partition.
		A {\tt pi} class has a fixed number of cycles in each MCB;
		a {\tt sli} class contains at most one cycle in any given MCB. 
	}
	\label{fig:partition_granularity}
\end{figure}

In this section, we introduce two partitions of the relevant cycles and related theory. The definitions of these partitions are inspired by those in Refs.~\cite{gleissInterchangeability2000} and~\cite{kolodzikUnique2012}. We summarize the granularity of these partitions in Figure~\ref{fig:partition_granularity}.

At a coarse level, we define the polyhedron-interchangeability ({\tt pi}) relation, 
a modification of the interchangeability relation from Ref.~\cite{gleissInterchangeability2000}.
Unlike the interchangeability relation, the {\tt pi} relation is defined using the basis exchange property between MCBs.
This allows us to more precisely characterize the lack of uniqueness in MCBs.
We define the {\tt pi} relation and elaborate on its advantages in Section~\ref{sub:DICE_Def}. 
Also, we introduce a physical interpretation of pairs of cycles that are {\tt pi-interchangeable} as being shared faces of an abstract polyhedron. We describe this in Section~\ref{sub:Poly}. Thus, we suggest the intuition to interpret {\tt pi} classes as polyhedron clusters.



At a granular level, we define the short loop-interchangeability {(\tt sli)} relation, which modifies the URFs from Ref.~\cite{kolodzikUnique2012}. In particular, the {\tt sli} relation is equivalent to the URF-pair-relation except that cycles do not need to share edges. Our definition is more amenable to theoretical results, e.g., {\tt sli} is an equivalence relation. In computation, the {\tt sli} classes are used in place of cycles since they are polynomial in number. The {\tt sli} relation is introduced in Section~\ref{sub:TICE_def}.

\subsection{Polyhedron-interchangeability: definition and basic concepts}
\label{sub:DICE_Def}

Here, we define and investigate the {\tt pi} relation. 
\begin{definition}
	$C_1,C_2\in\RelCyc$ are \emph{polyhedron-interchangeable} if there exists an MCB $\calM_2$, $C_2\in\calM_2$, such that $\calM_1=(\calM_2\setminus\{C_2\})\cup\{C_1\}$ is also an MCB.
	\label{def:dir_intchg}
\end{definition}
\noindent We define the operator $C_1\pInt C_2$ to denote $C_1$ is polyhedron-interchangeable for $C_2$.

An equivalent definition was introduced in Ref.~\cite{gleissInterchangeability2000} as a ``stronger interchangeability relation'', but it was not analyzed in detail. In particular, they argue this relation should be avoided because it is not symmetric. However, we find this is incorrect and that {\tt pi} is an equivalence relation. Furthermore, the {\tt pi} relation has two important benefits:
\begin{itemize}
	\item \emph{(Interpretability)} Two cycles $C_1$ and $C_2$ such that $C_1{\pInt}\, C_2$ are related by the basis exchange property between MCBs. 
	This is not the case for interchangeability.
	$C_L$ and $C_R$ in Figure~\ref{fig:overlapping_diamonds} are interchangeable by Definition~\ref{def:intchg}, but they are not related by the exchange of cycles in MCBs.
	\item \emph{(Computability)} The {\tt pi} classes are computed in polynomial time using the individual cycle prototypes from the V-families. See Theorem~\ref{Thm:DirInt_Comp} and the following discussion. Interchangeability classes are also computable in polynomial time, but they require additional theory~\cite{bergerComputing2017}.
\end{itemize}



\subsubsection{Theoretical properties}
\label{sub:DICE_Props}

\begin{figure*}
	\centering
	\includegraphics[width=6.5in]{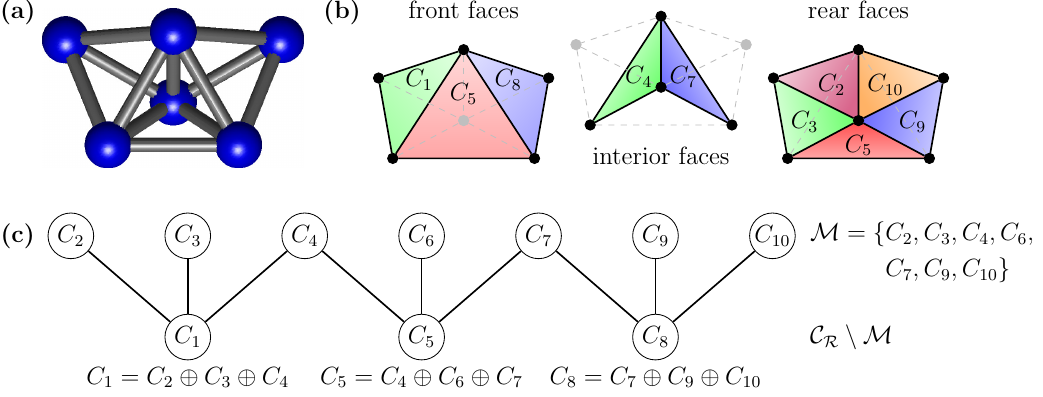}
	\caption{
		An example showing relevant cycles and the {\tt pi} classes.
		\textbf{(a)} Graph with 6 nodes composed of 3 overlapping tetrahedra. 
		\textbf{(b)} Depiction of the relevant cycles $\RelCyc=\{C_1,\hdots,C_{10}\}$. 
		\textbf{(c)} Bipartite graph representation used to compute the {\tt pi} classes.
	}
	\label{fig:DICE_alg}
\end{figure*}

Here, we state our theoretical results for the {\tt pi} relation and their implications. 


\begin{theorem}
	{\tt pi} is an equivalence relation on $\mathcal{C}_{\mathcal{R}}$.
	\label{Thm:DirInt_Eq_Rel}
\end{theorem}

The proofs of Theorem~\ref{Thm:DirInt_Comp} and the other results in Section~\ref{sub:DICE_Props} are in Appendix~\ref{app:DICE_Proofs}.

As an equivalence relation, {\tt pi} partitions the relevant cycles into equivalence classes. 
Although {\tt pi} satisfies the mathematical definition of an equivalence relation, we do not consider a pair of {\tt pi} cycles to be equivalent. For instance, the faces of the cube are {\tt pi}, but they are visually distinct. We find the {\tt pi} classes are better interpreted as polyhedron clusters as discussed in Section~\ref{sub:Poly}.

Definition~\ref{def:dir_intchg} for {\tt pi} is valuable because it explicitly uses the basis exchange property. 
However, Definition~\ref{def:dir_intchg} is hard to test in practice.
Thus, it is also useful to have an equivalent, algebraic statement of this definition, which we state below. 




\begin{lemma}
	Let $C_1,C_2\in\RelCyc$ where $|C_1|=|C_2|$. $C_1\pInt C_2$ 
	if and only if there exist cycles $\calX\subseteq\RelCyc$, $C_1,C_2\notin \calX$, such that
	\begin{enumerate}
		\item there exists an MCB $\calM_2$ such that $\calX\cup\{C_2\}\subseteq\calM_2$ and,
		\item $C_1$ is given by the sum
		\begin{equation}
			C_1 = C_2\oplus\bigoplus_{C'\in\calX}C',\label{eq:dir_int_C1_Ex}
		\end{equation}
		i.e., the expansion of $C_1$ in $\calM_2$ is given by $\expand{\calM_2}{C_1}=\calX\cup\{C_2\}$.
	\end{enumerate}
	\label{Lem:dir_intchg_2}
\end{lemma}
\noindent Note, the cycles in $\calX$ are bounded in length, $|C'|<|C_1|$ for $C'\in\calX$, by Statement 1 of Lemma~\ref{Lem:Relevant_MCB_Test}. Moreover, $\calX$ satisfies the conditions for $C_1,C_2$ to be interchangeable as in Ref.~\cite{gleissInterchangeability2000} (Definition~\ref{def:intchg}) with the added constraint that $\calX\cup\{C_2\}$ belongs to a MCB. 
Thus, {\tt pi} implies interchangeability. The reverse is not true, e.g., $C_L$ and $C_R$ in Figure~\ref{fig:overlapping_diamonds} are interchangeable but not {\tt pi}-interchangeable.

In brief, Lemma~\ref{Lem:dir_intchg_2} states a cycle $C_1$ is {\tt pi}-related to the equal length cycles in its expansion with respect to a MCB $\calM_2$. One may wonder, do we need many MCBs to construct the {\tt pi} classes? Fortunately, it is sufficient to generate a single MCB, as stated below.

\begin{theorem}
	Let $\calM$ be a given MCB. Let $C_2\in\calM$ and $C_1\in\RelCyc{\setminus}\calM$. If $C_2\in\expandd{\calM}{C_1}$ and $|C_1|=|C_2|$ then $C_1\pInt C_2$.
	The {\tt pi} classes are given by the transitive closure of these relations.\nolinebreak 
	\label{Thm:DirInt_Comp}
\end{theorem}


Theorem~\ref{Thm:DirInt_Comp} reduces the task of computing the {\tt pi} classes to breadth-first search. As input, we compute an MCB $\calM$ and the relevant cycles $\RelCyc$. Then, we expand the relevant cycles in $\calM$. Using this, we create a bipartite graph, where nodes represent the relevant cycles split by if they are in $\calM$. Two cycles $C_1\in\calM$ and $C_2\in\RelCyc\setminus\calM$ are joined by an edge if $C_2\in\expand{\calM}{C_1}$. The {\tt pi} classes are the connected components of this bipartite graph. See Figure~\ref{fig:DICE_alg}~(c) for an instance of this bipartite graph.


%

The {\tt pi} relation can be used to construct a path from one MCB to another by only swapping one pair of cycles at a time. 
\begin{lemma}
	Let $\calM_1,\calM_2$ be distinct MCBs. There exists $C_1\in\calM_1\setminus\calM_2$ and $C_2\in\calM_2\setminus\calM_1$ such that $\calM_1'\,{=}\,(\calM_1\setminus\{C_1\})\cup\{C_2\}$ is an MCB.
	\label{Lem:DICE_Irreducible}
\end{lemma}
In this procedure, each cycle in $\calM_1$ also belongs to $\calM_2$ or it is {\tt pi}-interchangeable for a cycle in $\calM_2\setminus\calM_1$. This results in a bijective mapping from $\calM_1$ and $\calM_2$ between {\tt pi}-interchangeable cycles. From this, we obtain the following result.
\begin{corollary}
	Let $\calM_1,\calM_2$ be MCBs and $\mathcal{PI}\subseteq\RelCyc$ be a {\tt pi} class. Then, $\calM_1$ and $\calM_2$ have the same number of cycles in $\mathcal{PI}$:
	\begin{equation}
		|\calM_1\cap\mathcal{PI}| = |\calM_2\cap\mathcal{PI}|.
	\end{equation}
\end{corollary}
\noindent This introduces a notion of rank for the {\tt pi} classes, given by the number of cycles in the {\tt pi} class that belong to a MCB $|\calM\cap\mathcal{PI}|$.

\subsubsection{{\tt pi} classes as polyhedron clusters}
\label{sub:Poly}


The partitions of the relevant cycles 
were introduced as a means of addressing the lack of uniqueness in MCBs. In the examples considered so far, lack of uniqueness can be caused by symmetries or by three-dimensional data. Here, we will discuss the latter. Specifically, we will discuss how {\tt pi} classes may be used to identify three-dimensional structures.

First, consider a graph composed of a convex polyhedron with face cycles $F_1,\hdots,F_k$ ordered by decreasing length $|F_1|\ge|F_2|\ge\hdots\ge|F_k|$.
For instance, our graph could be one of the platonic solids, e.g., the tetrahedron or cube. 
We observe, that each edge in the polyhedron participates in two face cycles. Thus, all edges cancel in the sum $F_1\oplus\hdots\oplus F_k$. In particular, the face cycles sum to the empty set
\begin{equation}
	F_1\oplus \hdots\oplus F_k=\emptyset.
	\label{eq:poly_null_sum}
\end{equation}
Next, we re-arrange Eq.~\eqref{eq:poly_null_sum} to express $F_1$ as the sum of the remaining face cycles
\begin{equation}
	F_1 = F_2\oplus (F_3\oplus\hdots \oplus F_{k}).\label{eq:dir_int_poly}
\end{equation}
See Figure~\ref{fig:cube} for the case of the cube. 
Eq.~\eqref{eq:dir_int_poly} resembles Eq.~\eqref{eq:dir_int_C1_Ex} in the definition of the {\tt pi} relation given by Lemma~\ref{Lem:dir_intchg_2}.
Indeed, if we add two constraints (i) $|F_1|=|F_2|$ and (ii) $F_2,\hdots,F_{k}$ form an MCB, then $F_1$ and $F_2$ are {\tt pi}-interchangeable by Lemma~\ref{Lem:dir_intchg_2}.

\begin{figure}[t]
	\centering
	\includegraphics{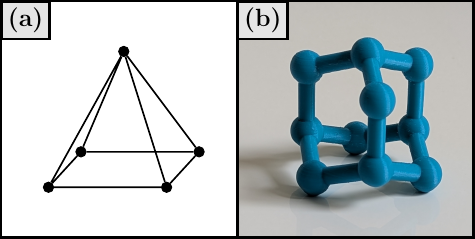}
	\caption{\textbf{(a)} The square pyramid, a polyhedron whose faces do not form a {\tt pi} class. 
		The unique MCB of this graph is the set of triangular faces. 
		\textbf{(b)} Carbon skeleton of adamantane, $\rm C_{10}H_{16}$, the smallest diamondoid, or a small molecule with the same structure as diamond. The faces of adamantane are a {\tt pi} class, 
		but they are not flat as in a classical polyhedron.}
	\label{fig:poly_counter}
\end{figure}

We remark, not all polyhedra can be used to obtain a pair of face cycles $F_1,F_2$ such that $F_1\pInt F_2$. 
For example, if $F_1$ is strictly longer than the other faces $|F_1|>|F_i|$, $i=2,\hdots,k$, then $F_1$ is not relevant as the sum of smaller cycles by Eq.~\eqref{eq:dir_int_poly}. This is the case for the square pyramid, where the square base is the largest cycle in Figure~\ref{fig:poly_counter} (a). Nevertheless, polyhedra are useful for constructing examples for the {\tt pi} relation.

The {\tt pi} classes can be used to identify three-dimensional structures. If $C_1,C_2$ are {\tt pi}-interchangeable, then we can use Lemma~\ref{Lem:dir_intchg_2} to obtain a set of cycles $\calP=\{C_1,C_2\}\cup\calX$ with null sum
\begin{equation}
	C_1\oplus C_2 \oplus\bigoplus_{C'\in\calX}C' = \emptyset.
\end{equation}
This matches the sum of face cycles for polyhedra, Eq.~\eqref{eq:poly_null_sum}. Thus, we propose to interpret these sets $\calP$ with null sums obtained by Lemma~\ref{Lem:dir_intchg_2} as polyhedra.

It is not necessarily the case that we can draw these null sums as polyhedra, e.g., adamantane in Figure~\ref{fig:poly_counter} (b) does not have flat faces. Instead, the sets $\calP$ obtained via {\tt pi} relations might be more appropriately defined as abstract polyhedra~\cite{mcmullenAbstract2002} with possibly curved faces and edges. 

\begin{figure}[t]
	\centering
	\includegraphics{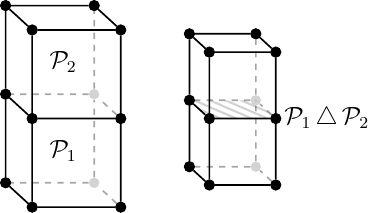}
	\caption{Graph composed of two cubes $\calP_1$ and $\calP_2$. Both $\calP_1$ and $\calP_2$ are sets of six face cycles. The symmetric difference $\calP_1\symdiff\calP_2$ (depicted right) gives the ten faces of the larger rectangular prism containing $\calP_1$ and $\calP_2$. The interior face cycle (shaded with gray stripes) belonging to both $\calP_1$ and $\calP_2$ is not present in $\calP_1\symdiff\calP_2$.}
	\label{fig:polyhedron_pair}
\end{figure}

We remark, a {\tt pi} class is often not a single polyhedron. For instance, if we glue two cubes together so that they share a face as shown in Figure~\ref{fig:polyhedron_pair}, then all square faces are {\tt pi}-interchangeable by transitivity. Thus, in the same way that a biconnected component is a cluster of loops, we argue that a {\tt pi} class should be interpreted as a cluster of polyhedra.

Furthermore, we may extend the polyhedra obtained by a {\tt pi} class to a vector space. As before, the addition operator of this vector space is the symmetric difference, and the scalars are the binary GF(2) field. This ``polyhedron space'' is simply the null space of the relevant cycles, i.e., if polyhedra $\calP_1$ and $\calP_2$ have null sums then so does their symmetric difference
\begin{equation}
	\bigoplus_{C'\in\calP_1\symdiff\calP_2} C' = \bigoplus_{C'\in\calP_1}C'\oplus\bigoplus_{C''\in\calP_2}C''.
\end{equation}
A polyhedron space is likely useful for graphs representing material structures, where there are three intrinsic spatial dimensions. It appears pertinent to define a notion of ``minimum polyhedron basis'' to serve as a summary statistic for this space. We lack the algorithms and data to justify such a construction. 

\subsection{Short loop-interchangeability: definition, equivalence classes, \& computation}

\label{sub:TICE_def}








Here, we introduce the second {\tt sli} relation used to partition the relevant cycles.
To motivate this, we discuss where the {\tt pi} classes are less effective as a descriptive tool.

In Section~\ref{sub:Poly}, the {\tt pi} relation was related to three-dimensional structures, i.e., polyhedra with non-flat faces.
Consider the graph shown in Figure~\ref{fig:2ManyPoly}~(a). 
The polyhedra are sampled using an MCB $\calM$ by merging each cycle $C\in\RelCyc\setminus\calM$ with its expansion $\expand{\calM}{C}$ as in Figure~\ref{fig:2ManyPoly}~(b).
These polyhedra describe three-dimensional structures like beads of a bracelet, as in Figure~\ref{fig:2ManyPoly}~(c).

Visually, the third polyhedron $\calP_3$ appears redundant, where its two beads are a combination of the other polyhedra. 
This is not resolved by treating the polyhedra as a vector space since $\calP_3\neq\calP_1\symdiff\calP_2$. 
As the bracelet grows in length, the number of visually redundant polyhedra grows exponentially with the number of relevant cycles.

\begin{figure}[t]
	\centering
	\includegraphics{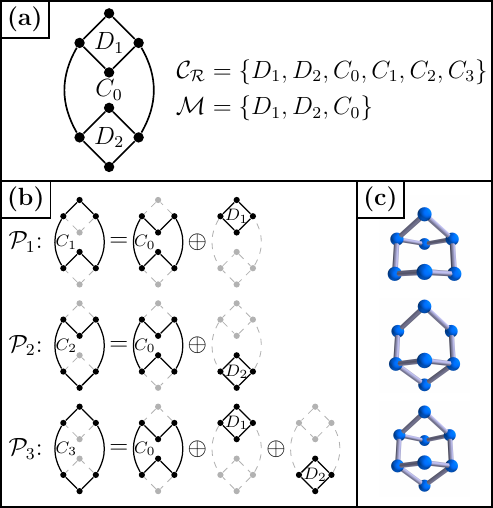}
	\caption{
		\textbf{(a)} Bracelet graph composed of two diamonds chained together. 
		The planar faces form an MCB $\calM$. 
		\textbf{(b)}~Three polyhedra $\calP_1\,{=}\,\{C_1{,}C_0{,}D_1\}, \calP_2\,{=}\,\{C_2{,}C_0{,}D_2\},$ and $\calP_3=\{C_3{,}C_0{,}D_1{,}D_2\}$ constructed using the cycles not in $\calM$ and their expansions in $\calM$.
		\textbf{(c)}~Three-dimensional representation of the polyhedra as the union of their face cycles.
	}
	\label{fig:2ManyPoly}
\end{figure}

To address this, 
we observe the cycles in the examples above vary in length. 
We argue that larger cycles that differ by shorter ones should be treated as equivalent. 
For instance, we would group the circumference cycles around the carbon nanotube into an equivalence class without sampling polyhedra. 
We formalize this equivalence with the {\tt sli} relation.


\begin{definition}
	$C_1,C_2\in\mathcal{C}_{\mathcal{R}}$, $|C_1|=|C_2|$, are \emph{short loop-interchangeable} if
	\begin{equation}
		C_1=C_2\oplus\smCyc,
		\label{eq:triv_intchg}
	\end{equation}
	i.e., there exist cycles $\calY$, $|C| < |C_1|$ for $C\in\calY$, such that
	\begin{equation}
		C_1 = C_2\oplus\bigoplus_{C\in\calY}C.
	\end{equation}
	\label{def:triv_intchg}
\end{definition}
\noindent We use the operator $C_1\slInt C_2$ to denote $C_1$ is short loop-interchangeable for $C_2$.

The numerous cycles around the circumference of the carbon nanotube in Figure~\ref{fig:tubes}~(a) form a single {\tt sli} class.

{\tt sli} is a stronger relation than {\tt pi}.
A pair of {\tt pi}-related cycles may differ by equal length cycles. Therefore, the faces of a polyhedron with uniform lengths (e.g., the cube) form a {\tt pi} class but not a {\tt sli} class. 
We sample polyhedra from {\tt pi} classes that do not reduce to a {\tt sli} class. This is discussed in more detail in Section~\ref{sub:DICE_TICE_Comp}.

\begin{figure}[t]
	\centering
	\includegraphics{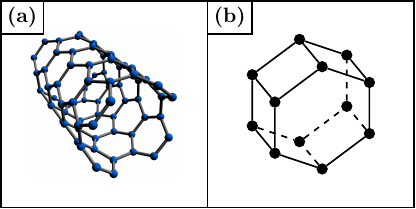}
	\caption{
		\textbf{(a)} Small carbon nanotube. The large cycles around the circumference of the nanotube are both a {\tt sli} class and a URF. 
		\textbf{(b)}~Hexagonal prism graph. The two hexagonal faces form a {\tt sli} class but not a URF. 
	}
	\label{fig:tubes}
\end{figure}

Definition~\ref{def:triv_intchg} was first introduced in the dissertation by Gleiss~\cite{gleissShort2001} as a stronger interchangeability relation. Two of the results shown here, Lemma~\ref{Lem:TrInt_Eq_Rel} and Corollary~\ref{Cor:tr_intchg_rank}, were first shown there. 
We include them for completeness. 
Here, we add a method of computing the {\tt sli} classes, Lemma~\ref{Lem:TrEx_Alg}. 
Also, our two-fold partition using both {\tt pi} and {\tt sli} classes is novel.
The proofs of these results are in Appendix~\ref{app:TICE_Proofs}

Our approach of equivocating cycles that differ by shorter cycles is inspired by the URFs in Ref.~\cite{kolodzikUnique2012}. 
The {\tt sli} relation is weaker than the URF-pair-relation.
Specifically, cycles do not need to share edges to be {\tt sli}-interchangeable. 
For example, the two hexagonal cycles around the perimeter of the prism shown in Figure~\ref{fig:tubes}~(b) are {\tt sli}-interchangeable but not URF-pair-related. 
In MD data, we find these two definitions behave similarly. 
We use the {\tt sli} relation because it is simpler, making it easier to develop theory.

\begin{lemma}
	{\tt sli} is an equivalence relation on $\mathcal{C}_{\mathcal{R}}$.
	\label{Lem:TrInt_Eq_Rel}
\end{lemma}

By Lemma~\ref{Lem:TrInt_Eq_Rel}, the {\tt sli} relation partitions the relevant cycles into equivalence classes. 
To compute these classes, we redefine {\tt sli} using an MCB.

\begin{lemma}
	\label{Lem:TrEx_Alg}
	Let $\calM$ be a given MCB. $C_1,C_2\in\RelCyc$, $|C_1|=|C_2|$, are {\tt sli}-interchangeable if and only if their expansions in $\calM$ are equal up to shorter cycles.
\end{lemma}

\noindent Cycles in Vismara's families differ by shorter cycles. Therefore, the relevant cycles can be broken into {\tt sli} classes in polynomial time using Vismara's cycle families.

Note, if a cycle $C_1$ belongs to the MCB $\calM_1$, then its expansion in $\calM_1$ is $\expand{\calM_1}{C_1}=\{C_1\}$. 
If $C_2\slInt C_1$ then the expansion of $C_2$ in $\calM_1$ is $C_1$ and shorter cycles by Lemma~\ref{Lem:TrEx_Alg}. 
Thus, we may exchange $C_2$ for $C_1$ in $\calM_1$ to obtain a new MCB by Lemma~\ref{Lem:Relevant_MCB_Test}. 
We state this as a corollary.

\begin{corollary}
	Let $C_1,C_2$ be cycles such that $C_1{\slInt}\, C_2$. 
	Let $\calM_1$ be an MCB where $C_1{\in}\,\calM_1$. Then, $\calM_2=(\calM_1\setminus\{C_1\})\cup\{C_2\}$ is an MCB.
	\label{Cor:TrInt_Exchanges}
\end{corollary}

Also, observe $C_2$ is not in the MCB $\calM_1$ since it is already in the span of $C_1$ and shorter cycles in $\calM_1$. 
Therefore, $C_1$ is the only cycle from its {\tt sli} class that belongs to $\calM_1$:

\begin{corollary}
	Let $\calM$ be an MCB and $\calS\subseteq\RelCyc$ a {\tt sli} class. At most one cycle in $\calS$ belongs to $\calM$, i.e., $|\calM\cap\calS|\le 1$. 
	\label{Cor:tr_intchg_rank}
\end{corollary}

Note, the intersection may be empty, i.e., $|\calM\cap\calS|=0$. For instance, consider the cube where each face is its own {\tt sli} class, and only five faces belong to each MCB.

\section{Algorithms for sampling {\tt pi} and {\tt sli} classes}
\label{sec:algorithm}

In this section, we review and upgrade algorithms for computing the relevant cycle families, {\tt pi} classes, and {\tt sli} classes. In Section~\ref{sub:Vis_Fam}, we modify the relevant cycle families from Ref.~\cite{vismaraUnion1997} to be easier to compute in our setting. In Section~\ref{sub:MCB_comp}, we {show how a MCB can be extracted from these families using the algorithms from Ref.~\cite{kavithaFaster2004}.} In Section~\ref{sub:MCB_COB}, we show how to compute the expansion of cycles with respect to a MCB. In Section~\ref{sub:DICE_TICE_Comp}, we compute the {\tt pi} and {\tt sli} classes using the MCB representation of the cycle families.
%
%

\subsection{Modified Vismara's cycle families}
\label{sub:Vis_Fam}


\subsubsection{Definition}

{Here, we describe a method for computing candidate families of relevant cycles. 
We modify the construction of the V-families as introduced by Vismara~\cite{vismaraUnion1997} by rooting at an edge belonging to a fundamental cycle basis rather than a node. 
Compare Figures~\ref{fig:Vis_Fams} and~\ref{fig:mod_vis_fams}.
We call the modified cycle families V'-families.}
\begin{definition}[V'-family]
	Let $G$ be a graph and $T$ a given spanning tree of $G$. Let $e_j\in E(G)\setminus E(T)$ be an enumeration of the edges outside of the spanning tree $T$.
	\begin{itemize}
		\item The odd V'-family with descriptors $e_j = (u_j,v_j)$ and node $p$, where $d(u_j,p)=d(v_j,p)$, is the set of cycles with largest-indexed edge $e_j$ in $E(G)\setminus E(T)$ composed of $e_j$ and shortest path pairs from $u_j$ and $v_j$ to $p$. See Figure~\ref{fig:mod_vis_fams}, left.
		\item The even V'-family with descriptors $e_j=(u_j,v_j)$ and edge $(p,q)$, where $d(u_j,p)=d(v_j,q)$, is the set of cycles with largest-indexed edge $e_j$ composed of $e_j$, $(p,q)$, and shortest path pairs from $u_j$ to $p$ and from $v_j$ to $q$. See Figure~\ref{fig:mod_vis_fams}, right.
	\end{itemize}
	\label{def:mod_vis_fams}
\end{definition}

\begin{figure}[t]
	\centering
	\includegraphics{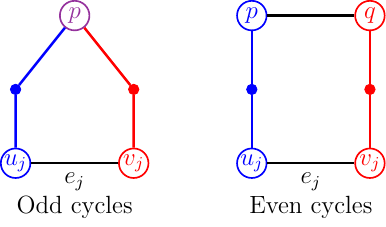}
	\caption{(Left) A cycle of length 5 whose descriptors are the largest-indexed edge $e_j\in E(G)\setminus E(T)$ for some spanning tree $T$ and the node $p$ antipodal to $e_j$. 
	(Right) A cycle of length 6 whose descriptors are its largest-indexed edge $e_j=(u,v)$ and antipodal edge $(p,q)$. 
	Nodes are colored red or blue if they are closer to $u$ or $v$ respectively. 
	The left node $p$ is purple because it is equidistant to $u$ and $v$.}
	\label{fig:mod_vis_fams}
\end{figure}

This has the following benefits:
\begin{itemize}
	\item \textbf{Node labeling.} The shortest paths for the V'-families are computed using breadth-first search starting at the root edge $e_j=(u_j,v_j)\in E(G)\setminus E(T)$. By rooting at an edge, we may recursively label whether a node is closer to $u_j$ or $v_j$. 
	See the blue and red nodes in Figure~\ref{fig:mod_vis_fams}.
	
	\item \textbf{Computation time.} 
	Repeatedly traversing the graph to construct V'-families is the primary bottleneck of our algorithm.
	To compute the V'-families, breadth-first search is run once per root edge in $E(G)\setminus E(T)$, the number of which is equal to the dimension of the cycle space $\nu$. 
	For the original Vismara families, breadth-first search is run once per node. 
	In our simulations of hydrocarbon pyrolysis, the dimension of the cycle space ranges from $0.01 n \le \nu \le 0.2 n$. 
	This reduces the number of times we traverse the graph by a scaling factor that depends on initial conditions.
	 
	For diamond, with mean degree $\langle k\rangle=4$, the dimension of the cycle space is $\nu = m - n + 1 \sim n$ since $n\langle k\rangle = 2m$. Thus, even for reasonably dense materials, the number of times we traverse the graph does not increase. This result does not extend to arbitrarily dense graphs where the assumption $m=O(n)$ fails.
	
	\item \textbf{Simple descriptors.} 
	By modifying the original V-families, the descriptors of even cycles change from a node and connected triple to a pair of edges. Connected triples are harder to work with, e.g.,
	the number of edges $m$ is given whereas the number of connected triples as a function of the degree sequence $\{k_i\}$ is given by $\sum_i \binom{k_i}{2}$.
\end{itemize}

All relevant cycles belong to a V'-family, but not all V'-families contain relevant cycles.
The method for filtering out irrelevant V'-families is discussed later in Section~\ref{sub:DICE_TICE_Comp}. 

\subsubsection{An algorithm for computing V'-families}



\begin{algorithm*}
	\footnotesize
	\KwIn{$\bullet\ G$ = an undirected, unweighted graph,\newline
		$\bullet\ \calF$ = fundamental cycle basis of $G$ with cycles associated with edges ${e_j}\in E(G)\setminus E(T),$ 
		\newline\phantom{$\bullet\ \calF$ = }
		$j=1,\hdots,\nu$, where $T$ is a spanning tree of $G$}
	\KwOut{$\bullet$ V'-families represented by their descriptors,\newline
		$\bullet\ D_j=$ a directed, acyclic network for computing shortest paths to $e_{j}$ for $j=1,\hdots,\nu$,\newline
		$\bullet$ node properties {$\tt ancestor_j$}, {$\tt distance_j$}, {$\tt num\_paths_j$} for $j=1,\hdots,\nu$\;}
	
	\For{${e_j}=({\tt u_j,v_j})$ {\normalfont in} $E(G)\setminus E(T)$}{
	Initialize root nodes $\tt ancestor_j[u_j]\gets u_j, distance_j[u_j]\gets 0, num\_paths_j[u_j]\gets 1,$ and $\tt ancestor_j[v_j]\gets v_j, distance_j[v_j]\gets 0, num\_paths_j[v_j]\gets 1$\;
	For each remaining node $\tt x$, mark $\tt observed[x]\gets False$ and $\tt valid[x]\gets True$\;
	Initialize (FIFO) $\tt queue\gets [u_j,v_j]$\;
	\For{{\tt p} {\normalfont in} $\tt queue$}{
		\# {\tt p} is a descriptor for an odd V'-family if it has two ancestors $\tt u_j$ and $\tt v_j$\;
		\uIf{$\tt ancestor_j[p]=\{u_j,v_j\}$ {\normalfont\textbf{and}} $\tt valid[p]$}{
			Count shortest paths from $\tt p$ to $\tt u_j$: $\tt n_1=\sum_{x, p\to x\in D_j}num\_paths_j[x]\cdot \textbf{1}[ancestor_j[x]=u_j]$\;
			Count shortest paths from $\tt p$ to $\tt v_j$: $\tt n_2=\sum_{x, p\to x\in D_j}num\_paths_j[x]\cdot \textbf{1}[ancestor_j[x]=v_j]$\;
			\uIf{$\tt n_1 n_2>0$}{
				Store the odd V'-family with descriptors $\tt e_{j}$ and $\tt p$ representing $\tt n_1 n_2$ cycles of length $\tt 2distance_j[p]{+}1$\;
			}
		}
		\For{$\tt q$ {\normalfont neighboring} $\tt p$ {\normalfont in} $G$}{
			\# {\tt p} is {\tt q}'s parent and {\tt q} is unobserved, initialize {\tt q}'s node properties\;
			\uIf{{\normalfont\textbf{not}} $\tt observed[q]$ }{
				Mark $\tt observed[q]\gets True$ and add $\tt q$ to $\tt queue$\;
				$\tt ancestor_j[q]\gets ancestor_j[p]$\;
				$\tt distance_j[q]\gets distance_j[p]+1$\;
				\uIf{${\tt(p,q)} = {\tt e_k}\in E(G)\setminus E(T)$  {\normalfont where} $k>j$}{
					$\tt num\_paths_j[q]\gets 0$\;
				}\uElse{
					$\tt num\_paths_j[q]\gets num\_paths_j[p]$\;
					Add edge $\tt q\to p$ to $D_j$\quad\# child to parent\;
				}
			}
			\# {\tt p} is {\tt q}'s parent and {\tt q} has been observed, update {\tt q}'s node properties\;
			\uElseIf{$\tt distance_j[q]=distance_j[p]+1$}{
				$\tt ancestor_j[q]\gets ancestor_j[q]\cup ancestor_j[p]$\;
				\uIf{$\tt ancestor_j[p] = \{u_j,v_j\}$}{
					Mark {\tt q} invalid\;
				}
				\uIf{$\tt num\_paths_j[p]>0$ {\normalfont\textbf{and} \textbf{not}} ${\tt(p,q)} = e_k\in E(G)\setminus E(T)$  {\normalfont where} $k>j$}{
					$\tt num\_paths_j[q] \gets num\_paths_j[q]+num\_paths_j[p]$\;
					Add edge $\tt q\to p$ to $D_j$\;
				}
			}
			\# ({\tt p,q}) is a descriptor of an even V'-family if {\tt p} and {\tt q} have different ancestors and are equidistant to $\tt e_j$, $\tt ancestor_j[p]=u_j$ to break symmetry\;
			\uElseIf{$\tt ancestor_j[p]=u_j$ {\normalfont\bf and} $\tt ancestor_j[q]=v_j$}{
				\uIf{{\normalfont \textbf{not}} ${\tt(p,q)} = e_k\in E(G)\setminus E(T)$  {\normalfont where} $k>j$}{
					\uIf{$\tt num\_paths_j[p]\cdot num\_paths_j[q]>0$}{
						Include the even C'-family with descriptors $\tt e_{j}$ and $\tt (p,q)$ representing $\tt num\_paths_j[p]\cdot num\_paths_j[q]$ cycles of length $\tt distance_j[p]+distance_j[q]+2$
					}
				}
			}
		}
		}
	}
	\caption{Compute V'-families }
	\label{alg:vis_fam}
\end{algorithm*}

\begin{figure}[t]
	\centering
	\includegraphics{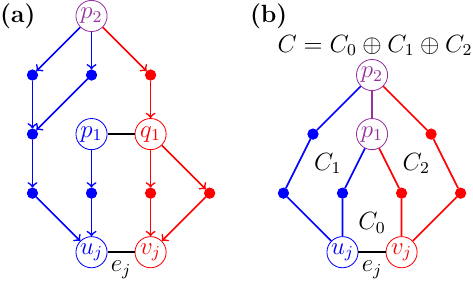}
	\caption{\textbf{(a)} The directed acyclic graph with edges pointing along shortest paths towards the root edge $e_j=(u_j,v_j)$. 
	There are two V'-families: one with descriptors $e_j$ and $(p_1,q_1)$ representing 2 cycles of length 6, and one with descriptors $e_j$ and $p_2$ representing 4 cycles of length 9.	
	\textbf{(b)} An odd V'-family with descriptors $e_j$ and $p_2$ where $p_2$ is a descendant of $p_1$ equidistant to $u_j$ and $v_j$. The paths from $p_2$ to $u_j$ and $v_j$ in the outer cycle $C$ are exchanged for paths through $p_1$ by adding $C_1$ and $C_2$ to $C$ respectively. This shows $C$ is not relevant because it is the sum of shorter cycles $C_0,C_1,C_2$.
	}
	\label{fig:ModVisFamInfo}
\end{figure}


We assume a spanning tree $T$ is given and the edges $E(G)\setminus E(T)$ are enumerated. 
Algorithm~\ref{alg:vis_fam} computes the V'-families by looping over each edge $e_j\in E(G)\setminus E(T)$.



Algorithm~\ref{alg:vis_fam} outputs V'-families represented by their descriptors from Definition~\ref{def:mod_vis_fams}.
The largest-indexed edge $e_j=(u_j,v_j)$ is given, and the $p$ and $(p,q)$ descriptors are computed. 
Nodes $p$ for odd V'-families are those that are equidistant to $u_j$ and $v_j$, obtained in lines 7-11 of Algorithm~\ref{alg:vis_fam}. 
Edges $(p,q)$ for even V'-families are such that $p$ is closer to $u_j$ than $v_j$, $q$ is closer to $v_j$ than $u_j$, and $d(p,u_j)=d(q,v_j)$, obtained in lines 32-35.

\begin{figure*}
	\centering
	\includegraphics[width=6.5in]{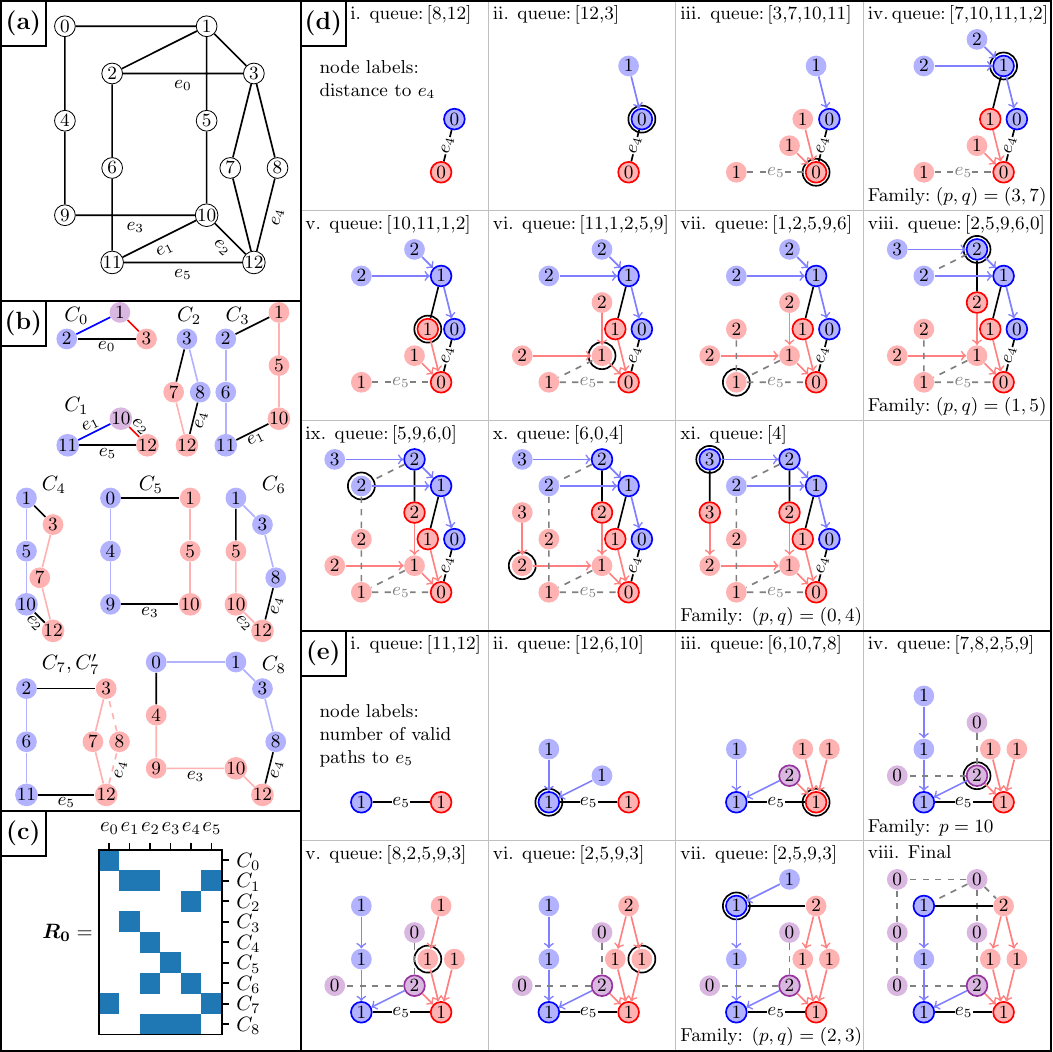}
	\caption{\textbf{(a)} A graph composed of a six-sided loop joined to a triangular prism with two paths $(3,7,12)$ and $(3,8,12)$ connecting $3$ and $12$. The enumerated edges $e_0,\hdots,e_5$ are used to construct a fundamental cycle basis. 
		\textbf{(b)} V'-families sorted by length 
		output by Algorithm~\ref{alg:vis_fam} iterated over $e_0,\hdots,e_5$. 
	The only V'-family with multiple cycles is $\{C_7,C_7'\}$.
	The secondary cycle $C_7'$ is plotted using dashed edges. 
	\textbf{(c)} Matrix of representative cycles from each V'-family as fundamental cycle basis vectors.
	\textbf{(d)} Depiction of Algorithm~\ref{alg:vis_fam} rooted at $e_4$. Each frame represents a step through the central for loop over nodes $\tt p$ (circled in black). Steps ${\tt p}=5,6,4$ are omitted.
	Node colors denote ancestory, node labels denote distance to $e_4$, and directed edges point along valid shortest paths to $e_4$.
	\textbf{(e)} Depiction of Algorithm~\ref{alg:vis_fam} rooted at $e_5$, where node labels denote the number of valid shortest paths to $e_5$.
	}
	\label{fig:VisFamEx}
\end{figure*}


The cycles in a V'-family are constructed by joining the pair of descriptors by shortest paths as in Figure~\ref{fig:mod_vis_fams}. Algorithm~\ref{alg:vis_fam} computes shortest paths from all nodes to $e_j = (u_j,v_j)$ using breadth-first search. Here, we define the distance to $e_j$ as
\begin{equation}
	d(x,e_j) = \min(d(x,u_j),d(x,v_j)),
\end{equation}
and a shortest path from $x$ to $e_j$ as a minimum length path from $x$ to either $u$ or $v$.
As a data structure, we store a directed acyclic graph whose edges point along shortest paths to $e_j$ -- see Figure~\ref{fig:ModVisFamInfo} (a).
We call a neighbor $q$ of node $p$ a child of $p$ if it is further from the root edge $e_j$. 
Edges in the directed acyclic graph point backwards from child to parent, which are assigned in lines 22 and 30.

We compute three additional node properties that further characterize the V'-families:



\begin{itemize}
		\item \textbf{Ancestor.} 
		A node's ancestor is the node $u_j$ or $v_j$ in the root edge $e_j$ closest to it. 
		If a node is equidistant to $u_j$ and $v_j$, then both $u_j$ and $v_j$ are its ancestors.
		The ancestors of a child node $q$ is the union of its parent's ancestors, computed in lines 16 and 25.
		Descendants of a node with both ancestors produce irrelevant cycles that can be decomposed into smaller cycles. See, Figure~\ref{fig:ModVisFamInfo} (b). We mark these descendants invalid in lines 26-27.
		
		\item \textbf{Distance.} 
		A node $q$ will first be observed by a neighbor $p$ lying on a shortest path to $e_j$. Thus, $q$ is one edge further from $e_j$ than its parent node $p$, as in line 15. This distance is used to compute the lengths of cycles in a family, as in lines 11 and 35. 
		\item \textbf{Number of paths.} 
		We count shortest paths from each node $q$ to $e_j$ as the sum of paths through its parents. See lines 19 and 26.
		This is used to compute the number of cycles in a V'-family, which is equal to the number of shortest path pairs to $u_j$ and $v_j$. 
		The size of a cycle family is computed in lines 8-11 and 35.
\end{itemize}

As a final note, in each loop of Algorithm~\ref{alg:vis_fam} we compute cycles with largest-indexed edge $e_j$. We reject any path that includes an edge $e_k\in E(G)\setminus E(T)$ with index $k>j$ as in lines 16-17, 23-24, and 29. We still traverse edges $e_k$, $k>j$, in Algorithm~\ref{alg:vis_fam} to accurately compute the distance and ancestor node properties.

As a post-processing step, we sample a representative cycle from each V'-family. We order these cycles, and the associated V'-families, by length, i.e., $|C_i|\le |C_j|$ for $i\le j$. See Figure~\ref{fig:VisFamEx} (b). 
We convert each cycle $C_i$ into a fundamental cycle basis vector using the non-tree edges $e_j\in E(G)\setminus E(T)$ that lie in $C_i$. 
Then, we construct a matrix whose rows are representative cycles from each V'-family as fundamental cycle basis vectors:
\begin{equation}
	\bm{R_0}(i,:) = (\bm{1}[e_1{\in} C_i], \bm{1}[e_2{\in} C_i], \hdots, \bm{1}[e_\nu{\in} C_i]),
	\label{eq:R0_def}
\end{equation}
where $C_i$ represents family $i$. 
See Figure~\ref{fig:VisFamEx} (c).

\subsubsection{The computational cost}

The complexity of algorithms in this section depend on the size of $\bm{R_0}$. The number of columns in $\bm{R_0}$ is equal to the dimension of the cycle space $\nu$. The number of rows in $\bm{R_0}$ is equal to the number of V'-families. $\bm{R_0}$ has at least $\nu$ linearly independent rows that can be used to construct an MCB.
The number of V'-families is bounded by
\begin{equation*}
	\begin{split}
		\nu \,&{\le}\, \textbf{\# of V'-families} \\&{\le}\,\big|E(G){\setminus}E(T)\big|\,(\#\text{V'-families per edge})\,{\le}\,\nu^{\mathrlap{2}}.
	\end{split}
\end{equation*}
We justify that Algorithm~\ref{alg:vis_fam} outputs at most $\nu$ cycles in each of its $\nu$ iterations below. This bound is tight, where the complete bipartite graph has $O(\nu^2)=O(n^4)$ cycle families. 

To bound the number of V'-families per iteration of Algorithm~\ref{alg:vis_fam}, consider a breadth-first spanning tree $T_j$ rooted at $e_j$. 
Note, $T_j$ is not the spanning tree $T$ used to construct a fundamental cycle basis.
$T_j$ is a subgraph of the directed acyclic graph $D_j$ used to construct paths to $e_j$, as in Figure~\ref{fig:ModVisFamInfo}, where each node has {out-degree} 1. 
The edge descriptor $(p,q)$ for an even V'-family lies outside $T_j$. 
The node descriptor $p$ for an odd V'-family has disjoint shortest paths to $u_j$ and $v_j$, so it has out-degree greater than or equal to 2 in $D_j$. 
Hence, one of these outgoing edges lies outside of $T_j$.
In both cases, we associate an edge outside of the tree $T_j$ with the descriptor $p$ or $(p,q)$, and so the number of V'-families with root edge $e_j$ is bounded by the number of non-tree edges $\nu$.

In summary, $\bm{R_0}$ is a $\{0,1\}^{O(\nu^2)\times\nu}$ matrix. 
Assuming dense matrix operations, the computational cost of the algorithms in the following sections is $O(\nu^4)$, i.e., the cost of na\"ively multiplying $\bm{R_0}$ by a $\nu\times\nu$ matrix.
However, we find $\bm{R_0}$ is not too large and typically very sparse. 
The number of rows in this matrix is typically closer to the dimension of the cycle space. 
In a loop cluster sampled from a carbon-rich simulation of hydrocarbon pyrolysis (initial composition $\rm C_{10}H_{16}$), the dimension of the cycle space is $\nu=1304$ and the number of cycle families is $1604$.
This observation agrees with other chemical systems where the number of relevant cycles is approximately equal to $\nu$, see Table 8 in Ref.~\cite{mayEfficient2014}.
Also, relevant cycles are typically short, bounding the number of fundamental cycle basis edges $e_j\in E(G)\setminus E(T)$ they have. 
As a result, their fundamental cycle basis vectors and the matrix $\bm{R_0}$ are very sparse. The $1604\times 1304$ $\bm{R_0}$ matrix described above has $7280$ non-zero entries. 
In the case where $\bm{R_0}$ is dense, breadth-first search may be used to compute matrix vector products $\bm{R_0}x$ in $O(\nu^2)$ time, as discussed in Section 5.2 in Ref.~\cite{mehlhornMinimum2010}, which can be used to reduce the complexity of matrix multiplication to $O(\nu^3)$.

Using our Python codes on a Ryzen 9 9900X CPU, the computation of the matrix $\bm{R_0}$ for the system mentioned above takes approximately 5 seconds.
For comparison, a sparse matrix-matrix product $\bm{R_0 W}$ is computed in less than 0.5 seconds. Indeed, Algorithm~\ref{alg:vis_fam} is the main bottleneck of our approach. The complexity of a single iteration of Algorithm~\ref{alg:vis_fam} is equal to the complexity of breadth-first search, i.e., $O(m)$. The total cost of Algorithm~\ref{alg:vis_fam} is $O(\nu m)$. 
For the $\rm C_{10}H_{16}$ system, $\nu m = 0.32 n^2$.
For diamond, as a proxy for dense material systems, $\nu m\approx 2n^2$. Therefore, it is reasonable to write the complexity as $O(n^2)$ for material networks. This procedure can be further accelerated by running Algorithm~\ref{alg:vis_fam} in parallel, because no communication between iterations over $e_j$ in Algorithm~\ref{alg:vis_fam} is required.

\subsection{Minimum cycle basis}
\label{sub:MCB_comp}


Our tests for relevance and for the {\tt pi} and {\tt sli} relations rely on expanding cycles in terms of an MCB. See Lemma~\ref{Lem:Relevant_MCB_Test}, Theorem~\ref{Thm:DirInt_Comp}, and Lemma~\ref{Lem:TrEx_Alg} respectively. 
Here, we construct an MCB.

There is a large body of literature dedicated to the computation of MCBs. See Refs.~\cite{kavithaCycle2009,mehlhornMinimum2010} for a review. 
We opt for the approach in Ref.~\cite{kavithaFaster2004}, which we review below. This approach is computationally efficient, easy to implement in our present setting, and can be used to expand cycles in terms of the MCB.

This approach relies on an orthogonality argument using the inner product of binary vectors in the GF(2) field:
\begin{equation}
	\langle A,B\rangle_\mathcal{F} := \bigoplus_i a_i b_i.
\end{equation} 
We emphasize that this inner product is defined in terms of a fundamental cycle basis $\calF$. In this context, a cycle $C$ should be interpreted as a binary vector in $\{0,1\}^\nu$ where entry $i$ indicates whether $C$ contains the $i$th edge outside of the spanning tree used to construct $\calF$.
See Figure~\ref{fig:FCB}~(b) for a reminder of this decomposition.


This construction of an MCB is iterative. Suppose cycles $B_1,\hdots, B_{j-1}$ belonging to an MCB are given. 
To obtain the next basis cycle, we utilize a \emph{witness vector} defined as a non-zero vector $S_j\in\{0,1\}^\nu$ satisfying the orthogonality condition
\begin{equation}
	\langle B_i, S_j\rangle_\calF = 0,\quad i<j.
	\label{eq:Wit_Orth}
\end{equation}
{We construct the vectors $S_j$ later in this section.}
The next cycle $B_j$ in the MCB is given by the optimization problem
\begin{equation}
	\min_{B\in\CycSp} |B|,\text{ such that }\langle{B,S_j}\rangle_\calF =1
	\label{eq:len_const_opt}
\end{equation}
where $\CycSp$ is the set of all cycles. 
We obtain this cycle from the first row of the matrix $\bm{R_0}$ from Eq.~\eqref{eq:R0_def} with inner product 1.

The last ingredient of this algorithm is to construct the witness vectors $S_j$. 
These vectors are computed dynamically, where the initial vectors $S_j,j=1,\hdots,\nu$, are the unit vector with value $1$ at coordinate $j$.
The witness vectors are updated each time a new basis vector $B_j$ is constructed. 
The future witness vectors are assumed to be orthogonal to the previous cycles, i.e., $\langle B_j,S_k\rangle=0$ for $j<k$.
At step $j$ if $\langle B_j,S_k\rangle=1$ for $k>j$ then we update $S_k$ as
\begin{align}
	S_k \gets S_k\oplus S_j.
\end{align}
This ensures $\langle B_j,S_k\rangle=0$ for $k>j$.

\begin{algorithm}[t]
	\KwIn{$\bullet\ G = $ simple, connected graph \newline$\phantom{\bullet\ G=}$ with $\nu$ independent cycles\newline
		$\bullet$ Matrix of cycles $\bm{R_0}$ from V'-\newline $\phantom{\bullet}$ families as fundamental cycle\newline $\phantom{\bullet}$ basis vectors
		}
	\KwOut{$\bullet\ $MCB $\calM = \{B_1,\hdots,B_\nu\}$\newline 
		$\bullet$ Witness vectors $\{S_1,\hdots,S_\nu\}$}
	Initialize $S_i$ to the standard unit vector in $\{0,1\}^\nu$ that is $1$ at coordinate $i$ and $0$ elsewhere\;
	\For{$j=1,\hdots,\nu$}{
		\# Basis cycle $B_j$ as first cycle $C_k$ from $\bm{R_0}$ with $\langle C_k,S_j\rangle_\calF=1$\;
		$B_j\gets$ row $k$ of $\bm{R_0}$ where $k$ is the first non-zero entry of $\bm{R_0} S_j$\;
		\For{$k=j+1,\hdots,\nu$}{
			\# Ensure $\langle C_j,S_k\rangle_\calF=0$ for $j<k$\;
			\uIf{$\langle C_j,S_k\rangle_\calF=1$}{
				$S_k\gets S_k\oplus S_j$\;
			}
		}
	}
	\caption{Compute MCB}
	\label{alg:MCB}
\end{algorithm}

To implement the approach above, we utilize the V'-families constructed in Section~\ref{sub:Vis_Fam}. This is shown in Algorithm~\ref{alg:MCB}. The cycles $B_j$ are drawn from the matrix $\bm{R_0}$ whose row vectors are the expansions of representative cycles from each family expanded with respect to our fundamental cycle basis $\calF$. If $C_k$ is the cycle represented by the $k$th row of $\bm{R_0}$, then the inner product $\langle C_k, S_j\rangle_\calF$ is the $k$th entry in the matrix-vector product $\bm{R_0} S_j$. Also, the cycles in $\bm{R_0}$ are sorted by length. Thus, the optimization problem in Eq.~\eqref{eq:len_const_opt} is solved by the row vector in $\bm{R_0}$ corresponding to the first non-zero entry in $\bm{R_0} S_j$, as in line 4 of Algorithm~\ref{alg:MCB}.

The main operations in Algorithm~\ref{alg:MCB} are matrix-vector products $\bm{R_0}S_j$, inner products $\langle C_j,S_k\rangle$, and vector additions $S_k\oplus S_j$.
If the entries are dense, the time complexity of Algorithm~\ref{alg:MCB} is $O(\nu^4)$. However, the actual cost of these operations is much less because $\bm{R_0}$ is very sparse. Additionally, the matrix-vector products $\bm{R_0} S_j$ only need to be computed up to the first non-zero entry.

\begin{lemma}
	The cycles $B_1,\hdots,B_\nu$ computed by Algorithm~\ref{alg:MCB} form an MCB.
\end{lemma}

\begin{proof}
	The correctness of Algorithm~\ref{alg:MCB} is proved via induction. Suppose $B_1,\hdots,B_{j-1}$ for $j\ge 1$ belong to an MCB $\calM_{j-1}$. The following identity is obtained by expanding $B_j$ in $\calM_{j-1}$ and taking the inner product with $S_j$
	\begin{equation}
		1 = \langle B_j,S_j\rangle_\calF = \enspace\bigoplus_{\mathclap{\calB\in\expandd{\calM_{j-1}}{B_j}}}\enspace \langle B,S_j\rangle_\calF.
	\end{equation}
	The cycles $B_1,\hdots,B_{j-1}$ are orthogonal to $S_j$. Hence, there exists $B\in\calM_{j-1}\setminus\{B_1,\hdots,B_{j-1}\}$ in the expansion of $B_j$ in $\calM_{j-1}$ such that $\langle B,S_j\rangle_\calF=1$. $|B|\le |B_j|$ because $B$ is in the expansion of $B_j$ in $\calM_{j-1}$. $|B_j|\le|B|$ because $B_j$ is a minimum length cycle such that $\langle B_j,S_j\rangle_\calF=1$. Therefore, $\calM_j=(\calM_{j-1}\setminus\{B\})\cup\{B_j\}$ is an MCB by Lemma~\ref{Lem:Relevant_MCB_Test}. By induction, $\calM_\nu=\{B_1,\hdots,B_\nu\}$ is an MCB.
\end{proof}

\subsection{Change of basis}
\label{sub:MCB_COB}

Next, we describe our procedure to compute the expansion of a cycle $C$ in terms of an MCB $\calM=\{B_1,\hdots,B_\nu\}$ computed by Algorithm~\ref{alg:MCB}.
We assume $C$ is given in vector form corresponding to its expansion with respect to the fundamental cycle basis $\calF$ used to construct $\calM$. 
Let the expansion of $C$ with respect to $\calM$ be given by
\begin{equation}
	C = b_1 B_1\oplus\hdots\oplus b_\nu B_\nu,\quad b_i\in\{0,1\},\label{eq:C_exp_MCB_0}
\end{equation}
where the coefficients $b_i$ are unknown.
We will compute these coefficients using the witness vectors $S_i$.

\begin{algorithm}[t]
	\KwIn{$\bullet$ Matrix of cycles $\bm{R_0}$ from V'-\newline $\phantom{\bullet}$ families as fundamental cycle\newline $\phantom{\bullet}$ basis vectors\newline
		$\bullet$ MCB $\calM=\{B_1{,}\hspace{1pt}.\hspace{1pt}.\hspace{1pt}.\hspace{1pt},B_\nu\}$ 
		\newline
		$\bullet$ Witness vectors $S_1,\hdots,S_\nu$}
	\KwOut{$\bullet$ Matrix of cycles $\bm{R_1}$ from\newline $\phantom{\bullet}$ V'-families as MCB vectors}
	Initialize 
	$\tilde{S}_i\gets S_i$ for $i=1,\hdots,\nu$\;
	\For{$k = \nu,\nu-1,\hdots,2$}{
		\# $\tilde{S}_k$ is known, update $\tilde{S_j}$ for $j<k$\;
		\For{$j = 1,\hdots,k-1$}{
			\uIf{$\langle B_k, S_j\rangle_\calF=1$}{
				$\tilde{S_j}\gets \tilde{S}_j\oplus\tilde{S}_k$\;
			}
		}
	}
	$\bm{W} \gets [\tilde{S}_1,\hdots,\tilde{S}_\nu]$\;
	$\bm{R_1} \gets \bm{R_0} \bm{W}$\;
	\caption{Compute MCB representation}
	\label{alg:COB}
\end{algorithm}

We start with the coefficient $b_\nu$. By construction, $S_\nu$ is orthogonal to every cycle in $\calM$ except $B_\nu$. Thus, 
\begin{equation}
	\langle C,S_\nu\rangle_\calF = b_\nu.
	\label{eq:bnu}
\end{equation}

For the coefficient $b_{\nu-1}$ we use the witness vector $S_{\nu-1}$ which is orthogonal to $B_1,\hdots,B_{\nu-2}$ by Eq.~\eqref{eq:Wit_Orth}. We take the inner product of Eq.~\eqref{eq:C_exp_MCB_0} with $S_{\nu-1}$ and substitute $\langle C,S_\nu\rangle_\calF$ for $b_\nu$ using Eq.~\eqref{eq:bnu}
\begin{equation}
	\langle C,S_{\nu-1}\rangle_\calF = b_{\nu-1} \oplus \langle C,S_\nu\rangle_\calF \langle B_\nu,S_{\nu-1}\rangle_\calF.
\end{equation}
Then, we rearrange and solve for $b_{\nu-1}$
\begin{equation}
	b_{\nu-1} = \langle C,S_{\nu-1}\rangle_\calF \oplus \langle C,S_\nu\rangle_\calF \langle B_\nu,S_{\nu-1}\rangle_\calF.\label{eq:bnum1}
\end{equation}
We move the inner product $\langle C_\nu,S_{\nu-1}\rangle_\calF$ into the right term of $\langle C,S_\nu\rangle_\calF$ by linearity. After doing so -- and combining similar terms -- we find that Eq.~\eqref{eq:bnum1} simplifies into a single inner product
\begin{align}
	b_{\nu-1} &= \langle C, \tilde{S}_{\nu-1}\rangle_\calF,\\ \tilde{S}_{\nu-1} &= S_{\nu-1}\oplus\langle B_\nu,S_{\nu-1}\rangle_\calF S_\nu.
\end{align}
We call $\tilde{S}_{\nu-1}$ a \emph{modified witness vector}.

In general, the coefficients $b_j$ can be computed via the inner product of $C$ with a modified witness $\tilde{S}_j$. This can be verified by induction where Eq.~\eqref{eq:bnu} is the base case. If we assume that $b_k = \langle C,\tilde{S}_k\rangle$ for $k>j$ then the inner product of Eq.~\eqref{eq:C_exp_MCB_0} with $S_j$ is
\begin{equation}
	\langle C,S_j\rangle_\calF = b_j\oplus\bigoplus_{{k=j+1}}^\nu \underbrace{\langle C,\tilde{S}_k\rangle_\calF}_{b_k} \langle B_k,S_j\rangle_\calF .
\end{equation}
Using the linearity of the inner product, we can combine terms to obtain $b_j$ and the associated modified witness $\tilde{S}_j$
\begin{align}
	b_j &= \langle C,\tilde{S}_j\rangle_\calF,\\ \tilde{S}_j&= S_j\oplus\bigoplus_{{k=j+1}}^\nu\langle B_k,S_j\rangle_\calF \tilde{S}_k.\label{eq:mod_wit}
\end{align}

Once the modified witness vectors $\tilde{S}_j$ are computed using Eq.~\eqref{eq:mod_wit}, it is useful to collect them as the columns of a matrix
\begin{equation}
	\bm{W} = [\tilde{S}_1, \tilde{S}_2,\hdots,\tilde{S}_\nu].
\end{equation}
Then if we left-multiply $\bm{W}$ by our cycle $C$ (as a vector expanded in terms of a fundamental cycle basis) we obtain its expansion in terms of $\calB$
\begin{equation}
	C \cdot\bm{W} = [\langle C{,}\hspace{1pt}\tilde{S}_1\rangle_\calF{,}\hspace{1pt}\langle C{,}\hspace{1pt}\tilde{S}_2\rangle_\calF{,}\hspace{1pt}.\hspace{1pt}.\hspace{1pt}.\hspace{1pt}{,}\hspace{1pt}\langle C{,}\hspace{1pt}\tilde{S}_\nu\rangle_\calF].
\end{equation}
Recall, in Section~\ref{sub:Vis_Fam} we computed the matrix $\bm{R}_0$ whose rows are representative cycles from the V'-families as fundamental cycle basis vectors. If we right-multiply this matrix by $\bm{W}$,
\begin{equation}
	\bm{R}_1 = \bm{R}_0 \bm{W},
	\label{eq:R1_comp}
\end{equation}
then we obtain the expansion of these cycles in terms of our MCB $\calM$.

Algorithm~\ref{alg:COB} computes the matrix $\bm{R_1}$. Lines 2-8 compute the modified witness vectors according to Eq.~\eqref{eq:mod_wit}. The inner products and vector sums in these lines are $O(\nu^3)$. In line 10, $\bm{R_1}$ is computed by Eq.~\eqref{eq:R1_comp}. 
Since $\bm{R_0}$ is a $O(\nu^2)\times \nu$ matrix, a simple implementation of this matrix multiplication is $O(\nu^4)$. In practice, Algorithm~\ref{alg:COB} is very efficient due to sparsity.

\subsection{{\tt pi} and {\tt sli} classes}
\label{sub:DICE_TICE_Comp}

Here we present a method for computing the {\tt pi} and {\tt sli} classes of an arbitrary unweighted, undirected graph. 
In Figure~\ref{fig:DICE_TICE_alg}~(a-d), we show the major steps for this procedure for the graph shown in Figure~\ref{fig:VisFamEx}~(a). 
We describe these steps below.

\begin{figure*}[t]
	\centering
	\includegraphics{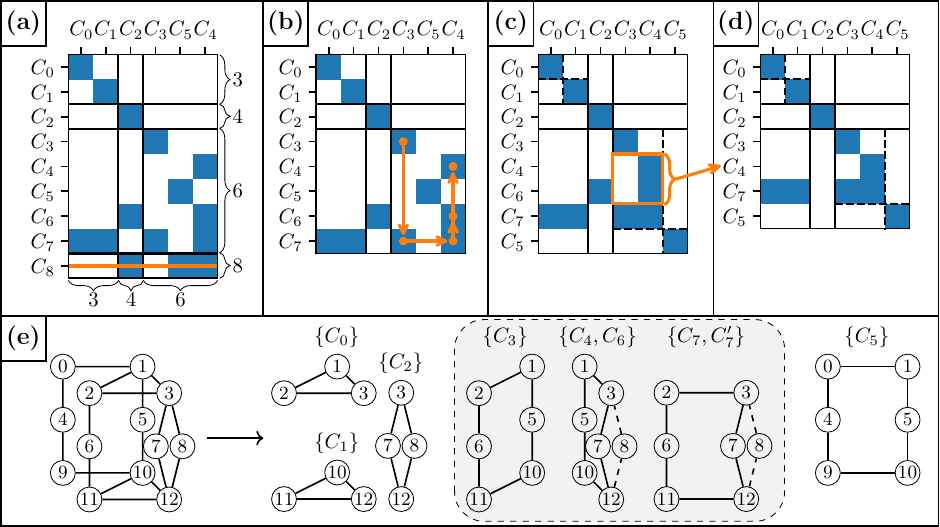}
	\caption{\textbf{(a)} Input matrix of representative cycles from each V'-family as MCB vectors. The blocks of this matrix separate cycles of differing lengths. \textbf{(b)} 
		Reduced matrix without the irrelevant cycle $C_8$ since it is the sum of shorter cycles. 
		The orange arrows depict the procedure for traversing rows and columns to compute the {\tt pi} classes. \textbf{(c)} The same matrix where the rows and columns are sorted by {\tt pi} classes. The block matrices for the {\tt pi} classes are separated by dashed lines. \textbf{(d)} Output matrix $\bm{R}$ with $C_4$ and $C_6$ merged into a single row representing a {\tt sli} class, since $C_4=C_6\oplus\smCyc$. \textbf{(e)} The graph (left) and its relevant cycles separated into {\tt sli} classes (right). The representative and secondary cycles in the {\tt sli} classes are drawn using solid and dashed edges. The gray box highlights the large {\tt pi} class $\{C_3,C_4,C_6,C_7,C_7'\}$ with multiple {\tt sli} classes.
	}
	\label{fig:DICE_TICE_alg}
\end{figure*}

\textbf{Inputs.} The inputs for our algorithm are (i) the V'-families output by Algorithm~\ref{alg:vis_fam} and (ii) the matrix $\bm{R_1}$ of representative cycles from each V'-family as MCB vectors from Algorithm~\ref{alg:COB}. 
We sort the rows and columns of $\bm{R_1}$ by increasing cycle length. 
We utilize the block structure of this matrix, where the row and column cycles have fixed length. See Figure~\ref{fig:DICE_TICE_alg} (a). 
This block matrix is lower triangular because the MCB representation of a cycle only includes equal or shorter length cycles -- see Lemma~\ref{Lem:Relevant_MCB_Test}.

\textbf{Remove irrelevant cycles.} By Lemma~\ref{Lem:Relevant_MCB_Test}, Statement 2, a cycle is relevant if and only if its MCB representation includes an equal length cycle. In Figure~\ref{fig:DICE_TICE_alg}~(a), the eight-sided cycle $C_8$ is not relevant as the sum of two six-sided cycles $C_4,C_5$ and a four-sided cycle $C_2$. Algorithmically, we loop through each row to find the largest cycle in each MCB vector to check if the row is relevant and remove it otherwise. 

\textbf{{\tt pi} classes.} 
From Lemma~\ref{Lem:dir_intchg_2}, a relevant cycle is {\tt pi}-interchangeable for the equal length cycles in its MCB representation. These relations are identified by the non-zero entries in the diagonal blocks of $\bm{R_2}$, e.g., $C_7$ is {\tt pi}-interchangeable for $C_3$ and $C_4$ in Figure~\ref{fig:DICE_TICE_alg}~(b).
From Theorem~\ref{Thm:DirInt_Comp}, the {\tt pi} classes are computed using the transitive closure of these relations.
This is done by traversing the rows and columns in a diagonal block of $\bm{R_2}$ using a queue -- see the orange arrows in Figure~\ref{fig:DICE_TICE_alg}~(b). 


We permute the rows and columns of $\bm{R_2}$ to group the {\tt pi} classes together.
We also permute rows within a {\tt pi} class by
\begin{enumerate}
	\item number of equal length cycles in the MCB vector, then by
	\item lexicographic ordering of the MCB vectors read \emph{right-to-left}.
\end{enumerate}
In Figure~\ref{fig:DICE_TICE_alg}~(c), the cycle $C_7$ follows $C_6$ because it has one more equal length cycle in its MCB representation. The cycle $C_6$ follows $C_4$ since $[0,1,0,1,0,0]$ follows $[0,1,0,0,0,0]$ in lexocographic ordering. We denote this sorted matrix by $\bm{R_3}$. 

\textbf{{\tt sli} classes.} 
Two cycles belong to the same {\tt sli} class if their MCB vectors are equal up to shorter cycles by Lemma~\ref{Lem:TrEx_Alg}. By using reverse lexicographic ordering, cycles in a {\tt sli} class will appear as consecutive rows of $\bm{R_3}$. See the cycles $C_4$ and $C_6$ in Figure~\ref{fig:DICE_TICE_alg}~(c). We merge cycles in a {\tt sli} class into a single row, and we keep track of the relevant V'-families they represent. We denote this final matrix as $\bm{R}$.

\textbf{Time-Complexity.} Most steps of this procedure can be done using a single pass through the working matrix $\bm{R_i}$. This has time-complexity $O(\text{nnz}(\bm{R_i}))=O(\text{nnz}(\bm{R_1}))$ where $\text{nnz}(\bm{X})$ is the number of nonzero entries in $\bm{X}$. 
The largest exception to this is sorting the rows of each {\tt pi} class using reverse lexicographic ordering. 
Using dense operations, this has $O(\nu^3\log\nu)$ time-complexity, where $\bm{R_2}$ has $O(\nu^2)$ rows and comparing two rows takes $O(\nu)$ time. Using sparse operations, this has $O(\text{nnz}(\bm{R_1})\log\nu)$ time-complexity, where the time to compare rows reduces to the average number of nonzero entries per row.
For our MD simulations, this procedure is an order of magnitude faster than computing the relevant cycle families and an MCB.

\textbf{Output.} The output relevant cycle matrix $\bm{R}$ can be interpreted as follows. Each row of this matrix is a representative cycle for a {\tt sli} class as an MCB vector. The choice of representative cycle only changes these MCB vectors by shorter cycles. Moreover, the entries in the diagonal blocks of $\bm{R}$ do not depend on the choice of representative cycles.

The {\tt pi} classes form diagonal blocks in $\bm{R}$. The rows in these blocks are sorted so that the cycles belonging to the MCB appear first. Because of this, the format of each block is the identity matrix followed by additional rows $[I,X']'$. The polyhedra are obtained via these additional rows $X$. The number of polyhedra obtained is equal to the dimension of the (left) null space of $\bm{R}$, given by the number of rows minus the number of columns in $\bm{R}$. 

For instance, from the row corresponding to $C_7$ in Figure~\ref{fig:DICE_TICE_alg}~(c), we obtain the polyhedron
\begin{equation}
	C_7 = C_3\oplus C_4\oplus\smCyc.
\end{equation}
Here, $\{C_3,C_4,C_7\}$ are the large faces of the triangular prism. The smaller cycles used to construct the polyhedron depend on which representative cycles are sampled from each family. For the cycles given, the full polyhedron is $\{C_0,C_1,C_3,C_4,C_7\}$.
{Codes for computing the decomposition of this graph are given in the notebook {\tt cycle\_decomposition.ipynb} on Github~\cite{ruthCyclic2025}.}

If a {\tt pi} class has multiple polyhedra, then the polyhedra sampled depend on the MCB used to represent the columns of $\bm{R}$. For instance, for the graph shown in Figure~\ref{fig:polyhedron_pair}, this procedure may identify either two cubes or a cube and a large rectangular prism. 
Intuitively, the most desirable polyhedra are the smallest ones, i.e., the cubes in Figure~\ref{fig:polyhedron_pair}.
We leave the challenge of computing smallest polyhedra for future work.

\section{Theory and algorithms for sampling minimum cycle bases}
\label{sec:sampling}

\subsection{Uniform random sampling of minimum cycle bases}
\label{sub:rand_MCB}


\begin{figure}[t]
	\centering
	\includegraphics[]{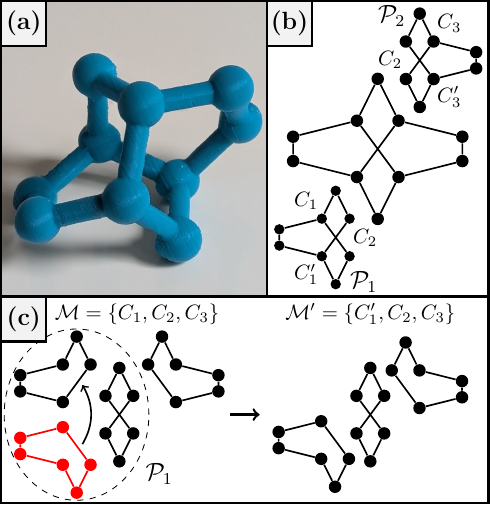}
	\caption{\textbf{(a)}~3D print of the carbon skeleton of twistane, a molecule with formula $\rm C_{10}H_{16}$. \textbf{(b)}~The carbon skeleton of twistane as a graph and its smallest polyhedra $\calP_1=\{C_2,C_3,C_3'\}$ and $\calP_2=\{C_1,C_1',C_2\}$. 
	\textbf{(c)}~Transition from the MCB $\calM=\{C_1,C_2,C_3\}$ to $\calM'=\{C_1',C_2,C_3\}$ by swapping $C_1'$ for $C_1$ using the polyhedron $\calP_1$. 
	}
	\label{fig:randMCB}
\end{figure}

Here, we introduce a method for sampling MCBs uniformly at random using a Markov Chain Monte Carlo approach.
We assume the reader is familiar with the basic theory of Markov chains~\cite{durrettEssentials2012}. 

\subsubsection{The basic algorithm}
A given MCB $\calM$ is randomly modified as follows.
\begin{enumerate}
	\item Randomly select a cycle $C'\in\RelCyc\setminus\calM$ to exchange with a cycle in $\calM$.
	\item Construct a polyhedron using $C'$ and $\calM$, i.e., $\calP = \{C'\}\cup\expand{\calM}{C'}$.
	\item Randomly select a cycle $C{\in}\calP$, $|C|{=}|C'|$, and return 
	$\calM'=(\calM\setminus\{C\})\cup\{C'\}$. 
	If $C=C'$, return the same basis $\calM=\calM'$.
\end{enumerate} 
See Figure~\ref{fig:randMCB}~(c) for an example of this procedure in a graph with multiple polyhedra. These steps are repeated towards equilibrium.

To validate this approach, we make the following observation.
\begin{prop}
	The above Markov chain is ergodic.
\end{prop}
\begin{proof}
	First, we need the above procedure to be aperiodic. This is true because there is nonzero holding probability, i.e., the probability we return to the same state $P(\calM'=\calM)$. These holding probabilities cause the MCB to leave its current state at a random (geometric) time, which breaks any otherwise periodic behavior in the system.
	
	Second, the system must be irreducible, i.e., any MCB must be reachable by any other. Consider two distinct MCBs $\calM_1$ and $\calM_2$. By Lemma~\ref{Lem:DICE_Irreducible}, we can exchange a cycle in $\calM_1$ for another in $\calM_2$ to obtain a new MCB $\calM_1{\to}\calM_1'$ that shares an additional cycle with $\calM_2$. We may repeat this procedure to obtain $\calM_2$ from $\calM_1$.
\end{proof}

By ergodicity, this Markov chain approaches a stationary distribution over the set of MCBs. This stationary distribution is the uniform distribution because the Markov chain has symmetric transition probabilities
\begin{equation}
	P(\calM{\to}\calM'){=}P(\calM'{\to}\calM){=}\frac{1}{|\RelCyc{\setminus}\calM|}\frac{1}{|\calP|}.
\end{equation}
The first term $1/|\RelCyc{\setminus}\calM|=1/(|\RelCyc|{-}\nu)$ is the constant probability with which we choose a given cycle $C'\in\RelCyc{\setminus}\calM$ in step 1 of the procedure above. 
The second term is the probability we select the cycle $C\in\calP$ to exchange for $C'$. 
The same polyhedron $\calP$ is used in both directions $\calM\to\calM'$ and $\calM'\to\calM$.

\begin{algorithm}[t]
	\KwIn{$\calM=$ input MCB\newline
		$\bm{R}=$ representative cycles for\newline $\phantom{\bm{R}=}$ {\tt sli} classes as vectors in $\calM$} 
	\KwOut{$\calM_r=$ random MCB}
	\For{$i=1,\hdots,N_{\rm steps}$}{
		Randomly sample $C'\notin\calM$ from $\bm{R}$\;
		Using the row for $C'$ in $\bm{R}$ construct the polyhedron $\calP=\{C'\}\cup\expand{\calM}{C'}$\;
		Sample $C\in\calP, |C|=|C'|,$ with probability proportional to $1/|\calS(C)|$\;
		\uIf{$C\neq C'$}{
			$\calM\gets(\calM\setminus\{C\})\cup\{C'\}$\;
			\# Update $\bm{R}$ to new basis\;
			\For{Row cycle $C_i=\bm{R}[i,:]$}{
				\uIf{$C\in\expand{\calM}{C_i}$}{
					Set $\bm{R}[i,:]$ to $\expandd{\calM'}{C_i}=$  
					$ (\expandd{\calM}{C_i}\setminus\{C\})\symdiff\expandd{\calM'}{C}$\;
				}
			}
		}
	}
	
	$\calM_r\gets [\ ]$ \# empty list\;
	\For{$C\in\calM$}{
		Randomly sample a cycle from the {\tt sli} class $\calS(C)$ and add it to $\calM_r$\;
	}
	\caption{Random MCB}
	\label{alg:rand_MCB}
\end{algorithm}



\subsubsection{The efficient algorithm}

We simplify this procedure by drawing cycles from the matrix $\bm{R}$ computed in Section~\ref{sub:DICE_TICE_Comp}. 
Each cycle in $\bm{R}$ represents a {\tt sli} class. 
The MCB we will sample from $\bm{R}$ is used to select a set of {\tt sli} classes to draw cycles from. 
It is valuable to sample cycles from $\bm{R}$ since there is only a polynomial number of {\tt sli} classes. Cycles can be efficiently sampled from a {\tt sli} class directly, which we discuss at the end of this subsection. 


By Corollary~\ref{Cor:TrInt_Exchanges}, we can exchange pairs of cycles within a {\tt sli} class arbitrarily between MCBs. 
Thus, the total number of MCBs we can obtain using valid {\tt sli} classes $\calS_1,\hdots,\calS_\nu$ is $|\calS_1|\cdots|\calS_\nu|$. 
We want the stationary distribution to be proportional to the number of MCBs that $\calM$ drawn from $\bm{R}$ represents,
\begin{equation}
	\pi_\calM \propto \prod_{C\in\calM} |\calS(C)|
	\label{eq:rand_stat_dist}
\end{equation}
where $\calS(C)$ is the {\tt sli} class containing $C$.

To achieve the stationary distribution given by Eq.~\eqref{eq:rand_stat_dist}, we bias the transition probabilities to satisfy the detailed balance condition
\begin{equation}
	P(\calM\to\calM')\pi_\calM = P(\calM'\to\calM)\pi_{\calM'}.\label{eq:rand_det_bal}
\end{equation}
We re-arrange Eq.~\eqref{eq:rand_det_bal} to obtain the desired ratio between the forward and backward transition probabilities
\begin{equation}
	\frac{P(\calM\to\calM')}{P(\calM'\to\calM)} = \frac{\pi_{\calM'}}{\pi_\calM} = \frac{|\calS(C')|}{|\calS(C)|}
\end{equation}
where $\calM'=(\calM\setminus\{C\})\cup\{C'\}$. This is achieved by removing the cycle $C$ from our polyhedron with probability proportional to $1/|\calS(C)|$, as in Algorithm~\ref{alg:rand_MCB}, line 4. Then, the transition probabilities are given by
\begin{equation}
	\begin{split}
	&P(\calM\to\calM') \\&\quad= \frac{1}{\text{dim}(\text{coker}(\bm{R}))}\frac{1/|\calS(C)|}{\sum_{C'\in\calP}1/|\calS(C')|}
	\end{split}\label{eq:rand_trans_prob}
\end{equation}
where $\text{dim}(\text{coker}(\bm{R}))$ is the dimension of the left null space of $\bm{R}$, also the number of row cycles in $\bm{R}$ that do not belong to $\calM$.
Only the numerator $1/|\calS(C)|$ changes in Eq.~\eqref{eq:rand_trans_prob} between the forward ($\calM{\to}\calM'$) and reverse ($\calM'{\to}\calM$) exchange. The same polyhedron $\calP$ is used in both directions.


After performing the exchange $\calM\to\calM'$ we update the rows of $\bm{R}$ to be vectors in $\calM'$. We only need to update the representation of a cycle $C_i=\bm{R}[i,:]$ if its expansion $\expandd{\calM}{C_i}$ includes the cycle $C$ swapped for $C'$. If $C\in\expandd{\calM}{C_i}$, we replace $C$ with its expansion in $\calM'$
\begin{equation}
	\expandd{\calM'}{C_i} = (\expandd{\calM}{C_i}\setminus\{C\})\symdiff\expandd{\calM'}{C}.
\end{equation}

\begin{figure}[t]
	\centering
	\includegraphics{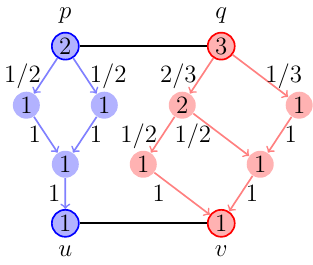}
	\caption{
		Procedure used to randomly sample a cycle from an even cycle family. 
		A cycle is sampled by joining two random paths $p\to u$ and $q\to v$ to the descriptors $(p,q)$ and $(u,v)$. 
		Node labels represent the number of valid shortest paths from a node to $(u,v)$. 
		Edge labels represent the probability a node is selected to construct a shortest path.
}
	\label{fig:rand_fam_backtrack}
\end{figure}

Last, we describe how cycles are sampled from each {\tt sli} class. A {\tt sli} class is the union of disjoint V'-families output by Algorithm~\ref{alg:vis_fam}. 
We choose a V'-family with probability proportional to the number of cycles it has. 
A cycle is sampled from the V'-family by backtracking from either the node $p$ or edge $(p,q)$ to the root edge $(u,v)$. 
At each step, the probability we step through a given node is proportional to the number of paths it has -- see Figure~\ref{fig:rand_fam_backtrack}. The final MCB is obtained by sampling a random cycle from each {\tt sli} class. The whole procedure is summarized in Algorithm~\ref{alg:rand_MCB}.







\subsection{Pairwise intersections in a minimum cycle basis}
\label{sub:pair_intersect}

\begin{figure}[t]
	\centering
	\includegraphics[trim={0 .6cm 0 .6cm},clip]{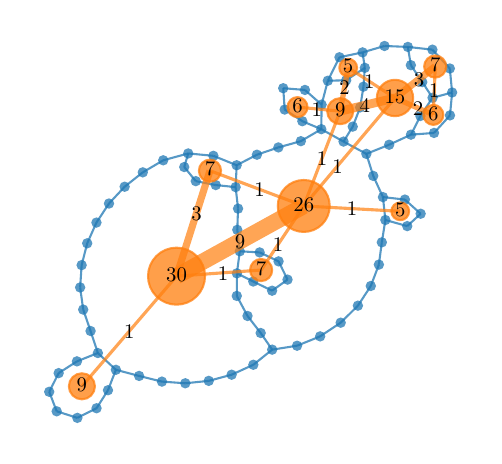}
	\caption{Ring cluster and its dual graph from a $\rm C_8H_{18}$ simulation at 4000K and 40.5GPa. Blue nodes and edges denote carbon atoms and bonds that belong to the cluster. Orange nodes represent carbon rings in the dual graph. Orange edges represent a pair of rings that are joined by a path. The weights of the orange nodes and edges are equal to the number of edges in the corresponding ring or path.}
	\label{fig:dual_graph}
\end{figure}

The overarching goal of this manuscript is to develop mathematical theory that is capable of describing ring clusters in hydrocarbon pyrolysis. So far, we have discussed how to obtain the relevant cycles and how to partition these cycles into a polynomial number of {\tt pi} and {\tt sli} classes. These serve as the building blocks of our ring clusters. In this section, we present new theory for how these cycles intersect. 
This consists of two main results, the proofs of which are contained in Appendix~\ref{app:intersect_proofs}

\begin{enumerate}
	\item We characterize the general format for how pairs of cycles in an MCB may intersect. See Theorem~\ref{Thm:MCB_Intersect}. 
	\item We introduce a method for modifying a MCB $\calM\to\calM'$ such that pairs of cycles in the new MCB $C_i,C_j\in\calM'$ intersect over a single path $P_{i,j}=C_i\cap C_j$. See Theorem~\ref{Thm:MCB_Pair_Swap}. This makes it possible to represent the loop cluster as a graph of cycles, also referred to as a \emph{dual graph}. The nodes of the dual graph are the cycles $\calM'=\{C_1,\hdots,C_\nu\}$, and the edges are the non-empty paths of intersections $\{P_{i,j}\}$. See Figure~\ref{fig:dual_graph}. 
\end{enumerate}


\begin{figure}
	\centering
	\includegraphics{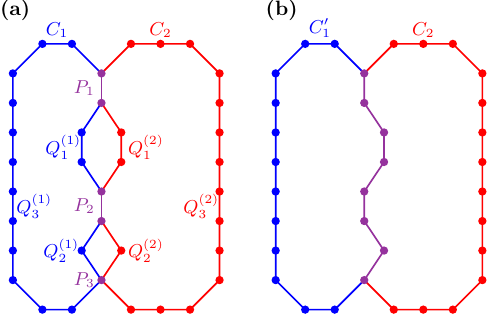}
	\caption{
		\textbf{(a)} Two cycles $C_1,C_2$ that intersect over three paths. Blue nodes belong to $C_1$, red nodes to $C_2$, and purple nodes to both.
		If we traverse $C_1$ clockwise and $C_2$ counterclockwise, then the intersection paths appear in the same order $P_1\,{\to}\,P_2\,{\to}\,P_3$ and with the same orientation.
		The shorter separating paths have equal lengths $\big|Q_1^{(1)}\big|=\big|Q_1^{(2)}\big|$ and $\big|Q_2^{(1)}\big|=\big|Q_2^{(2)}\big|$. \textbf{(b)}~A cycle $C_1'$ that can be exchanged for $C_1$ such that $C_1'$ and $C_2$ intersect over a single path. $C_1'$ is constructed from $C_1$ by swapping the paths $Q_1^{(1)}$ for $Q_1^{(2)}$ and $Q_2^{(1)}$ for $Q_2^{(2)}$}
		\label{fig:Intersect_Thm}
\end{figure}


\begin{theorem}
	Let $\calM$ be an MCB and $C_1,C_2\in\calM$, $|C_1|\le|C_2|$, such that $C_1$ and $C_2$ share at least one vertex.
	There exist circulations of $C_1$ and $C_2$ such that: 
	\begin{enumerate}
		\item $C_1$ and $C_2$ intersect over $k$ paths $P_1,P_2,\ldots,P_k$ that appear with the same order and orientation in both cycles. 
		
		\item Let $Q_{i}^{(j)}$ for $i\,{=}\,1{,}...{,}\,k$, $j\,{=}\,1,2,$ be the path in $C_j$ joining $P_i$ to $P_{i+1}$, where $P_{k+1}\,{\coloneq}\, P_1$. These paths are equal in length between $C_1$ and $C_2$, possibly except for the largest path pair which is ordered last:
		\begin{align}
			|Q^{(1)}_i| &\!=\! |Q^{(2)}_i|, &i=1,\!...,k-1,
			\label{eq:sm_path_equal}\\
			|Q^{(j)}_i| &\!\le\! |C_1|/2,\! &i{=}1,\!...,\!k{-}1, j{=}1,\!2,
			\label{eq:sm_path_len}\\
			|Q^{(2)}_k| &\!=\! \mathrlap{|Q^{(1)}_k|+|C_2|-|C_1|,}
			\label{eq:lg_path_equal}\\
			|Q^{(j)}_k| &\!\ge\! |C_j|/2,\! &j=1,2.
			\label{eq:lg_path_len}
		\end{align}
	\end{enumerate}
	\label{Thm:MCB_Intersect}
\end{theorem}

\begin{figure}[t]
	\centering
	\includegraphics{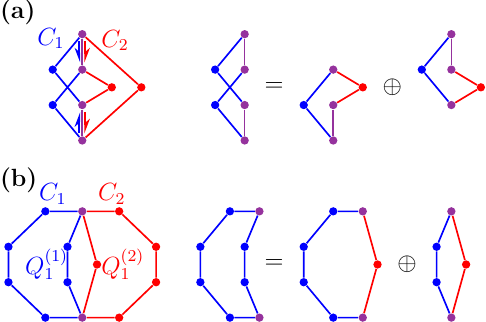}
	\caption{Two pairs of cycles that do not satisfy the format in Theorem~\ref{Thm:MCB_Intersect}. In both cases, the left cycle $C_1$ is not relevant because it is the sum of shorter cycles. 
	\textbf{(a)} Two cycles $C_1,C_2$ that intersect over two paths. 
	The red and blue arrows depict an example pair of orientations of $C_1$ and $C_2$. 
	There are no orientations of $C_1$ and $C_2$ where the intersection paths have the same orientation in $C_1$ and $C_2$. 
	Hence, $C_1$ appears twisted next to $C_2$.
	\textbf{(b)}~Two cycles $C_1,C_2$ that are separated by shorter paths $Q_1^{(1)}$ and $Q_1^{(2)}$ with different lengths $|Q_1^{(1)}|\ne |Q_1^{(2)}|$.}
	\label{fig:non-rel_overlap}
\end{figure}

In summary, Theorem~\ref{Thm:MCB_Intersect} states that a pair of cycles belonging to some MCB may intersect over a long path broken up by shorter cycles. 
See Figure~\ref{fig:Intersect_Thm} (a). 
In general, we may define the intersection of two cycles using a series of paths, which can be stored as a list of lists. 
We emphasize, that Theorem~\ref{Thm:MCB_Intersect} does not apply to arbitrary pairs of cycles. 
See Figure~\ref{fig:non-rel_overlap}.

Our theory allows us to go further. Specifically, we can stitch together a pair of cycles in an MCB so that they intersect over a single path.


%

\begin{theorem}
	Let $\calM$ be an MCB and let cycles $C_1,C_2\in\calM$ share at least two nodes. There exists a cycle $C_1'$ such that 
	\begin{enumerate}
		\item $\calM'=(\calM\setminus\{C_1\})\cup\{C_1'\}$ is an MCB and,
		\item $C_1'$ and $C_2$ intersect over a single path.
	\end{enumerate}
	\label{Thm:MCB_Pair_Swap}
\end{theorem}

\noindent The cycle $C_1'$ is obtained by exchanging paths in $C_1$ for $C_2$. See Figure~\ref{fig:Intersect_Thm}~(b) for an example of this procedure.

It is desirable that $C_1$ and $C_2$ intersect over a single path because a path can be stored as a single list. We propose to postprocess an MCB so that all pairs of cycles share at most one path, as in Algorithm~\ref{alg:merge_MCB}. Then, the cluster of cycles may be characterized using a set of cycles and the paths that connect them. This is the dual graph representation shown in Figure~\ref{fig:dual_graph}. We remark that the dual graph as shown in Figure~\ref{fig:dual_graph} is not lossless because it only includes path lengths, not the positions of these paths in the cycles.


\begin{algorithm}[t]
	\KwIn{MCB $\calM$}
	\KwOut{MCB $\calM'$, where all pairs $C_1,C_2\in\calM'$ intersect over a single path}
	Compute the intersection paths for all pairs of cycles $C_i,C_j\in\calM$\;
	\For{$i=1,\hdots,N_{\max}$}{
		Randomly sample a pair $C_i,C_j\in\calM$ that intersect over multiple paths\;
		Exchange $\calM\gets(\calM\setminus\{C_i\})\cup\{C_i'\}$ by Theorem~\ref{Thm:MCB_Pair_Swap} where $C_i',C_j$ intersect over a single path\;
		Recompute intersection paths between $C_i'$ and other cycles\;
	}
	\caption{Postprocess MCB}
	\label{alg:merge_MCB}
\end{algorithm}

At this point, we cannot guarantee that Algorithm~\ref{alg:merge_MCB} always outputs an MCB with every pair of cycles intersecting at most over a single path. It always succeeds on our hydrocarbon data in 10-20 iterations. We leave the investigation of this issue for future work.




\section{Applications to hydrocarbon pyrolysis systems}
\label{sec:results}

\begin{figure*}[t]
	\includegraphics[width=\linewidth]{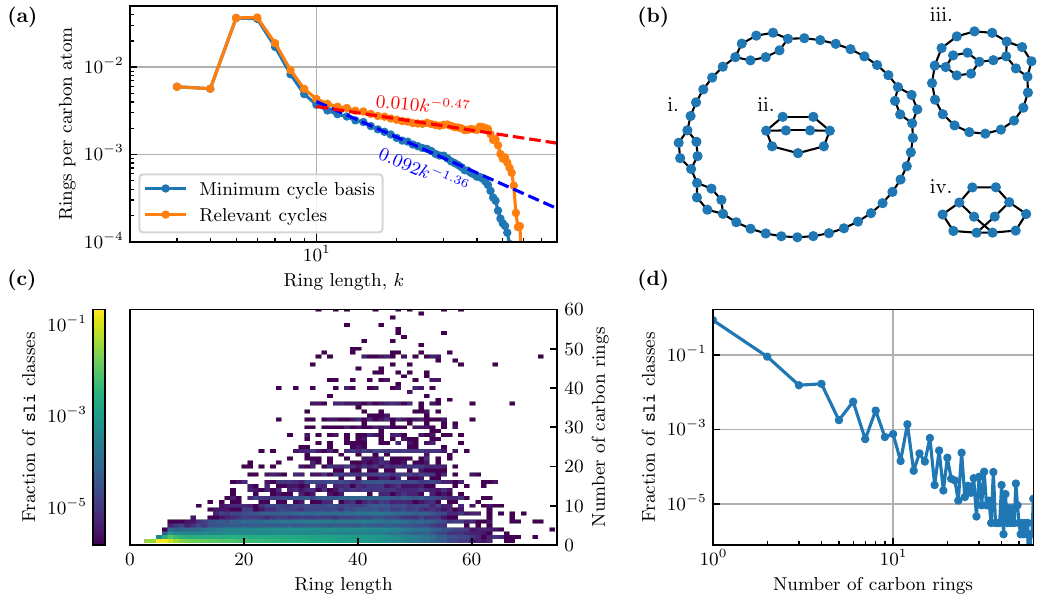}
	\caption{Analysis of cycles for a system initialized to adamantane with formula $\rm C_{10}H_{16}$. \textbf{(a)}~Rate distribution of carbon rings obtained via an MCB (blue) and the relevant cycles (orange). The dashed lines represent a power law fit of these distributions for rings of length $15\le k \le 40$. \textbf{(b)} Four examples of {\tt sli} classes: (i) 16 rings of length 45, (ii) 2 rings of length 7, (iii) 4 rings of length 19, and (iv) 3 rings of length 8. \textbf{(c)} Joint distribution of ring lengths and number of rings in {\tt sli} classes. \textbf{(d)} Distribution of the number of rings in {\tt sli} class.}
	\label{fig:TICE_Res}
\end{figure*}

\subsection{Generating data} 
Here, we apply the theory for the {\tt pi} and {\tt sli} relations to analyze loop clusters in hydrocarbon pyrolysis. 
This data comes from a single simulation using a ReaxFF force field~\cite{ashrafExtension2017} in LAMMPS~\cite{plimptonFast1995,thompsonLAMMPS2022,aktulgaParallel2012}. 
The simulation was initialized with 800 adamantane molecules, $\rm{C_{10}H_{16}}$, resulting in $N_{\rm C}=8000$ carbon and $N_{\rm H}=12800$ hydrogen atoms. 
The system was ramped to a temperature of 4000K and pressure of 40.5~GPa. 
This was followed by an NPT simulation for a total of 1.2 ns. 
Bonds were sampled using a length and duration criterion as discussed in previous works~\cite{ruthCyclic2025,dufour-decieuxPredicting2023}. 
The accuracy of the ReaxFF force field~\cite{ashrafExtension2017} at these thermodynamic conditions is questionable, but it provides sufficiently rich data to investigate our ring statistics. 

We started sampling after 0.6 ns to allow the system to reach equilibrium. 
Every 1.2 ps (500 samples total) we decompose the loop cluster into {\tt pi} and {\tt sli} classes as discussed in Section~\ref{sub:DICE_TICE_Comp}. 
From this decomposition we collect the following:
\begin{itemize}
	\item A random MCB.
	\item The length and number of rings within each {\tt sli} class.
	\item The length, number of rings, and number of independent polyhedra within each {\tt pi} class.
\end{itemize}

%
%

\subsection{Relevant cycles versus minimum cycle basis} 
\subsubsection{$\rm C_{10}H_{16}$ simulation data}
In Section~\ref{sub:rand_MCB}, we introduced Algorithm~\ref{alg:rand_MCB} for sampling MCBs uniformly at random. 
To highlight the value of this approach, we compare counts of large carbon rings from random MCBs and from the relevant cycles.
Specifically, we focus on the ring rate distribution, or the number of rings per carbon atom as a distribution over ring length -- see Figure~\ref{fig:TICE_Res}~(a).
This statistic has also been considered in computational studies of amorphous carbon~\cite{dernovMapping2025,deringerMachine2017}. 

Most rings are short with five or six carbon atoms. 
The amorphous structure of these hydrocarbon networks results in many large rings, as shown by the heavy tail of the ring rate distribution.
This tail approximately fits to a power law distribution $Ck^{-\alpha}$ with a cutoff.

Importantly, the power law constant $\alpha$ varies with the method of sampling loops.
When the loops are sampled from an MCB we obtain $\alpha\approx1.36$. 
In contrast, sampling loops from the relevant cycles results in a power of $\alpha\approx0.47$.

Without the cutoff, the power law tail $k^{-0.47}$ of the loop rate distribution for the relevant cycles is not summable.
We observe this cutoff increases with respect to both system size and the carbon-to-hydrogen ratio. 
Larger simulations are needed to see if this cutoff continues to increase with system size. 
If so, then the power law tail $k^{-0.47}$ will cause the relevant cycles to grow asymptotically faster than the dimension of the cycle space $\nu$:
\begin{equation}
	\frac{|\RelCyc|}{\nu}\to\infty \quad\text{as}\quad N_{\rm C}\to\infty.
\end{equation}
For our simulation, the time-averaged number of relevant cycles is $|\RelCyc|=1736$ and the average dimension of the cycle space is $\nu=1311$.

\subsubsection{Random geometric graphs}

\begin{figure}[t]
	\centering
	\includegraphics{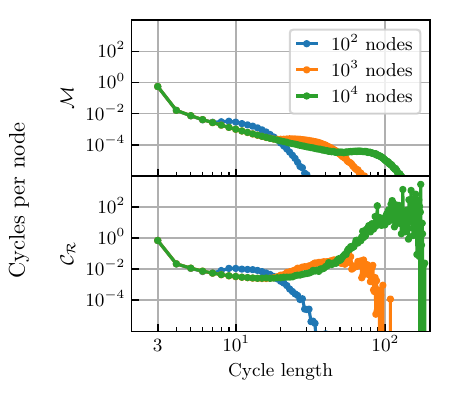}
	\caption{Analysis of the cycle rate distribution in random geometric graphs as a function of graph size. Cycles are sampled using an MCB $\calM$ (top) and the relevant cycles $\RelCyc$ (bottom). Nodes are assigned random positions in the unit cube $[0,1]^3$ and the cutoff radius for edges is chosen so that the expected degree is 3. The random geometric graph with $10^2$, $10^3$, and $10^4$ nodes is sampled $10^5$, $10^4$ and $10^3$ times respectively.}
	\label{fig:CycRateDist_RGG}
\end{figure}

To further investigate the scaling of $|\RelCyc|$ with system size, we consider random geometric graphs~\cite{penroseRandom2003}. A random geometric graph is a graph with $n$ nodes randomly positioned in the $d$-dimensional cube $[0,1]^d$. A pair of nodes are joined by an edge if their distance is less than a given fixed radius. The cycle statistics from these random graphs share some qualitative characteristics with our simulation because they have a spatial component. We choose a dimension of $d=3$ and periodic boundary conditions for sampling edges. 

We choose the radius so that the average degree is 3. For comparison, the average degree of the carbon skeleton in our ReaxFF simulation is 2.25. We choose a higher mean degree of 3 because it results in more large loops. This reduces the number of samples needed to observe enough large loops.

As in molecular dynamics simulation, the tail of the cycle rate distribution of MCBs in random geometric graphs roughly follows a power law. The cutoff for this power law distribution increases with system size. See Figure~\ref{fig:CycRateDist_RGG}, top. For the random geometric graph with $10^2$ nodes, the cycle rate distribution for the relevant cycles has the same shape as the cycle rate distribution given by MCBs. See the blue curve in Figure~\ref{fig:CycRateDist_RGG}, bottom. However, as system size increases, the number of large relevant cycles increases much faster than the number of large cycles in MCBs. For the random geometric graphs with $10^4$ nodes, we observed a shocking $9900$ relevant cycles per node. See the green curve in Figure~\ref{fig:CycRateDist_RGG}, bottom.

\subsubsection{A potential explosion of the relevant cycle count}

It is unclear if the ring rate distribution in molecular dynamics will follow this same trend with system size.
Given this ambiguity, we argue that the relevant cycles should not be used to count rings in large amorphous chemical systems. Our studies of random geometric graphs suggest the number of large rings may explode at scales achievable using multiple GPUs in a molecular dynamics simulation. Beyond the relevant cycles, many other unique sets of rings in the chemistry literature are also worst case exponential in size. See, e.g., the review~\cite{bergerCounterexamples2004}. This is not an issue for our approach, where the size of an MCB grows linearly with the number of edges.

\subsection{Sources of additional cycles} 
The {\tt pi} and {\tt sli} classes can be used to determine where the extra relevant cycles are coming from. For the given initial conditions, we find that the {\tt sli} classes are the most common. Over 99\% of {\tt pi} classes reduce to a single {\tt sli} class. Most {\tt sli} classes consist of a large ring with nested smaller rings -- see Figure~\ref{fig:TICE_Res}~(b.i.). In some cases, there is only a minor distinction between large and small rings, e.g., the two larger 7-rings in Figure~\ref{fig:TICE_Res}~(b.ii.) are separated by a smaller 6-ring. In other cases, the structure of a {\tt sli} class is more complex -- see Figure~\ref{fig:TICE_Res}~(b.iii-iv.).

In Figure~\ref{fig:TICE_Res}~(c), we show the joint distribution of the lengths and number of rings within a {\tt sli} class. 
From this distribution, we observe most (roughly 85\%) of the {\tt sli} classes consist of a single cycle, and the average size of the {\tt sli} classes increases with ring length. 
In Figure~\ref{fig:TICE_Res}~(d), we show the size distribution of the {\tt sli} classes, which roughly follows a power law distribution. 
Interestingly, the oscillations in this size distribution are not noise, and they can be seen as horizontal bands in the joint length size distribution in Figure~\ref{fig:TICE_Res}~(c). 
The number of cycles in a {\tt sli} class tends more often to a composite than a prime number.

In summary, the relevant cycles have an increased number of large cycles forming {\tt sli} classes. At the same time, these large cycles are highly dependent, e.g., the cycles in the {\tt sli} classes in Figure~\ref{fig:TICE_Res}~(b.i-iii.) share many edges. 
Our codes can be used to isolate and visualize these {\tt sli} classes.

\begin{figure*}[t]
	\centering
	\includegraphics[width=\linewidth]{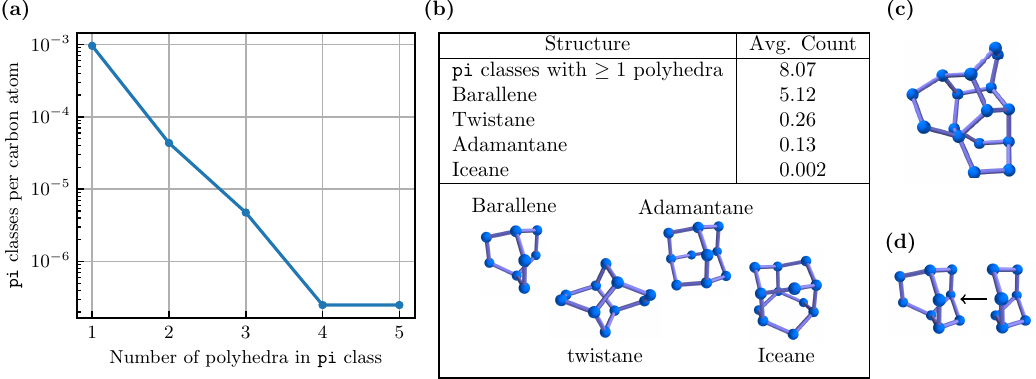}
	\caption{Analysis of polyhedra for a system initialized to adamantane with formula $\rm C_{10}H_{16}$. \textbf{(a)}~Size distribution of the {\tt pi} classes given by the number of polyhedra in the class. Only {\tt pi} classes with at least one polyhedron are considered, otherwise, they reduce to {\tt sli} classes. \textbf{(b)}~Average counts of common {\tt pi} classes per time step and their visualizations. \textbf{(c)}~The largest observed {\tt pi} class composed of five barellene polyhedra glued together by faces. \textbf{(d)}~A {\tt pi} class composed of a polyhedron that has a hexagonal face that is not uniquely defined because it is a {\tt sli} class.}
	\label{fig:DICE_Res}
\end{figure*}

\subsection{Polyhedra and diamond} 
Last, we investigate the {\tt pi} classes that do not reduce to {\tt sli} classes. As discussed in Section~\ref{sub:Poly}, these {\tt pi} classes can be described as polyhedra clusters. These are relatively rare, there are roughly eight of these {\tt pi} classes per time step. Since more theory is needed to describe these structures, we will consider individual polyhedra that are output by these codes in a case-by-case basis.
We label the {\tt pi} classes by the small hydrocarbons with the same carbon skeleton. Note, the {\tt pi} classes are not molecules but substructures within the loop cluster.

The majority of the {\tt pi} classes are composed of highly organized short rings. 
Roughly 75\% of these are composed of rings of length 6.
See Figure~\ref{fig:DICE_Res}~(b-d) for some examples of these {\tt pi} classes. 
Most of these {\tt pi} classes are small, consisting of a single polyhedron. See Figure~\ref{fig:DICE_Res}~(a). 
Sometimes, the faces of a polyhedron are not uniquely defined. For instance, the observed polyhedron in Figure~\ref{fig:DICE_Res}~(d) has two hexagonal cycles that can be chosen for its right face. However, polyhedra with non-unique faces are more rare.

The most obvious polyhedron to investigate is adamantane. 
Adamantane is the smallest diamondoid, a nanoscale structure with the same network topology as diamond, and it has been observed in petroleum~\cite{fangOrigin2013}.
Even though our simulation was initialized with adamantane molecules, adamantane was not common in our simulation. See Figure~\ref{fig:DICE_Res}~(b). 
 
Instead, the most commonly observed polyhedron is barallene\footnote{
	It is likely that many of these polyhedra more closely match the molecule bicyclo[2,2,2]octane with the same topology as barallene without double bonds. 
	However, we do not store bond order or infer double bonds. Hence, we use the simpler name barallene in both cases.
	}. 
Barallene can be used to form hexagonal diamond -- also called londsdaleite -- by gluing its faces to iceane. Iceane is relatively rare in our simulations, so it seems unlikely that this crystal structure will form. Instead, the majority of the larger {\tt pi} classes are composed of multiple barallene polyhedra glued together by faces. See twistane in Figure~\ref{fig:DICE_Res}~(b) and the larger structure in Figure~\ref{fig:DICE_Res}~(c).

\section{Discussion}
\label{sec:discussion}

\subsection{Sampling cycles (ring perception)} 
In Section~\ref{sub:rand_MCB}, we introduced a method for sampling MCBs uniformly at random, Algorithm~\ref{alg:rand_MCB}. 
This differs from other approaches because we do not sample a deterministic set of rings from our chemical data. 
Regardless, as a probabilistic structure, the random MCB is well-defined, and can be used to sample summary statistics from a chemical network. 
This approach results in fewer cycles, as discussed in Section~\ref{sec:results}. In particular, the size of the random MCB grows linearly with the number of bonds in the system.

A potential critique of using a random MCB is that it may not uncover all useful cycles.
Unfortunately, it is challenging to define such a set of cycles that is always small in size. 
The set of relevant cycles can be exponential in number, and it only includes cycles from MCBs. 
However, all cycles are within the span of a MCB by definition.
Thus, we argue it is better to interpret additional cycles as a consequence of the overlapping structure of a cycle cluster.

In Theorem~\ref{Thm:MCB_Intersect} in Section~\ref{sub:pair_intersect}, we fully characterized how a pair of cycles in a given MCB may intersect as a series of paths that appear with the same order and orientation in both cycles. This gives insight on what data structures are needed to properly store these pairwise intersections of cycles. Furthermore, in Algorithm~\ref{alg:merge_MCB} we introduced a method of modifying an MCB so that all pairs of cycles intersect over a single path. This permits a dual graph representation of the loop cluster, where nodes represent cycles in an MCB and edges represent their intersection paths.
These intersection paths are easy to store as a list of nodes.

Theorem~\ref{Thm:MCB_Intersect} is specialized to pairs of cycles in an MCB. Its proof largely relies on the fact that these cycles preserve distances. Because of this, Theorem~\ref{Thm:MCB_Intersect} can be easily adapted to the full set of relevant cycles. The largest set of cycles this proof technique applies to is the set of shortest path rings from Ref.~\cite{franzblau1991computation}.


%
%

\subsection{Polyhedra.} In Section~\ref{sub:Poly}, we related lack of uniqueness of MCBs to polyhedral structures. These generalized polyhedra were defined as sets of non-flat faces obtained via a cycle and its expansion with respect to an MCB
\begin{equation*}
	\calP=\{C\}\cup\expand{\calM}{C},\quad C\in\RelCyc\setminus\calM.
\end{equation*}
The {\tt pi} classes were interpreted as clusters of these polyhedra.

A potential application of these polyhedra is in crystallography. A crystal can be represented by a periodic tiling of these polyhedra glued together by faces, e.g., diamond is a tetrahedral tiling of adamantane. This representation has a benefit when compared to the periodic unit cell representation of a crystal; the {\tt pi} relation is defined for general graphs. There is no requirement that the chemical network or its substructures are large. Small {\tt pi} classes might be interpreted as nanocrystals. Also, the input graph does not need to be periodic. In an inhomogeneous crystal (e.g., a mixture of cubic and hexagonal diamond) it may be useful to analyze how these polyhedra vary within the material.

Before investigating these applications, a pair of limitations must be addressed. 
Algorithmically, the polyhedra are left-kernel vectors of the diagonal blocks of the matrix $\bm{R}$ from Section~\ref{sub:DICE_TICE_Comp}. These correspond to {\tt sli} classes that sum to shorter cycles. 
This notion of polyhedra is abstract.
In some cases, this may be ok, where the {\tt sli} classes can be interpreted as generalized cycles and visualized as the union of their cycles (e.g., Figure~\ref{fig:DICE_Res}~(d)). 
If not, then a cycle may be sampled from each {\tt sli} class to represent a polyhedron as a set of face cycles. Polyhedra sampled using this latter approach are not unique.

The second limitation is that these polyhedra are the basis of a vector space. As with cycles, we are only interested in small polyhedra. Ultimately, one would want to properly define a minimum polyhedron basis, and to find an algorithm to compute it. This task is likely complex. A greedy approach may be to bias Algorithm~\ref{alg:rand_MCB} to obtain a matrix $\bm{R}$ that outputs small polyhedra.

\section{Conclusion}
\label{sec:conclusion}

We reviewed and advanced theory and algorithms for identifying and grouping cycles in graphs. 
We proposed a new method for sampling MCBs uniformly at random. The {\tt pi} relation helps to extract generalized polyhedra, that may be start-ups for forming crystalline structures. The {\tt sli} classes may comprise large numbers of relevant cycles as those around the perimeter carbon nanotubes. We propose to include only one representative of each {\tt sli} class in an MCB.
We also provide theory for the format in which pairs of cycles in an MCB intersect, and we introduce a postprocessing step that makes these intersections single paths. 
This enables a dual graph representation for a cluster of cycles. We propose using this dual graph approach to model complex material networks. 
We illustrate this theory and algorithms by applying them to a reactive simulation of hydrocarbon pyrolysis MD simulation data initialized with 800 adamantane, $\rm C_{10}H_{16}$, molecules.



\section*{Acknowledgments}
This work was partially supported by ONR MURI Grant No. N000142412547 (M.C.), AFOSR MURI Grant No. FA9550-20-1-0397 (M.C.), and NSF GRFP Grant No. DGE 2236417 (P.R.).

\appendix

\section{Proofs for Section~\ref{sec:Theory}}
\subsection{Mutual relevance and useful lemmas}
\label{app:additional_lems}

To start,
Lemma~\ref{Lem:SmallCycs2MCB} below generalizes Statement 1 of Lemma~\ref{Lem:Relevant_MCB_Test} to sets of cycles. This lemma is useful for many of the results in this appendix.

\begin{lemma}
	Let $\calY$ be a set of cycles bounded in length by $k\in\mathbb{N}$, $|C|\le k$ for $C\in\calY$. Let $\calM$ be a given MCB. There exists a subset $\calY'\subseteq\calM$ also bounded in length by $k$, $|C'|\le k$ for $C'\in\calY'$, such that
	\begin{equation}
		\bigoplus_{C'\in\calY'}C' = \bigoplus_{C\in\calY}C.
		\label{eq:Y_replace_Lem}
	\end{equation}
	\label{Lem:SmallCycs2MCB}
\end{lemma}
\begin{proof}
	Let $\calY$ be a set of cycles bounded in length by $k$ and $\calM$ an MCB. By Definition~\ref{def:ex_op} of the expansion operator, the set of cycles
	\begin{equation}
		\calY' = \expand{\calM}{\bigoplus_{C\in\calY} C}
	\end{equation}
	is the unique subset $\calY'\subseteq\calM$ that satisfies Eq.~\eqref{eq:Y_replace_Lem}. We want to show the cycles $\calY'$ are also bounded in length by $k$. We observe, the cycles $\calY'$ satisfy the inclusion
	\begin{equation}
		\calY' \subseteq \bigcup_{{C\in\calY}} \expand{\calM}{C}.
	\end{equation}
	Thus, the cycles $\calY'$ are bounded in length by
	\begin{equation}
		|C'|\le|C|\le k, \quad C\in\calY,C'\in\expand{\calM}{C}
	\end{equation}
	where the first inequality is Statement 1 of Lemma~\ref{Lem:Relevant_MCB_Test}.
\end{proof}

Next, we introduce mutual relevance, which generalizes the notion of relevance to sets of cycles. We will use mutual relevance in our proofs regarding the {\tt pi} relation.

\begin{definition}
	A set of cycles $\mathcal{X}$ is \emph{mutually relevant} if there exists an MCB $\calM$ such that $\mathcal{X}\subseteq\calM$.
	\label{def:mut_rel}
\end{definition}

\noindent We emphasize linearly independent, relevant cycles are not necessarily mutually relevant. See, e.g., the cycles $C_L,C_L',C_R$ in Figure~\ref{fig:overlapping_diamonds} in Section~\ref{sub:background_partitions}.

Recall, vectors $v_1,\hdots,v_k$ from a vector space are linearly independent if the only linear combination that achieves
\begin{equation*}
	c_1v_1+\hdots +c_kv_k=0
\end{equation*}
is when $c_1=\hdots=c_k=0$. In Lemma~\ref{Lem:mut_rel_2} below, we define mutual relevance in a way that resembles this notion of linear independence.

\begin{lemma}
	A set of cycles $\calX\subseteq\RelCyc$ of equal length are not mutually relevant if and only if there exist cycles $C_1,C_2,\hdots,C_k\in\calX$ such that
	\begin{equation}
		C_1\oplus C_2 \oplus\hdots\oplus C_k = \smCyc
		\label{eq:mut_rel_ind}
	\end{equation}
	where $\smCyc$ is the sum of a set of smaller cycles $\calY$, $|C|<|C_1|$ for $C\in\calY$.
	\label{Lem:mut_rel_2}
\end{lemma}
\begin{proof}
	($\Rightarrow$) Let $\calX\subseteq\RelCyc$ not be mutually relevant. The individual cycles in $\calX$ are relevant. Hence, there exists a mutually relevant subset $\calY\subset\calX$ and a cycle $C\in\calX\setminus\calY$ such that $\calY\cup\{C\}$ is not mutually relevant. By the definition of mutual relevance, let $\calM$ be an MCB where $\calY\subseteq\calM$.
	
	We want to construct a subset of $\calX$ that sums to smaller cycles. First, we expand $C$ with respect to $\calM$,
	\begin{equation}
		C = \enspace\bigoplus_{\mathclap{C'\in\expand{\calM}{C}}} C' = \enspace\bigoplus_{\mathclap{\substack{C'\in\expand{\calM}{C}\\|C'|=|C|}}} C' \enspace\oplus\enspace \bigoplus_{\mathclap{\substack{C''\in\expand{\calM}{C}\\|C''|<|C|}}}C''.
	\end{equation}
	We move the cycles of length $|C|=|C'|$ to the left-hand side, so that the right-hand side contains only smaller cycles:
	\begin{equation}
		C\oplus\enspace\bigoplus_{\mathclap{\substack{C'\in\expand{\calM}{C}\\|C'|=|C|}}}C'\enspace =  \enspace\bigoplus_{\mathclap{\substack{C''\in\expand{\calM}{C}\\|C''|<|C|}}}C''. 
		\label{eq:C_exp_toSmCyc}
	\end{equation}
	We argue the cycles in the left-hand side of Eq.~\eqref{eq:C_exp_toSmCyc} belong to $\calY\cup\{C\}\subseteq\calX$.
	
	Assume towards a contradiction that there exists $C'\in\expand{\calM}{C}$ of length $|C'|=|C|$ such that $C'\notin\calY$. 
	By Statement 2 of Lemma~\ref{Lem:Relevant_MCB_Test}, we can exchange $C$ for $C'$ to obtain a new minimum cycle basis $\calM'=(\calM\setminus\{C'\})\cup\{C\}$.
	Indeed, this contradicts our initial assumptions, as it implies $\calY\cup\{C\}\subseteq\calM'$ is mutually relevant.
	Thus, the cycles in the left-hand side of Eq.~\eqref{eq:C_exp_toSmCyc} form a subset of $\calX$ that sums to smaller cycles as desired.
	
	($\Leftarrow$) Let $C_1,C_2,\hdots,C_k$ satisfy Eq.~\eqref{eq:mut_rel_ind}. We want to show that these cycles are not mutually relevant. We assume $C_2,\hdots,C_k$ are mutually relevant, otherwise our proof is complete. Let $\calM$ be an MCB such that $C_2,\hdots,C_k\in\calM$. By Lemma~\ref{Lem:SmallCycs2MCB} we can choose smaller cycles $\calY'\subseteq\calM$ satisfying Eq.~\eqref{eq:mut_rel_ind}. Re-arranging Eq.~\eqref{eq:mut_rel_ind} and substituting in $\calY'$ gives
	\begin{equation}
		C_1 = C_2\oplus\hdots\oplus C_k\oplus\bigoplus_{C'\in\calY'}C'.
		\label{eq:C_exp_all_mut_rel}
	\end{equation}
	Hence, $C_1$ is in the span of cycles in $\calM$, and $C_1$ does not belong to $\calM$. 
	Here, $\calM$ is an arbitrary MCB containing $C_2,\hdots,C_k$, so there is no MCB containing $C_1,C_2,\hdots,C_k$. In other terms, $C_1,C_2,\hdots,C_k$ are not mutually relevant.
\end{proof}

\subsection{Polyhedron-interchangeability}
\label{app:DICE_Proofs}


Here we prove the results for the {\tt pi} relation from Section~\ref{sub:DICE_Props}. These proofs rely on the lemmas and definitions from Appendix~\ref{app:additional_lems}.

First, we prove an equivalent definition of the {\tt pi} relation. Instead of proving Lemma~\ref{Lem:dir_intchg_2}, we prove Lemma~\ref{Lem:dir_intchg_3} below. Lemma~\ref{Lem:dir_intchg_3} has weaker requirements than Lemma~\ref{Lem:dir_intchg_2}. In particular, the shorter cycles used to exchange $C_1$ for $C_2$ no longer need to belong to an MCB. It is straightforward to prove Lemma~\ref{Lem:dir_intchg_2} from Lemma~\ref{Lem:dir_intchg_3}. Lemma~\ref{Lem:dir_intchg_3} is often easier to verify in proofs.

\begin{lemma}
	Let $C_1,C_2\,{\in}\,\RelCyc$ where $|C_1|\,{=}\,|C_2|$. $C_1{\pInt}\, C_2$ if and only if there exists equal length cycles $\calX$, $|C|\,{=}\,|C_1|$ for $C\,{\in}\,\calX$ where $C_1,C_2\,{\notin}\,\calX$, and shorter cycles $\calY$, $|C|{<}|C_1|$ for $C{\in}\calY$, such that
	\begin{enumerate}
		\item (mutual relevance) there exists an MCB $\calM$ where $\calX\cup\{C_2\}\subseteq\calM$,
		\item (expansion) $C_1$ is given by the sum 
		\begin{equation}
			C_1 = C_2\oplus \bigoplus_{C'\in\calX} C' \oplus \bigoplus_{C''\in\calY} C''.
			\label{eq:dir_exchg_split}
		\end{equation}
	\end{enumerate}
	The set of equal length cycles in the expansion of $C_1$ in $\calM$ is $\calX$.
	\label{Lem:dir_intchg_3}
\end{lemma}

\begin{proof}
	($\Leftarrow$) Let $\calX$ and $\calY$ satisfy the given hypothesis. We want to show $C_1\pInt C_2$. By mutual relevance, let $\calM_2$ be an MCB such that $\calX\cup\{C_2\}\subseteq\calM$. By Lemma~\ref{Lem:SmallCycs2MCB}, there exists shorter cycles $\calY'\subseteq\calM_2$ with the same sum as $\calY$. We replace $\calY'$ for $\calY$ in Eq.~\eqref{eq:dir_exchg_split}
	\begin{equation}
		C_1 = C_2\oplus \bigoplus_{C'\in\calX} C' \oplus \bigoplus_{C''\in\calY'} C'',
		\label{eq:dir_exchg_split2}
	\end{equation}
	The right-hand side of Eq.~\eqref{eq:dir_exchg_split2} is the expansion of $C_1$ in $\calM_2$. By Lemma~\ref{Lem:Relevant_MCB_Test}, Statement 2, we may exchange to a new MCB $\calM_1=(\calM_2\setminus\{C_2\})\cup\{C_1\}$. Thus, $C_1\pInt C_2$.
	
	($\Rightarrow$) Let $C_1{\pInt}\, C_2$ via the MCB $\calM_2$. We want to construct sets $\calX$ and $\calY$ that satisfy the given hypothesis. First, we show that $C_2$ belongs to the expansion of $C_1$ in $\calM_2$, or $C_2\in\expandd{\calM_2}{C_1}$. To do so, consider the MCB $\calM_1=(\calM_2\setminus\{C_2\})\cup\{C_1\}$ given by the {\tt pi} relation. If $C_2\notin\expandd{\calM_2}{C_1}$, then the full expansion $\expandd{\calM_2}{C_1}$ belongs to $\calM_1$. This is not possible, since the set $\{C_1\}\cup\expandd{\calM_2}{C_1}$ is linearly dependent,
	\begin{equation*}
		C_1 = \bigoplus_{C'\in\expandd{\calM_2}{C_1}}C',
	\end{equation*}
	so they cannot all belong to $\calM_1$. 
	Since $C_2$ is in the expansion of $C_1$ in $\calM_2$, we may decompose the expansion as follows
	\begin{equation}
		C_1 = C_2\oplus\enspace
		\bigoplus_{\mathclap{\substack{C'\in\expand{\calM_2}{C_1}\setminus\{C_2\}\\|C'|=|C_1|}}}C'
		\quad\oplus\quad
		\bigoplus_{\mathclap{\substack{C''\in\expand{\calM_2}{C_1}\\|C''|<|C_1|}}}C'.
	\end{equation}
	Thus, the cycles $\calX$ are the cycles of length $|C_1|$ in $\expand{\calM_2}{C_1}$ other than $C_2$, and the cycles $\calY$ are the shorter cycles in $\expand{\calM_2}{C_1}$.
\end{proof}

\begin{figure*}
	\includegraphics{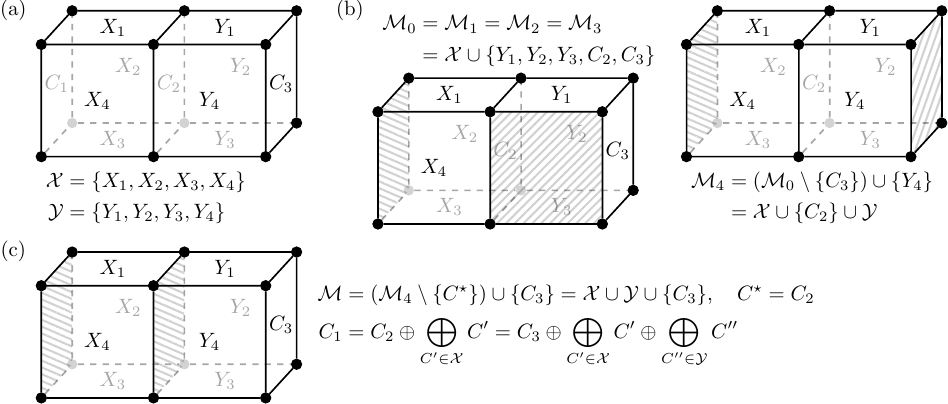}
	\caption{Construction of transitivity for the {\tt pi} relation. \textbf{(a)} A graph with relevant cycles $\RelCyc=\calX\cup\calY\cup\{C_1,C_2,C_3\}$, where $C_1\pInt C_2$ and $C_2\pInt C_3$ via $\calX$ and $\calY$ respectively. \textbf{(b)}~MCBs $\calM_j$ such that the mutually relevant sets $\calZ_j=\calX\cup\{C_2\}\cup\{Y_1,\hdots,Y_j\}$ satisfy $\calZ_j\subseteq\calM_j$. The relevant cycles not in $\calM_j$ are shaded with gray stripes. \textbf{(c)} A final MCB $\calM$ obtained by exchanging $C_3$ for $C_2$ in $\calM_4$. $C_3$ is in the expansion of $C_1$ in $\calM$ so $C_1\pInt C_3$.}
	\label{fig:DICE_trans_ex}
\end{figure*}

Next, we prove that {\tt pi} is an equivalence relation, Theorem~\ref{Thm:DirInt_Eq_Rel}. 
First, consider the following property about the expansion operator. Let $\calB$ be a cycle basis. If we sum $C_1$ and $C_2$ and expand in $\calB$ then we obtain
\begin{equation}
	C_1\oplus C_2 = \bigoplus_{\mathclap{C'\in\expand{\calB}{C_1}}}C' 
	\;\oplus\; 
	\bigoplus_{\mathclap{C''\in\expand{\calB}{C_2}}} C'' 
	\quad=\quad
	\bigoplus_{\mathclap{C'\in\expand{\calB}{C_1}\symdiff\expand{\calB}{C_2}}} C'
\end{equation}
where the reduced sum in the final equality is over the symmetric difference $\expand{\calB}{C_1}\symdiff\expand{\calB}{C_2}$ because of cancellations of the form $C\oplus C=\emptyset$. Thus, the expansion operator is linear
\begin{equation}
	\expand{\calB}{C_1\oplus C_2} = \expand{\calB}{C_1}\symdiff\expand{\calB}{C_2}
	\label{eq:expand_linear}
\end{equation}
where addition is given by the symmetric difference operator. We will use linearity of the expansion operator in the following proofs.



\begin{proof}[Proof of Theorem~\ref{Thm:DirInt_Eq_Rel}]
	We show that the {\tt pi} relation satisfies three properties. The reflexive property, $C\,{\pInt}\, C$, is immediate from an MCB $\calM,C\in\calM$, since $\calM=(\calM\setminus\{C\})\cup\{C\}$ is already an MCB.

	(Symmetry) Here we use the definition of {\tt pi} as in Definition~\ref{def:dir_intchg}. Let $C_1\pInt C_2$ via $\calM_2$, $C_2\in\calM_2$, an MCB such that the exchange $\calM_1=(\calM_2\setminus\{C_2\})\cup\{C_1\}$ is an MCB. Then, the reverse exchange $\calM_2=(\calM_1\setminus\{C_1\})\cup\{C_2\}$ also results in a MCB. Hence, $C_2\pInt C_1$.

	(Transitivity) Transitivity is trickier to verify for the {\tt pi} relation. As an intuition leading up to this result, recall from Section~\ref{sub:Poly} that {\tt pi} relations can be interpreted as polyhedra. For instance, consider the graph depicted in Figure~\ref{fig:DICE_trans_ex} (a). The cycles $C_1$ and $C_2$ satisfy $C_1{\pInt}\, C_2$ since they belong to the same cube $\calP_1=\{C_1,C_2\}\cup\calX$. Similarly, $C_2\,{\pInt}\, C_3$ since they belong to the cube $\calP_2=\{C_2,C_3\}\cup\calY$. The transitive relation can be verified using the rectangular prism $\calP_3=\calP_1\symdiff\calP_2 = \{C_1,C_3\}\cup\calX\cup\calY$. Unfortunately, the symmetric difference $\calP_3$ does not define a {\tt pi} relation in general, as the cycles $\calP_3\setminus\{C_1\}$ may not be mutually relevant. Instead, we reproduce this procedure using mutually relevant sets below. We illustrate these methods in Figure~\ref{fig:DICE_trans_ex} as well.
	
	To verify transitivity, we use the definition of the {\tt pi} relation given by Lemma~\ref{Lem:dir_intchg_3}. 
	Let $C_1{\pInt}\, C_2$ and $C_2\,{\pInt}\, C_3$ via $\calX$ and $\calY$ up to smaller cycles respectively:
	\begin{align}
		C_1 &= C_2 \oplus \bigoplus_{C\in\calX} C \oplus\text{(smaller cycles)}, 
		\label{Eq:C1C2DirEx}
		\\
		C_3 &= C_2 \oplus \bigoplus_{C\in\calY} C\oplus\text{(smaller cycles)},
		\label{Eq:C2C3DirEx}
	\end{align}
	where (i) $|C|=|C_1|$ for $C\in\calX\cup\calY$, (ii) $\calX\cup\{C_2\}$ and $\calY\cup\{C_2\}$ are mutually relevant, and  (iii) $C_1,C_2\notin\calX$ and $C_2,C_3\notin\calY$.
	Note, the set $\calX\cup\{C_1\}$ is mutually relevant, as it belongs to the MCB obtained by the exchange $\calM_1=(\calM_2\setminus\{C_2\})\cup\{C_1\}$, where $\calM_2$ is an MCB such that $\calX\cup\{C_2\}\subseteq\calM_2$.
	$\calY\cup\{C_3\}$ is mutually relevant by a similar argument.

	
	
	Motivated by the polyhedron $\calP_3$ combining $\calP_1$ and $\calP_2$ in our example, we iteratively merge the mutually relevant sets $\calX\cup\{C_2\}$ and $\calY$. 
	First, initialize $\calZ_0=\calX\cup\{C_2\}$ and arbitrarily order the cycles $\calY=\{Y_1,\hdots,Y_k\}$. Then, for $j=1,\hdots,k$ we do the following
	\begin{enumerate}
		\item if there exists $\tilde{Y}_j\in\Span (\calZ_{j-1})$ such that $Y_j=\tilde{Y}_j\oplus\text{(smaller cycles)}$ then $\calZ_{j}=\calZ_{j-1}$,
		\item otherwise, set $\calZ_{j}=\calZ_{j-1}\cup\{Y_{j}\}$.
	\end{enumerate}
	Recall from Lemma~\ref{Lem:mut_rel_2}, a set of cycles is mutually relevant if it does not have a subset that sums to smaller cycles. Thus, at each step the cycles $\calZ_j$ are mutually relevant. Furthermore, all cycles $C\in\calX\cup\calY\cup\{C_2\}$ are within the span of $\calZ_k\subseteq\calX\cup\calY\cup\{C_2\}$ up to smaller cycles.
	
	By definition of mutual relevance, let $\calM_k$ be an MCB containing $\calZ_k$. Consider the expansion of $C_3$ with respect to $\calM_k$. 
	By Eq.~\eqref{Eq:C2C3DirEx}, $C_3$ is in the span of $\calZ_k$ up to smaller cycles.
	By Lemma~\ref{Lem:mut_rel_2}, the mutually relevant cycles $\calY\cup\{C_3\}$ cannot be combined to sum to smaller cycles. Thus, there exists an additional cycle in the expansion $C^\star\in\expand{\calM_k}{C_3}$, $C^\star\in\calZ_k\setminus\calY$, of length $|C^\star|=|C_3|$. By Statement 2 of Lemma~\ref{Lem:Relevant_MCB_Test} we can exchange $C_3$ for $C^\star$ to obtain a new MCB $\calM=(\calM_k\setminus\{C^\star\})\cup\{C_3\}$.
	
	Finally, we want to show that $C_1\pInt C_3$ in $\calM$ by showing $C_3\in\expand{\calM}{C_1}, |C_1|=|C_3|$. By linearity of the expansion operator
	\begin{equation}
		\expand{\calM}{C_1} = \expand{\calM}{C_1\oplus C_3}\symdiff \{C_3\}.
		\label{eq:C1C3_Ex_M}
	\end{equation}
	where $C_1=(C_1\oplus C_3)\oplus C_3$ and $\expand{\calM}{C_3}=\{C_3\}$. Next, consider the identity
	\begin{equation}
		\expand{\calM_k}{C_1\oplus C_3} = \expand{\calM_k}{C_1}\symdiff \expand{\calM_k}{C_3}.
	\label{eq:C1C3_Ex_Mk}
	\end{equation}
	By construction, $C^\star$ belongs to both terms in the right-hand side of Eq.~\eqref{eq:C1C3_Ex_Mk}, where the cycles of length $|C_1|$ in $\expand{\calM_k}{C_1}$ are $\calX\cup\{C_2\}$ by Eq.~\eqref{Eq:C1C2DirEx}. Hence, $C^\star$ cancels in the symmetric difference in Eq.~\eqref{eq:C1C3_Ex_Mk} and
	\begin{equation}
		\expand{\calM_k}{C_1\oplus C_3} \subseteq \calM_k\setminus\{C^\star\} \subset \calM.
	\end{equation}
	Thus, $\expand{\calM_k}{C_1\oplus C_3}$ is the expansion of $C_1\oplus C_3$ in $\calM$ as well. We substitute $\expand{\calM_k}{C_1\oplus C_3}$ for $\expand{\calM}{C_1\oplus C_3}$ in Eq.~\eqref{eq:C1C3_Ex_M}:
	\begin{equation}
		\expand{\calM}{C_1} = \expand{\calM_k}{C_1\oplus C_3}\cup\{C_3\}
	\end{equation}
	where $C_3\notin\calM_k$.
	To conclude $C_3\in\expand{\calM}{C_1}$ so $C_1\pInt C_3$ in $\calM$ as desired.
\end{proof}

Next, we prove the {\tt pi} classes can be computed using a single MCB.

\begin{proof}[Proof of Theorem~\ref{Thm:DirInt_Comp}]
	Let $\calM$ be a given MCB. 
	We have already shown that if $C_1\in\RelCyc$ and $C_2\in\expand{\calM}{C_1}$, $|C_1|=|C_2|$, then $C_1\pInt C_2$. 
	See Lemma~\ref{Lem:Relevant_MCB_Test}, Statement 2. We want to show the {\tt pi} classes are obtained from the transitive closure of these relations.
	
	Let $C_1\pInt C_2$. 
	By Lemma~\ref{Lem:dir_intchg_3}, there exists a set of cycles $\calX$, $C_1,C_2\notin\calX$, such that
	\begin{equation}
		C_1 = C_2 \oplus\bigoplus_{C\in\calX}C\oplus\smCyc
		\label{eq:C1C2_4_DirExAlg}
	\end{equation}
	where $|C|=|C_1|$ for $C\in\calX$ and $\calX\cup\{C_2\}$ is mutually relevant. 
	In the following identity, we expand Eq.~\eqref{eq:C1C2_4_DirExAlg} in $\calM$ and use linearity of the expansion operator
	\begin{equation}
		\begin{split}
		\expand{\calM}{C_1}\symdiff\expand{\calM}{C_2}\symdiff\left(\symdiffseries_{C\in\calX}\expand{\calM}{C}\right) \\= \smCyc
		\end{split}
		\label{eq:DICE_alg_Exs}
	\end{equation}
	In Eq.~\eqref{eq:DICE_alg_Exs}, $\smCyc$ represents a set of smaller cycles $\calY$, $|C'|<|C_1|$ for $C'\in\calY$, which 
	belonging to $\calM$ by Lemma~\ref{Lem:SmallCycs2MCB}.
	We utilize Eq.~\eqref{eq:DICE_alg_Exs} to obtain {\tt pi} relations from $\calM$ which we can stitch together to show that $C_1\pInt C_2$.
	
	Consider the following procedure
	\begin{enumerate}
		\item Initialize $\calX_0=\emptyset$ and index $j=1$.
		\item Let $B_j\in\calM$ be a cycle of length $|B_j|{=}|C_1|$ in the expansion of $\calX_{j-1}\cup\{C_2\}$ in $\calM$,
		\begin{equation*}
			B_j\in
			\expand{\calM}{C_2}\symdiff\bigg(\symdiffseries_{{C\in\calX_{j-1}}}\expand{\calM}{C}\bigg).
		\end{equation*}
		\item Let $X_j\in (\calX\cup\{C_1\})\setminus\calX_{j-1}$ such that $B_j\in\expand{\calM}{X_j}$ and set $\calX_j = \calX_{j-1}\cup\{X_j\}$.
		\item If $X_j=C_1$ then exit, otherwise repeat steps 2-4 with $j\gets j+1$.
	\end{enumerate}
	The feasibility of this procedure is as follows.
	In step 2, the cycles $\calX_{j-1}\cup\{C_2\}\subseteq\calX\cup\{C_2\}$ are mutually relevant, so they do not sum to shorter cycles by Lemma~\ref{Lem:mut_rel_2}. Thus, there exists an equal length cycle $B_j$, $|B_j|=|C_1|$, in the expansion of $\calX_{j-1}\cup\{C_2\}$ with respect to $\calM$.
	For step 3, we observe that $B_j$ is not contained in the right-hand side of Eq.~\eqref{eq:DICE_alg_Exs} because it is not a smaller cycle. 
	Therefore, there exists an additional cycle $X_j$ such that $B_j\in\expand{\calM}{X_j}$ which cancels with the $B_j$ in the expansion of $\calX_{j-1}\cup\{C_2\}$ in the left-hand side of Eq.~\eqref{eq:DICE_alg_Exs}. 
	
	We show by an inductive argument the cycles $\calX_j\cup\{C_2\}$ are in the same {\tt pi} class via the transitive closure of the {\tt pi} relations given by $\calM$.
	Initially, $\calX_0\cup\{C_2\}$ is a single cycle, so there is nothing to prove.
	For the inductive step, let $\calX_{j-1}\cup\{C_2\}$ belong to the same {\tt pi} class by 
	the transitive closure of {\tt pi} relations from $\calM$. By construction, $B_j$ is an equal length cycle in the expansion of both $\calX_{j-1}\cup\{C_2\}$ and $X_j$ in $\calM$. Thus, $B_j$ is {\tt pi}-related to a cycle in $\calX_{j-1}\cup\{C_2\}$ and to $X_j$ via $\calM$. Therefore, the cycles $\calX_j\cup\{C_2\}=\calX_{j-1}\cup\{C_2\}\cup\{X_j\}$ are in the same {\tt pi} class by 
	the transitive closure of {\tt pi} relations from $\calM$. 
	
	The stopping condition of this procedure is that $C_1\in\calX_j\cup\{C_2\}$. Hence, $C_1{\pInt} C_2$ by the transitive closure of the {\tt pi} relations from $\calM$. The cycles $C_1,C_2$ are arbitrary, so $\calM$ can be used to identify all {\tt pi} relations.
\end{proof}

Last, we prove we can always exchange cycles between distinct MCBs.

\begin{proof}[Proof of Lemma~\ref{Lem:DICE_Irreducible}]
	Let $\calM_1,\calM_2$ be distinct MCBs and $C_2\in\calM_2\setminus\calM_1$. We want to construct a cycle $C_1\in\calM_1\setminus\calM_2$ such that $(\calM_1\setminus\{C_1\})\cup\{C_2\}$ is an MCB.
	
	We expand $C_2$ in $\calM_1$ as follows
	\begin{equation}
		C_2 = \bigoplus_{{\substack{C'\in\expand{\calM_1}{C_2}\\|C'|=|C_2|}}}C'
		\oplus
		\bigoplus_{{\substack{C''\in\expand{\calM_1}{C_2}\\|C''|<|C_2|}}}C''.
	\end{equation} 
	Next, we move the equal length cycles to the left-hand side.
	\begin{equation}
		C_2 \oplus \bigoplus_{\substack{C'\in\expand{\calM_1}{C_2}\\|C'|=|C_2|}}C'
		=
		\bigoplus_{\substack{C''\in\expand{\calM_1}{C_2}\\|C''|<|C_2|}}C''.
	\end{equation}
	Importantly, the cycles of length $|C_2|$ in $\{C_2\}\cup\expand{\calM_1}{C_2}$ sum to shorter cycles, so they are not mutually relevant by Lemma~\ref{Lem:mut_rel_2}.
	
	By definition, the cycles of length $|C_2|$ in $\calM_2$ are mutually relevant. A subset of a mutually relevant set is mutually relevant, so the cycles of length $|C_2|$ in $\{C_2\}\cup\expand{\calM_1}{C_2}$ are not contained in $\calM_2$. In particular, there exists $C_1\in\expand{\calM_1}{C_2}$ of length $|C_1|=|C_2|$ such that $C_1\in\calM_1\setminus\calM_2$. By Lemma~\ref{Lem:Relevant_MCB_Test}, Statement 2, we may exchange $C_2$ for $C_1$ in $\calM_1$ as desired.
\end{proof}

\subsection{Short loop-interchangeability}
\label{app:TICE_Proofs}
Here, we prove the theoretical results for the {\tt sli} relation from Section~\ref{sub:TICE_def}.
Recall, two equal length cycles $C_1$ and $C_2$ satisfy $C_1\slInt C_2$ if there exist shorter cycles $\calY$, $|C|{<}|C_1|$ for $C{\in}\calY$, such that
\begin{equation}
	C_1 = C_2 \oplus\bigoplus_{C\in\calY}C.\label{eq:TrEx_def2}
\end{equation}
First, we prove {\tt sli} is an equivalence relation.

\begin{proof}[Proof of Lemma~\ref{Lem:TrInt_Eq_Rel}]
	The reflexive property is immediate by $C=C\oplus\emptyset$.
	
	(symmetry) Let $C_1\slInt C_2$ via shorter cycles $\calY$. By adding cycles $\calY$ to both sides of Eq.~\eqref{eq:TrEx_def2} we find that $C_2\slInt C_1$:
	\begin{equation*}
		C_1 = C_2\oplus\bigoplus_{C\in\calY} C \enspace\Rightarrow\enspace 
		C_2 = C_1\oplus\bigoplus_{C\in\calY}C.
	\end{equation*}
	
	(transitivity) Let $C_1\slInt C_2$ and $C_2\slInt C_3$ by sets of smaller cycles $\calY$ and $\calY'$ respectively,
	\begin{align*}
		C_1 = C_2\oplus\bigoplus_{C\in\calY}C,\qquad
		C_2 = C_3\oplus\bigoplus_{\mathclap{C'\in\calY'}}C'.
	\end{align*}
	Substituting the right-hand side of the second equation for $C_2$ in the first equation gives
	\begin{equation*}
		C_1 = \left(C_3\oplus\bigoplus_{\mathclap{C'\in\calY'}}C'\right)\oplus\bigoplus_{C\in\calY}C=C_3\oplus\bigoplus_{\mathclap{C''\in\calY\symdiff\calY'}}C''.
	\end{equation*}
	Hence, $C_1{\slInt}\,C_3$ via shorter cycles $\calY\symdiff\calY'$.
\end{proof}

Next, we prove $C_1$ and $C_2$ satisfy $C_1\slInt C_2$ if and only if their expansions with respect to a given MCB differ by shorter cycles.

\begin{proof}[Proof of Lemma~\ref{Lem:TrEx_Alg}]
	($\Leftarrow$) Let $\calM$ be an MCB and $C_1,C_2\in\RelCyc$ such that $|C_1|=|C_2|$. Let the expansions of $C_1$ and $C_2$ in $\calM$ differ by shorter cycles, i.e., there exist cycles $\calY\subseteq\calM$, $|C|<|C_1|$ for $C\in\calY$, such that
	\begin{equation}
		\expand{\calM}{C_1}=\expand{\calM}{C_2}\symdiff\calY.
		\label{eq:MCB_TrEx}
	\end{equation}
	By summing over cycles in Eq.~\eqref{eq:MCB_TrEx} we obtain
	\begin{equation*}
		C_1 = C_2\oplus\bigoplus_{C\in\calY} C.
	\end{equation*}
	Thus, $C_1{\slInt}\,C_2$ via shorter cycles $\calY$ as desired.
	
	($\Rightarrow$) Let $C_1\slInt C_2$ by shorter cycles $\calY$ as in Eq.~\eqref{eq:TrEx_def2}.
	Then, by linearity of the expansion operator applied to Eq.~\eqref{eq:TrEx_def2} we obtain
	\begin{equation}
		\begin{split}
			\expand{\calM}{C_1} &= \expand{\calM}{C_2} \symdiff \expand{\calM}{\bigoplus_{C\in\calY} C}\\
			&= \expand{\calM}{C_2} \symdiff \calY'
		\end{split}
	\end{equation}
	where $\calY'\subseteq\calM$ is a set of shorter cycles in $\calM$ by Lemma~\ref{Lem:SmallCycs2MCB}. Thus, the expansions of $C_1$ and $C_2$ in $\calM$ differ by shorter cycles $\calY'$ as desired.	
\end{proof}

\section{Proofs for Section~\ref{sub:pair_intersect}}
\label{app:intersect_proofs}


We prove these results using a distance-based approach. 
The distance $d(u,v)$ between nodes of a graph $G$ is the length of a shortest path connecting $u$ and $v$.
Recall, a path $P$ belonging to a relevant cycle $C$ of length $|P|\le |C|/2$ is a shortest path~\cite{vismaraUnion1997} -- see Lemma~\ref{lem:shortest_paths}. 
A pair of nodes $u$ and $v$ in $C$ are connected by two paths in $C$, one for each orientation of $C$. 
The shorter of these paths, $P_{u,v}$, is bounded in length by $|P_{u,v}|\le |C|/2$. 
Thus, every pair of nodes in $C$ are connected by a shortest path of $G$ in $C$. We restate this as a result:

\begin{corollary}[distance preservation]
	Let $C$ be a relevant cycle of a graph $G$ and let $u,v$ be nodes belonging to $C$. The distance between $u$ and $v$ with respect to $C$ is equal to their distance with respect to $G$.
	\label{Cor:dist_pres}
\end{corollary}

Using Corollary~\ref{Cor:dist_pres}, we employ the notion of resolving sets. A resolving set~\cite{tillquist2023getting} is a set of nodes $R=\{r_1,\hdots,r_k\}\subseteq V$ such that each node $x\in V$ is uniquely identified by its distances to the nodes $R$. For general graphs, an efficient resolving set is hard to construct. However, a resolving set of a cycle requires only two nodes.

\begin{lemma}
	Let $C$ be a cycle graph and $u,v$ nodes in $C$ such that $d(u,v)<|C|/2$. Then, every node $x$ in $C$ is uniquely identified by its distances to $u$ and $v$, i.e., the function
	\begin{equation}
		\varphi(x) = (d(x,u),d(x,v)),\quad x\in V(C)\label{eq:phi_def}
	\end{equation}
	is one-to-one.
	\label{Lem:cyc_embed}
\end{lemma}

\begin{proof}
	Let $u,v$ be nodes belonging to a cycle graph $C$ and $\varphi$ given by Eq.~\eqref{eq:phi_def}. We want to show $\varphi$ is bijective. We also want to describe the shape given by $\varphi(C)$ which will be useful in the proof of Theorem~\ref{Thm:MCB_Intersect} as well. When plotting $\varphi$, the x and y coordinates correspond to the distance of a given node to $u$ and $v$ respectively. This is shown for an even cycle in Figure~\ref{fig:Res_Set_Lem1}.
	\begin{figure}[h]
		\includegraphics{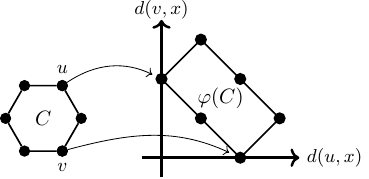}
		\caption{Even cycle for Lemma~\ref{Lem:cyc_embed}.}
		\label{fig:Res_Set_Lem1}
	\end{figure}
	\noindent We observe that $\varphi$ maps this cycle to the integer coordinates of a rectangle.
	
	In contrast, $\varphi$ maps an odd cycle to a rectangle with two of its corners cut off as shown in Figure~\ref{fig:Res_Set_Lem2} below.
	\begin{figure}[h]
		\includegraphics{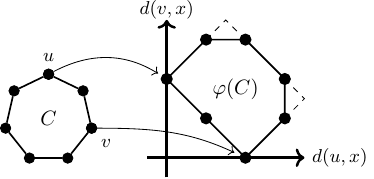}
		\caption{Odd cycle for Lemma~\ref{Lem:cyc_embed}.}
		\label{fig:Res_Set_Lem2}
	\end{figure}
	
	In either case, the function $\varphi$ maps our cycle to the integer coordinates of a polygon. 
	This shape is not self-intersecting, so $\varphi$ is one-to-one.
	The only exception to this is if the nodes $u$ and $v$ are antipodal ($d(u,v)=|C|/2$) as shown in Figure~\ref{fig:Res_Set_Lem3} below.
	\begin{figure}[h]
		\includegraphics{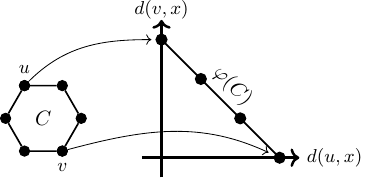}
		\caption{Degenerate case for Lemma~\ref{Lem:cyc_embed}.}
		\label{fig:Res_Set_Lem3}
	\end{figure}
	\noindent In this case, the symmetries of $u$ and $v$ cause the mapping $\varphi(C)$ to collapse from a rectangle into a line. 
	Because this line doubles back on itself the function $\varphi$ is not one-to-one when $d(u,v)=|C|/2$.
\end{proof}

Next, we use the resolving set representation of cycles to describe the intersection format of two cycles in an MCB.



\begin{proof}[Proof of Theorem~\ref{Thm:MCB_Intersect}]
	Let $C_1,C_2,|C_1|\le|C_2|,$ be cycles in an MCB $\calM$. 
	We split this proof into three cases, a main case where the methods in Lemma~\ref{Lem:cyc_embed} apply, and two degenerate cases.
	
	We remark that a path can be a single node, e.g., the path $P_3$ in Figure~\ref{fig:intersect_pf_2}. 
	The length of a path is the number of edges it has. A path with one node has length 0.


	\emph{Main case.} Let $u$ and $v$ be nodes belonging to both $C_1$ and $C_2$ such that $d(u,v)\le |C_1|/2$. 
	Then, we may use the function $\varphi(x)=(d(u,x),d(v,x))$ to embed $C_1$ and $C_2$ as overlapping rectangles $\varphi(C_1)$ and $\varphi(C_2)$ (or rectangles with cut corners as in Figure~\ref{fig:Res_Set_Lem2}). 
	By Lemma~\ref{Lem:cyc_embed}, the mapping $\varphi$ is one-to-one when its domain is restricted to $C_1$ or to $C_2$.
	This shared embedding is shown in Figure~\ref{fig:intersect_pf_1}.
	\begin{figure}[h]
		\includegraphics[]{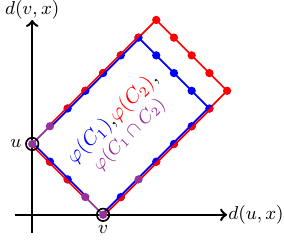}
		\caption{Embedding of $C_1$ and $C_2$ for the main case of Theorem~\ref{Thm:MCB_Intersect}.}
		\label{fig:intersect_pf_1}
	\end{figure}
	In Figure~\ref{fig:intersect_pf_1}
	blue nodes belong to $C_1$, red nodes to $C_2$, and purple nodes to both. 
	Dual blue and red points denote a pair of nodes $x_1\in V(C_1)$, $x_2\in V(C_2)$ that share the same embedding $\varphi(x_1)=\varphi(x_2)$.
	Edges are colored similarly.
	
	The mapping $\varphi$ couples the cycles $C_1$ and $C_2$ by stacking them on top of each other. 
	We orient $C_1$ and $C_2$ by traversing $\varphi(C_1)$ and $\varphi(C_2)$ in the counter-clockwise direction.
	If $C_1$ and $C_2$ are equal in length, then the rectangles $\varphi(C_1)$ and $\varphi(C_2)$ overlap completely. 
	If $C_2$ is longer than $C_1$, then the rectangle $\varphi(C_2)$ extends past $\varphi(C_1)$, as in Figure~\ref{fig:intersect_pf_1}.
	
	
	We use this shared embedding to decompose $C_1$ and $C_2$ into paths, starting with their intersection paths $P_1,\hdots,P_k$.
	Since we traverse both $\varphi(C_1)$ and $\varphi(C_2)$ in the counter-clockwise direction, these intersection paths have the same order and orientation in $C_1$ and $C_2$.
	See Figure~\ref{fig:intersect_pf_2}.
	\begin{figure}[h]
	\includegraphics[]{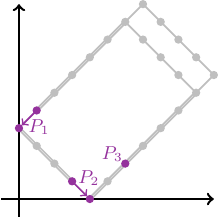}
	\caption{Intersection paths of $C_1$ and $C_2$ and their directions in the main case of Theorem~\ref{Thm:MCB_Intersect}.}
	\label{fig:intersect_pf_2}
	\end{figure}
	In Figure~\ref{fig:intersect_pf_2},
	we mark the orientation of $P_1$ and $P_2$ using purple arrows. The orientation of $P_3$ is arbitrary, because it is a single point.
	We remark that we have not selected a starting point when traversing $\varphi(C_1)$ and $\varphi(C_2)$. Because of this, the order of the intersection paths is circular, e.g., in Figure~\ref{fig:intersect_pf_2} $P_2$ follows $P_1$, $P_3$ follows $P_2$, and $P_1$ follows $P_3$.
	
	Next, we discuss the paths $Q_i^{(j)}$ that separate $C_1$ and $C_2$. All of the path pairs $Q_i^{(1)},Q_i^{(2)}$ fully overlap -- as in Figure~\ref{fig:intersect_pf_3} (a) -- unless $|C_1|<|C_2|$. If $C_1$ is shorter than $C_2$, then the pair that includes the northeast lines of the rectangles $\varphi(C_1),\varphi(C_2)$ do not overlap, as in Figure~\ref{fig:intersect_pf_3} (b).
	We prove three identities regarding the lengths of these pairs.
	\begin{figure}[h]
		\includegraphics[]{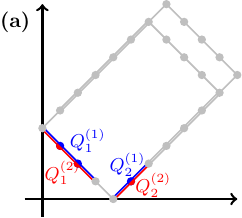}
		\includegraphics[]{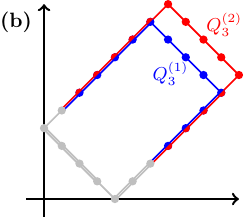}
		\caption{Separated path pairs $Q_i^{(1)},Q_i^{(2)}$ for the main case of Theorem~\ref{Thm:MCB_Intersect}.}
		\label{fig:intersect_pf_3}
	\end{figure}
	
	\begin{enumerate}
		\item Let $|C_1|<|C_2|$, and order the path pair with different lengths last: $|Q_k^{(1)}| \ne |Q_k^{(2)}|$. This is the pair shown in Figure~\ref{fig:intersect_pf_3} (b). We will show this pair satisfies
		\begin{align}
			|Q_k^{(2)}| &= |Q_k^{(1)}|+|C_2|-|C_1|\label{eq:Int_Thm_QId1a}\\
			|Q_k^{(j)}| &\ge |C_j|/2,\quad j=1,2.\label{eq:Int_Thm_QId1b}
		\end{align}
		Eq.~\eqref{eq:Int_Thm_QId1a} is obvious because $Q_k^{(1)}$ and $Q_k^{(2)}$ account for the difference in lengths between $C_1$ and $C_2$:
		\begin{equation*}
			|Q_k^{(2)}| - |Q_k^{(1)}| = |C_2|-|C_1|
		\end{equation*}
		As for Eq.~\eqref{eq:Int_Thm_QId1b}, consider the cycle $C_2'$ obtained by swapping the path $Q_k^{(2)}$ in $C_2$ for $Q_k^{(1)}$,
		\begin{equation*}
			C_2'= C_2\oplus C_k^e,\quad C_k^e = Q_k^{(1)}\cup Q_k^{(2)}.
		\end{equation*}
		$C_2$ cannot be broken into shorter cycles because it is relevant. $C_2'$ is shorter than $C_2$ because $Q_k^{(1)}$ is shorter than $Q_k^{(2)}$. Hence, the cycle $C_k^e$ used to exchange $C_2$ for $C_2'$ satisfies
		\begin{equation}
			|C_k^e| = |Q_k^{(1)}|+|Q_k^{(2)}| \ge |C_2|.
			\label{eq:Int_Thm_Exk_Ineq}
		\end{equation}
		Eq.~\eqref{eq:Int_Thm_QId1b} is obtained by combining Eqs.~\eqref{eq:Int_Thm_QId1a} and~\eqref{eq:Int_Thm_Exk_Ineq}.
		
		\item Let $|C_1|=|C_2|$. We will show there exists a unique large path pair satisfying
		\begin{equation}
			|Q_k^{(1)}|,|Q_k^{(2)}|\ge |C_1|/2.
		\end{equation}
		This pair is ordered last to match the previous item. 
		
		Suppose towards a contradiction that each path pair is shorter,
		\begin{equation*}
			|Q_i^{(1)}|=|Q_i^{(2)}| < |C_1|/2,
		\end{equation*}
		for $i=1,\hdots,k$. Then $C_2$ can be obtained from $C_1$ by swapping each path $Q_i^{(1)}$ for $Q_i^{(2)},$
		\begin{equation*}
			\begin{split}
			C_2 &= C_1\oplus(C_1^e\oplus\hdots\oplus C_k^e)\\
			C_i^e &= Q_i^{(1)}\cup Q_i^{(2)},\quad i=1,\hdots,k.
			\end{split}
		\end{equation*}
		The cycles $C_i^e$ are shorter
		\begin{equation*}
			|C_i^e| = |Q_i^{(1)}|+|Q_i^{(2)}|<|C_1|.
		\end{equation*}
		Thus, $C_1\slInt C_2$. This is a contradiction because at most one cycle in a {\tt sli} class may belong to an MCB by Corollary~\ref{Cor:TrInt_Exchanges}.
		
		\item Last, we show the remaining path pairs are bounded above in length by
		\begin{equation}
			|Q_i^{(1)}|,|Q_i^{(2)}|<|C_1|/2, 
			\label{eq:Int_Pf_Short}
		\end{equation}
		for  $i=1,\hdots,k-1$. 
		Consider a $Q_i^{(j)}$ that lies on a shortest path from $u$ to $v$, such as $Q_1^{(1)}$ or $Q_1^{(2)}$ in Figure~\ref{fig:intersect_pf_3}~(a). Such a path is bounded in length by
		\begin{equation*}
			|Q_i^{(j)}|\le d(u,v) < |C_1|/2.
		\end{equation*}
		Next, consider a $Q_i^{(j)}$ that does not lie on a shortest path from $u$ to $v$, e.g., $Q_2^{(1)}$ and $Q_2^{(2)}$ in Figure~\ref{fig:intersect_pf_3}~(a). Then, $Q_i^{(j)}, i\neq k$ is bounded in length by
		\begin{equation*}
			\begin{split}
			|Q_i^{(j)}| = |Q_i^{(1)}| &\le |C_1| - |Q_k^{(1)}|-d(u,v) \\&< |C_1|/2
			\end{split}
		\end{equation*}
		because (i) $|Q_i^{(1)}|=|Q_i^{(2)}|$ for $i\neq k$, (ii) the path $Q_k^{(1)},|Q_k^{(1)}|>|C_1|/2$ also does not lie on a shortest path from $u$ to $v$, and (iii) $d(u,v)>0$ because $u\neq v$.
		In both cases $Q_i^{(j)}$ satisfies the length bound in Eq.~\eqref{eq:Int_Pf_Short}.
	\end{enumerate}
	
	In summary, these paths satisfy the following:
	\begin{align}
		|Q^{(1)}_i| &\!=\! |Q^{(2)}_i|, &i=1,\!...,k-1,
		\label{eq:sm_path_equal_pf}\\
		|Q^{(j)}_i| &\!<\! |C_1|/2,\! &i{=}1,\!...,\!k{-}1, j{=}1,\!2,
		\label{eq:sm_path_len_pf}\\
		|Q^{(2)}_k| &\!=\! \mathrlap{|Q^{(1)}_k|+|C_2|-|C_1|,}
		\label{eq:lg_path_equal_pf}\\
		|Q^{(j)}_k| &\!\ge\! |C_j|/2,\! &j=1,2.
		\label{eq:lg_path_len_pf}
	\end{align}
	We note that, Eq.~\eqref{eq:sm_path_len_pf} is a strict inequality, which is a stronger result than we intended to show. 
	This observation will be useful in the proof of Theorem~\ref{Thm:MCB_Pair_Swap}. 
	This inequality does not extend to degenerate case 2 below.

	\emph{Degenerate case 1.} Let $C_1$ and $C_2$ share only a single node $u$. See Figure~\ref{fig:intersect_pf_deg1}. 
	\begin{figure}[h]
		\includegraphics{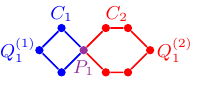}
		\caption{Degenerate case 1 for Theorem~\ref{Thm:MCB_Intersect}.}
		\label{fig:intersect_pf_deg1}
	\end{figure}
	Then they overlap over a single path $P_1=(u)$. They are separated by a pair of paths $Q_1^{(1)}$ and $Q_1^{(2)}$ connecting $u$ to itself. These paths are of length $|Q_1^{(1)}|=|C_1|$ and $|Q_1^{(2)}|=|C_2$ which satisfies the required length requirements of the larger path pairs $Q_k^{(1)}$, $Q_k^{(2)}$.

	\emph{Degenerate case 2.} Let $C_1$ and $C_2$ share two nodes $u$ and $v$ which are antipodal in $C_1$, i.e., $d(u,v)=|C_1|/2$. 
	See Figure~\ref{fig:intersect_pf_deg2}. 
	\begin{figure}[h]
		\includegraphics{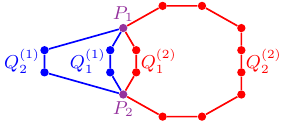}
		\caption{Degenerate case 2 for Theorem~\ref{Thm:MCB_Intersect}.}
		\label{fig:intersect_pf_deg2}
	\end{figure}
	In this case, $C_1$ and $C_2$ overlap over two paths consisting of single nodes $P_1=(u)$ and $P_2=(v)$. $C_1$ is composed of two shortest paths from $u$ to $v$. Label these paths $Q_1^{(1)}$ and $Q_2^{(1)}$ arbitrarily. Also, $C_2$ has a shortest path connecting $u$ and $v$. Label this path $Q_1^{(2)}$ where $|Q_1^{(1)}|=|Q_1^{(2)}|=|C_1|/2$ satisfying Eqs.~\eqref{eq:sm_path_equal} and~\eqref{eq:sm_path_len}. Let $Q_2^{(2)}$ be the other path connecting $u$ and $v$ in $C_2$. This path has length $|Q_2^{(2)}| = |C_2| - |C_1|/2$ which satisfies Eqs.~\eqref{eq:lg_path_equal} and~\eqref{eq:lg_path_len}.
\end{proof}

Next, we prove that if $C_1$ and $C_2$ belong to an MCB and they intersect over multiple paths, then we may exchange $C_1'$ for $C_1$ such that $C_1'$ and $C_2$ intersect over a single path.

\begin{proof}[Proof of Theorem~\ref{Thm:MCB_Pair_Swap}]
	
	Here, we assume the cycles $C_1,C_2$ intersect over multiple paths, otherwise, no modification is needed and we may set $C_1'=C_1$. By Theorem~\ref{Thm:MCB_Intersect}, let $C_1$ and $C_2$ intersect over paths $P_1,\hdots,P_k$ and deviate over path pairs $Q_i^{(j)}$, $j=1,2$, for $i=1,\hdots,k$. We split the proof into two cases. Note, we do not assume $|C_1|<|C_2|$.
	
	\emph{Main case.} 
	Let $C_1$ and $C_2$ belonging to an MCB $\calM$ share two nodes $u$ and $v$ of distance $d(u,v)<|C_1|/2$. As mentioned in the main case in the proof of Theorem~\ref{Thm:MCB_Intersect}, the paths $Q_i^{(1)}$ separating $C_1$ from $C_2$ satisfy the length bound $|Q_i^{(1)}|<|C_1|/2$ for $i=1,\hdots,k-1$. Also by Theorem~\ref{Thm:MCB_Intersect}, the paths $Q_i^{(2)}$ satisfy $|Q_i^{(2)}|=|Q_i^{(1)}|$ for $i=1,\hdots,k-1$.
	Thus, we may exchange the paths in $C_1$ for those in $C_2$ by adding cycles
	\begin{equation}
		C_i^e = {\color{blue}Q_i^{(1)}}\cup {\color{red}Q_i^{(2)}},\quad i=1,\hdots,k-1,
	\end{equation}
	resulting in the cycle $C_1'$ shown in Figure~\ref{fig:MCB_Pair_Swap_Main} below.
	
	\begin{figure}[h]
		\includegraphics{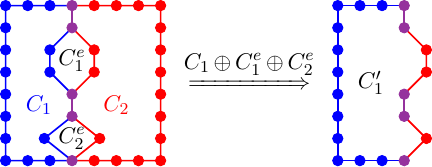}
		\caption{Main case of Theorem~\ref{Thm:MCB_Pair_Swap}.}
		\label{fig:MCB_Pair_Swap_Main}
	\end{figure}
	
	Importantly, the exchange cycles are shorter than $C_1$
	\begin{equation}
		|C^e_i| = |Q_i^{(1)}|+|Q_i^{(2)}|<|C_1|,
	\end{equation}
	Thus, the cycle
	\begin{equation}
		C_1' = C_1\oplus(C_1^e\oplus\hdots\oplus C_{k-1}^e)
	\end{equation}
	satisfies $C_1'\slInt C_1$.
	By Corollary~\ref{Cor:TrInt_Exchanges}, we may exchange $C_1'$ for $C_1$ to obtain a new MCB $\calM'=(\calM\setminus\{C_1\})\cup\{C_1'\}$. Moreover, $C_1'$ and $C_2$ intersect over a single larger path.
	
	\emph{Degenerate case.} 
	Let $C_1$ and $C_2$ belonging to an MCB $\calM$ intersect at two nodes $u$ and $v$ of distance $d(u,v)=|C_1|/2$. In this case, $C_1$ is composed of two shortest paths $Q_1^{(1)}$ and $Q_2^{(1)}$ joining $u$ to $v$. Also, $C_2$ has a shortest path $Q_1^{(2)}$ connecting $u$ to $v$. We consider two candidate cycles $C_1'$ and $C_1''$ obtained by switching one of the paths from $u$ to $v$ in $C_1$ for $Q_2^{(1)}$ in $C_2$. The cycles $C_1'$ and $C_1''$ and their relation to $C_1$ are shown in Figure~\ref{fig:MCB_Pair_Swap_Deg} below.
	
	\begin{figure}[h]
		\includegraphics{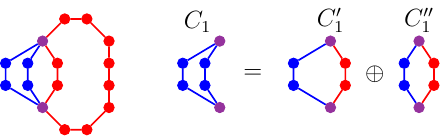}
		\caption{Degenerate case of Theorem~\ref{Thm:MCB_Pair_Swap}.}
		\label{fig:MCB_Pair_Swap_Deg}
	\end{figure}
	
	Algebraically, $C_1'$ and $C_1''$ are given by
	\begin{align}
		C_1' &= C_1\oplus({\color{red}Q_1^{(2)}}\cup{\color{blue}Q_2^{(1)}}) ={\color{blue}Q_1^{(1)}}\cup {\color{red}Q_1^{(2)}},\\
		C_1'' &= C_1\oplus({\color{blue}Q_1^{(1)}}\cup {\color{red}Q_1^{(2)}})={\color{red}Q_1^{(2)}}\cup {\color{blue}Q_2^{(1)}},
	\end{align}
	and $C_1$ is equal to their sum
	\begin{equation}
		C_1 = C_1'\oplus C_1'' = {\color{blue}Q_1^{(1)}}\cup{\color{blue}Q_2^{(1)}}.
	\end{equation}
	
	We remark that, $C_1'$ and $C_1''$ deviate from $C_1$ by equal length cycles, so we can not exchange them for $C_1$ using the {\tt sli} relation as in the main case.
	However, by linearity of the expansion operator, we observe $C_1$ is in the expansion of $C_1'$ or $C_1''$
	\begin{equation*}
		\expandd{\calM}{C_1'}\symdiff\expandd{\calM}{C_1''} = \mathcal{E}_{\calM}(\underbrace{C_1'\oplus C_1''}_{C_1}) = \{C_1\}.
	\end{equation*}
	Without loss of generality, let $C_1\in\expandd{\calM}{C_1'}$. Then, by Lemma~\ref{Lem:Relevant_MCB_Test}, $\calM'=(\calM\setminus\{C_1\})\cup\{C_1'\}$ is an MCB. Furthermore, $C_1'$ and $C_2$ intersect over a single path $Q_2^{(1)}$, completing our proof.
\end{proof}

	\bibliographystyle{ieeetr}
	\bibliography{bibliography.bib}

\end{document}